\documentclass[a4paper,10pt]{amsart}
\usepackage{listings,qtree}
\usepackage{amsfonts, enumitem, imakeidx}%
\usepackage{amssymb,amsmath,amstext,appendix, amsthm, mathtools}
\usepackage[english]{babel}
\usepackage{graphicx}
\usepackage{pgf}
\usepackage{tikz}
\usetikzlibrary{calc, babel, arrows}
\usepackage{pgflibraryarrows}
\usepackage{pgflibrarysnakes}
\usepackage{enumitem}
\usepackage{color}
\usepackage{xcolor}
\usepackage{hyperref}
\hypersetup{backref,
    colorlinks,%
    citecolor=blue,%
    filecolor=black,%
    linkcolor=black,%
    urlcolor=cyan,
    }

\usepackage{mathrsfs}

\newcommand*{\longhookrightarrow}{\ensuremath{\lhook\joinrel\relbar\joinrel\rightarrow}}

\newcommand\restr[2]{{
  \left.\kern-\nulldelimiterspace 
  #1 
  \vphantom{\big|} 
  \right|_{#2} 
  }}

\makeatletter
\def\paragraph{\@startsection{paragraph}{4}%
  \z@\z@{-\fontdimen2\font}%
  {\normalfont\bfseries}}
\makeatother

\newcommand\A{\mathbb A}
\newcommand\C{\mathbb C}

\newcommand\F{\mathbb F}
\newcommand\G{\mathbb G}

\newcommand\N{\mathbb N}
\newcommand\PP{\mathbb P}

\newcommand\R{\mathbb R}
\newcommand\Z{\mathbb Z}

\newcommand\CA{{\mathcal A}}
\newcommand\CB{{\mathcal B}}

\newcommand\CM{{\mathcal M}}

\newcommand\CU{{\mathcal U}}

\newcommand\CZ{{\mathcal Z}}

\newcommand\scrB{{\mathscr B}}
\newcommand\scrC{{\mathscr C}}
\newcommand\scrD{{\mathscr D}}

\newcommand\scrF{{\mathscr F}}
\newcommand\scrG{{\mathscr G}}
\newcommand\scrH{{\mathscr H}}

\newcommand\scrL{{\mathscr L}}
\newcommand\scrM{{\mathscr M}}
\newcommand\scrN{{\mathscr N}}

\newcommand\scrP{{\mathscr P}}

\newcommand\scrS{{\mathscr S}}

\newcommand\scrU{{\mathscr U}}
\newcommand\scrV{{\mathscr V}}
\newcommand\scrW{{\mathscr W}}

\newtheorem{theorem}{Theorem}[section]
\newtheorem{lemma}[theorem]{Lemma}
\newtheorem{proposition}[theorem]{Proposition}

\newtheorem{corollary}[theorem]{Corollary}
\newtheorem{algorithm}[theorem]{Algorithm}
\newtheorem{main}[theorem]{Main Theorem}
\newtheorem{definition}{Definition}
\newtheorem{problem}{Problem}

\newtheorem*{claim}{Claim}

\theoremstyle{definition}

\theoremstyle{remark}
\newtheorem{remark}[theorem]{Remark}

\newtheorem{problema-cl}[theorem]{Problema Cl\'asico}
\newtheorem{open-problem}[theorem]{Problema Abierto}
\newtheorem{problema-dist}[theorem]{Problema Distinguido}

\newtheorem{example}[theorem]{Example}

\newenvironment{proof-claim}[1][Proof of the Claim]{\noindent\textbf{#1.} }{\
\rule{0.5em}{0.5em}\medskip}

\numberwithin{equation}{section}

\DeclareMathOperator\codim{\rm codim}
\DeclareMathOperator\Diag{\rm Diag}

\DeclareMathOperator{\rank}{\rm rank}
\DeclareMathOperator{\Res}{\rm Res}

\DeclareMathOperator\Tr{\rm Tr}

\DeclareMathOperator\Spec{\rm Spec}

\newcommand\lestricto{\subsetneq}

\newcommand\isomorf{\cong}

\normalbaselines

\begin{document}

\title[Promenade through CTS's I]{A promenade through Correct Test Sequences I: \\
Degree of constructible sets, B\'ezout's Inequality and density  }

\author{Luis M. Pardo}
\address{Depto. de Matem\'aticas, Estad\'istica y Computaci\'on. Facultad de Ciencias.
Universidad de Cantabria. Avda. Los Castros s/n. E-39071 Santander, Spain.}
\email{luis.m.pardo@gmail.com, luis.pardo@unican.es}

\author{Daniel Sebasti\'an}
\address{Depto. de Matem\'aticas, Estad\'istica y Computaci\'on. Facultad de Ciencias.
Universidad de Cantabria. Avda. Los Castros s/n. E-39071 Santander, Spain.}
\email{daniel.sebastian@unican.es}
\thanks{Daniel Sebasti\'an San Mart\'{\i}n was partially supported by the Program of Pre-doctoral Grants ``Concepci\'on Arenal'' of the University of Cantabria.}
\thanks{Authors also wish to thank  Antonio Fern\'andez Gonz\'alez and Toraya Fern\'andez Ruiz for their careful reading of preliminary versions of parts of this manuscript.}


\maketitle


\begin{abstract}
In \cite{HeintzSchnorr}, Heintz and Schnorr introduced the notion of \emph{correct test sequence} and since then it has been widely used to design probabilistic algorithms for \emph{Polynomial Equality Test}. The aim of this manuscript is to study the \emph{foundations} and \emph{generalizations} of this notion.  We show that correct test sequences are almost omnipresent and appear in many different forms in the mathematical literature: As \emph{identity sequences} for \emph{Function Identity Test}, as \emph{norming sets} in the field of Banach algebras or as  \emph{samples} in the context of Reproducing Kernel Hilbert Spaces. As main outcome, we generalize the main statement of \cite{HeintzSchnorr} proving that short correct test sequences for constructible sets of lists of polynomials do exist and are densely distributed in any constructible set of accurate dimension and degree. The main tool used to prove this result is the theory of degree of constructible sets, which we introduce and develop in this manuscript, generalizing the results of \cite{Heintz83} and proving two Bezout's Inequalities for two different notions of degree. We present a ${\bf BPP}_K$ algorithm to exhibit the \emph{power} of correct test sequences, this algorithm decides whether a list of polynomials is a secant sequence by just evaluating the input list at some well-suited points. We show the differences between correct test sequences and Demillo-Lipton-Schwartz-Zippel probabilistic tests and we reformulate, prove and generalize two well-known results of the \emph{Polynomial Method}: We prove \emph{Dvir's exponential lower bounds for Kakeya sets} from lower bounds for the length of correct test sequences and  generalize \emph{Alon's Combinatorial Nullstellensatz}.
\end{abstract}

\section{Introduction}
This manuscript is a first step of a long term project to understand the meaning of \emph{correct test sequences (CTS's)}. This notion was introduced in \cite{HeintzSchnorr} and it has been extensively used to design probabilistic algorithms for \emph{Polynomial Equality Tests}, when polynomials are given by arithmetic programs or black-boxes that evaluate them (see, for instance, \cite{Pardo-survey}, \cite{Krick-Pardo96}, \cite{Lickteig}, \cite{Koiran}, \cite{Hardness} and references in all of them).

In this manuscript we do not focus on the applications of the notion in complexity terms. \emph{We focus on the foundations and generalizations of the notion as is:} Seeking for foundational results (as B\'ezout's Inequality for constructible sets), trying to see how they underlie other notions in the existing mathematical literature  and introducing mathematical generalizations and (as in the original \cite{HeintzSchnorr} manuscript) proving not only the existence but also probabilistic results. We finally exhibit a strange algorithm in ${\bf BPP}_K$ to decide whether a list of polynomial equations is a secant sequence. The main feature of this algorithm is not its complexity but the fact that it decides the generic co-dimension simply by evaluating the input list of polynomials at some well-suited points.

Correct test sequences appear in the mathematical literature in many different forms. You may see equivalent notions in different fields: From the notion of finite length \emph{norming set} in a Banach algebra to the notion of \emph{sample} in Reproducing Kernel Hilbert Spaces. Also, many correct test sequences satisfy the main ingredient of the ``Polynomial Method'' and, in particular, they are implicitly used in the lower bound for \emph{Kakeya} sets in \cite{Dvir}. We discuss some of these characterizations in our manuscript.

Correct test sequences have been seen a as concurrent technique to \emph{DeMillo-Lipton-Schwartz-Zippel probabilistic tests} (cf.  \cite{DeMilloLipton}, \cite{Zippel} and \cite{Schwartz}). But both approaches have subtle differences. Some of these features are the following ones:
\begin{itemize}
\item Correct test sequences are by definition ``correct'' for the whole class of inputs under consideration: The error probability of their answers is always zero, independent of the input instance.
\item DeMillo-Lipton-Schwartz-Zippel probabilistic tests are based on an easy-to-construct sampling set, which is itself a correct test sequence. In such large class, each input instance has some accurate sample where it does not vanish, but different inputs require different accurate sample point. Bounding the error probability of a random guess in the sampling set, we get the probabilistic algorithm.
\item Correct test sequences have a length of the same asymptotic order than the Krull dimension of the input set under consideration, whereas DeMillo-Lipton-Schwartz-Zippel probabilistic tests require a sample set of size exponential in the number of variables to get a zero error probability.
\item There is no known efficient deterministic procedure to generate correct test sequences and, hence, their existence (and generation) is based on a statement that proves that they are highly dense in probability terms.
\item Under some simple hypothesis, we even see that correct test sequences are highly dense inside  the (exponentially big)  DeMillo-Lipton-Schwartz-Zippel sampling set (see Corollary \ref{Zippel-Schwartz-repleto-CTS:corol}).
\end{itemize}

These subtle differences require of further explanations and this is the motivation of the research trend we initiate with this manuscript, whose main outcomes will serve as foundations for future research.  The manuscript is twofold and based on two main frameworks:

\begin{itemize}
\item To our knowledge,  there is no \emph{B\'ezout's Inequality for constructible sets} in the mathematical literature. Constructible sets are natural geometric objects that arise naturally as projections of algebraic varieties (see \cite{Chevalley}, for instance) and, hence, they are natural objects of Elimination Theory or Computational Algebraic Geometry. Nevertheless, the notion of the degree of these geometric objects has not been correctly stated in the mathematical literature.  The study was initiated in \cite{Heintz83}, who introduced a first notion of degree of a constructible set: This author defined the degree of a constructible set in some affine space $C\subseteq \A^n(K)$ as the degree of its Zariski closure (which we denote by $\deg_z(C)$ below). Unfortunately, Heintz's  notion of degree does not satisfy B\'ezout's Inequality  (see Example \ref{croix-de-berny-def:ej}, for instance). In \cite{Heintz85} the author already observed that his results in \cite{Heintz83} only hold for locally closed subsets of some affine space. We devote Part 1 (Sections \ref{Bezout-inequality:sec}, \ref{upper-intersection-several:sec} and Appendix \ref{degree-projection-constructible:sec}) of the present manuscript to this topic. We introduce two different notions of degree of a constructible set that satisfy the paradigms of affine degree stated in \cite{Heintz83}: $\deg_\pi(C)$ and $\deg_{\rm lci}(C)$. We prove that both of them satisfy a B\'ezout's Inequality and we study several other relevant properties of both notions.
\item We devote the rest of the manuscript to \emph{correct test sequences}. We first observe different equivalent forms of the notion of correct test sequence in different contexts (in Section \ref{CTS:sec}). We emphasize the fact that correct test sequences are not naturally defined to distinguish between zero and non-zero functions, but to detect by sample evaluation whether an input $f\in \Omega$ belongs or not to some ``discriminant'' subset $\Sigma\subseteq \Omega$ of co-dimension at least one. We observe that correct test sequences also satisfy a \emph{curse of dimensionality}: The Krull dimension of the restricted input set $\Omega$ is a lower bound for the length of any correct test sequence for $\Omega$ with respect to $\{0\}$. Once some good degree notion for constructible sets has been established, we have the tools to provide statements that claim that correct test sequences do exist in any constructible set of accurate dimension and degree (and not only in sets of ``grid'' shape as established in the existing literature). This is done in Section \ref{probability-CTS-zero-dimensional:sec}. In this Section we not only establish an existential  outcome, but also a probabilistic statement that shows that correct test sequences of asymptotical optimal length are densely distributed in any locally closed set of accurate dimension and degree: We exhibit probability estimates for success. Next, we emphasize that detecting dimension gaps is one of the underlying ingredients in applications of correct test sequences. In Section \ref{suite-secante:sec},  we present a ${\bf BPP}_K$ algorithm to detect secant sequences which do not manipulate the input list of polynomials in any form. It simply evaluates lists of polynomials at some well-suited sample points. In a final Appendix \ref{Kakeyaetal-CTS:sec} we compare correct test sequences to two celebrated combinatorial results, basic for  \emph{The Polynomial Method} in Combinatorics: \emph{Dvir's exponential lower bounds for Kakeya sets} and \emph{Alon's Combinatorial Nullstellensatz}.

\end{itemize}

In forthcoming subsections we exhibit in more detail our main outcomes.

\subsection{Main outcomes of Part 1: Degree, B\'ezout's Inequality and variations for constructible sets (Sections \ref{Bezout-inequality:sec}, \ref{upper-intersection-several:sec} and Appendix \ref{degree-projection-constructible:sec})}

We first need to fix some notations that are going to be used in the sequel. Let $\kappa$ be a field and let $K$ be its algebraic closure. For every positive integer $m\in \N$, we denote by $\A^m(\kappa)$ the $m-$dimensional affine space over $\kappa$. A \emph{constructible set} is a subset $C\subseteq \A^m(K)$ which can be obtained as the projection of an algebraic variety of some affine space $\A^M(K)$ of higher dimension $M\geq m$ (cf. \cite{Chevalley}, for instance). A technical difficulty is the lack of a consistent theory of degree of constructible sets that satisfy (at least) a B\'ezout's Inequality. As soon noticed in \cite{Heintz85}, B\'ezout's Inequality in \cite{Heintz83} is only valid for locally closed sets. In Example \ref{croix-de-berny-def:ej} we exhibit a constructible set $C$ and  an algebraic variety $V$ such that $\deg_z(C\cap V)$ is not bounded by the product $\deg_z(C)\deg_z(V)$. Hence, $\deg_z$ does not satisfy either B\'ezout's Inequality or some other statements in \cite{Heintz83}, which we revise here.

According to the \emph{paradigm} stated in   \cite{Heintz83}, a notion of degree of constructible sets must be a quantity (acting as a volume) which satisfies the following properties:
\begin{itemize}
\item It is always finite and agrees with the usual notion in the case of algebraic varieties and locally closed sets.
\item It is sub-additive.
\item It has a good behaviour with respect to intersections with linear varieties and with respect to Cartesian products.
\item It satisfies a B\'ezout's Inequality.
\item It is controlled under projections and images by linear mappings.
\item It satisfies some variation of Proposition 2.3 in \cite{HeintzSchnorr}.
\end{itemize}

From these properties, only Proposition 2.3 of \cite{HeintzSchnorr} requires of some additional explanations.

Observe that a literal application of B\'ezout's Inequality puts the co-dimension at the exponent of bounds for the degree of several geometric objects. Namely, given $f_1,\ldots, f_m\in K[X_1,\ldots, X_n]$ such that $V:=V_\A(f_1,\ldots, f_m)$ has dimension $n-m$, and let $W\subseteq \A^n$ be an algebraic variety. B\'ezout's Inequality yields the following upper bound:
$$\deg(W\cap V_\A(f_1,\ldots, f_m))\leq \deg(W)\prod_{i=1}^m \deg(f_i) \leq \deg(W)\left( \max\{\deg(f_1), \ldots, \deg(f_m)\}\right)^{\rm codim(V)}.$$
Nevertheless, Proposition 2.3 in  \cite{HeintzSchnorr} replaces co-dimension of $V$ by dimension of $W$, emphasizing the role of dimension of $W$ in this kind of upper bounds. Namely, Proposition 2.3 in \cite{HeintzSchnorr} yields:
$$\deg(W\cap V_\A(f_1,\ldots, f_m))\leq  \deg(W)\left( \max\{\deg(f_1), \ldots, \deg(f_m)\}\right)^{\rm dim(W)}.$$
The trade-off between these two upper bounds is the key ingredient to prove the existence and high probability of optimal length correct test sequences.

Having in mind these requirements, we introduce two notions of degree for constructible sets. One is the notion of degree of a constructible set $C\subseteq\A^n$ based on the fact that they may be presented as projections $C=\pi(W)$ of locally closed sets $W\subseteq \A^m$, where $m\geq n$. We denote this notion of degree by $\deg_\pi(C)$ (see Definition \ref{grado-como-proyeccion:def} for details). Another one is based on the fact that constructible sets $C\subseteq \A^n$ may be presented as finite unions of locally closed irreducible sets (i.e. open sets in irreducible varieties for the Zariski topology). We denote by $\deg_{\rm lci}(C)$ this notion of degree (see Definition \ref{grado-constructibles:def} for details). We see how different properties apply accurately to either $\deg_\pi$ or $\deg_{\rm lci}$ and we have not taken any decision on which is best choice. Let us resume some of the properties we prove in this manuscript.

In Example \ref{La-Croix-de-Berny-grados-varios:ej}, we  observe that they are different notions by exhibiting a constructible set $C\subseteq \A^2 (\C)$ such that:
$$\deg_z(C)< \deg_\pi(C) < \deg_{\rm lci}(C).$$
We discarded $\deg_z$ as notion of degree since it does not satisfy B\'ezout's Inequality. We concentrate our study in $\deg_\pi$ and $\deg_{\rm lci}$. We then prove that both notions of degree satisfy the expected properties: They are sub-additive (cf. Proposition \ref{grado-constructibles-subaditivo:prop}), they generalize the notion of degree for  locally closed sets of \cite{Heintz83} (cf. Proposition \ref{generaliza-grado-localmente-cerrados:prop}) and they behave well with intersections of linear affine varieties and Cartesian products (see Proposition \ref{varias-propiedades-basicas-grado-construtibles:prop}). Finally, we prove  that both notions satisfy a \emph{B\'ezout's Inequality for constructible sets} (see Theorem \ref{Desigualdad-Bezout-constructibles:teor}). This statement simply says that  given two constructible subsets $C, D\subseteq \A^n (K)$, the following inequalities hold:
$$\deg_\pi(C\cap D) \leq \deg_\pi(C)\deg_\pi(D) , \;\; \deg_{\rm lci}(C\cap D) \leq \deg_{\rm lci}(C)\deg_{\rm lci}(D).$$

Nevertheless, we have not been able to clarify the complete role of both notions with respect to the remaining two properties: \emph{Proposition 2.3 of \cite{HeintzSchnorr}} and \emph{degree of images under linear mappings}. It is immediate to extend the first of those results for the intersection of a constructible set with several locally closed sets for both  $\deg_{\rm lci}$ (see Proposition \ref{grado-constructible-cerrados-multiple:prop}) and $\deg_\pi$ (under the globally equi-dimensional hypothesis, see Proposition \ref{grado-interseccion-equidimensional-lc:corol}).

In Section \ref{upper-intersection-several:sec} we prove that $\deg_{\rm lci}$ satisfies a similar estimate to Proposition 2.3 of \cite{HeintzSchnorr}  for constructible sets of any kind (cf. Theorem \ref{cotas-grado-varias-intersecciones:teor}) exchanging co-dimension by dimension in the exponent of the upper bound. We reproduce here that statement for helping readability:

\begin{main}[{\bf Bounds for the $LCI$-degree of the intersection of several constructible sets}]\label{cotas-grado-varias-intersecciones:main}
Let $C_1, \ldots, C_s\subseteq \A^n$ be a sequence of constructible sets. Let $r:=\dim(C_1)$ be the dimension of the constructible set $C_1$. The following inequalities hold:

\begin{enumerate}
\item  \emph{First upper bound:}
$$\deg_\pi\left(\bigcap_{i=1}^s C_i\right)\leq\deg_{\rm lci}\left( \bigcap_{i=1}^s C_i\right)\leq {{s+r-1}\choose {r}}\deg_{\rm lci}(C_1) \left( \max\{\deg_{\rm lci}(C_j)\; :\; 2\leq j \leq s\}\right)^{r}.$$
\item \emph{Second upper bound:}
$$\deg_\pi\left(\bigcap_{i=1}^s C_i\right)\leq\deg_{\rm lci}\left(\bigcap_{i=1}^s C_i\right)\leq \deg_{\rm lci}(C_1) \left(1+ \sum_{i=1}^s \deg_{\rm lci}(C_i) \right)^{\dim(C_1)}.$$
\item \emph{Upper bound in terms of the average degree:} Given a family of constructible sets $C_1,\ldots,C_s\subseteq \A^n$, we define its average $E-$degree as:
    $$\deg_{\rm av}^{(E)}(C_1,\ldots, C_s):= {{1}\over{s}}\sum_{i=1}^s \deg_{\rm lci}(C_i).$$
    Then, we also have:
$$\deg_\pi\left(\bigcap_{i=1}^s C_i\right)\leq \deg_{\rm lci}\left(\bigcap_{i=1}^s C_i\right)\leq \deg_{\rm lci}(C_1) s^{\dim(C_1)}\left(\deg_{\rm av}^{(E)}(C_1,\ldots, C_s)\right)^{\dim(C_1)}.$$
\end{enumerate}

\end{main}

We have not been able to establish a similar result with $\deg_\pi$ on the right hand side of any of the inequalities.

On the other hand, we also consider the transformation of the degree by taking images by linear mappings (Lemma 2 of \cite{Heintz83}). We consider a constructible subset $C\subseteq \A^n(K)$ and a linear mapping $\ell:\A^n(K) \longrightarrow \A^m(K)$. In Example \ref{La-Croix-de-Berny-grados-varios:ej} we prove that there are algebraic varieties $W\subseteq \A^3(\C)$ such that $\deg_{\rm lci}(\ell(W))> \deg_{\rm lci}(W)=\deg(W)$. Hence, Lemma 2 of \cite{Heintz83} is false for constructible sets and $\deg_{\rm lci}$. It may increase under linear images.  In Proposition \ref{constructibles-grado-imagenes-por-lineal:prop}  we prove that $\deg_z(\ell(C))\leq \deg_{\rm lci}(C)$, whereas in Claim $iii)$ we exhibit an example of a constructible set that proves that $\deg_{\rm lci}(\ell(C))$ is not bounded by $\deg_{\rm lci}(C)$.

 As expected, we prove in Proposition \ref{pi-degree-propiedades:prop} that $\deg_\pi$ has a good behaviour with respect to linear transformations. Namely, we prove that the following holds:
 $$\deg_z(\ell(C))\leq \deg_\pi(\ell(C))\leq \deg_\pi(C)\leq \deg_{\rm lci}(C),$$
 for any constructible set and any linear transformation $\ell$. This yields to the notion of \emph{defect} of the image $\ell(C)$ of some constructible set given as the following difference:
 $$\deg_{\rm lci}(\ell(C))-\deg_z(\ell(C)).$$
This manuscript would be somehow uncomplete if we do not try to show some upper bounds for this defect. This is done in our Appendix \ref{degree-projection-constructible:sec}. There, in Subsection \ref{deffective-projection:subsec} we state some upper bound for the defect in Lemma \ref{cota-defecto-lema:lema}. In Remark \ref{optimalidad-defecto-ejemplo:rk}, we also exhibit examples such that the bound of Lemma \ref{cota-defecto-lema:lema} is optimal. Nevertheless, we hope to have sharper upper bounds for the defect in forthcoming works.

Additionally, we claimed that one of the paradigmatic requirements of the degree must be some kind of control of the $LCI$-degree of $\ell(C)$. In Subsection \ref{extrinsic-bound-projection:subsec} we also exhibit an extrinsic (not sharp, in our opinion) upper bound for $\deg_{\rm lci}(\ell(C))$. Just for giving some idea of what we mean, we reproduce here the following statement, which is Theorem \ref{extrinisic-bound-Nullstellensatz:teor} of Subsection \ref{extrinsic-bound-projection:subsec}, that stated an upper bound for $\deg_{\rm lci}(\ell(C))$. We prove that $\deg_{\rm lci}(\ell(C))$ is bounded (and, hence, ``controlled'') in terms of some syntactical quantities depending only on $C$:

\begin{main}\label{extrinisic-bound-Nullstellensatz:main}
Assume $\kappa=  K$ is an algebraically closed field. Let $V\subseteq \A^n(K)$ be an irreducible algebraic variety of dimension $r$.  Let $m$ be an integer such that $m\leq n$. Assume that there are polynomials $g_1, \ldots, g_s\in K[X_1,\ldots, X_n]$ of degrees $d_i:=\deg(g_i)$, $1\leq i \leq r$, such that $V=V_\A(g_1,\ldots, g_s)$ and the following inequalities hold:
$$d_1\geq d_2 \geq \ldots \geq d_s.$$
Let us consider the following quantity:
$$N:=N(d_1, \ldots, d_s, n, r):=\left\{\begin{matrix} \prod_{i=1}^s d_i & {\hbox {\rm if $s \leq n-m$}}\\
2d_s\left(\prod_{i=1}^{n-m-1}d_i\right)-1 & {\hbox {\rm if $s > n-m$}} \end{matrix}\right.$$

Let us define the following quantities:
$$\widetilde{N}:={{N+(n-m)}\choose{n-m}},$$
$$M:=\sum_{i=1}^s {{N-d_i+(n-m)}\choose{(n-m)}},$$
and
$$\scrN':=\min\{N, M+1\}.$$

Let $\pi: \A^n(K) \longrightarrow \A^m (K)$ be the canonical projection that forgets the last $n-m$ variables and let $W:=\overline{\pi(V)}^z$ be the Zariski closure of $\pi(V)$ in $\A^m(K)$. Then, we have:

$$\deg_{\rm lci}(\pi(V)) \leq \deg(V)\left(2d_1\right)^{\dim (W)} \left(\scrN'\right)^{\dim(W)+1}.$$
\end{main}

\subsection{Main outcomes of Part 2: Correct Test Sequences, generalizations, relations to other notions, density estimates and immediate applications (Sections \ref{CTS:sec}, \ref{probability-CTS-zero-dimensional:sec}, \ref{suite-secante:sec} and Appendix  \ref{Kakeyaetal-CTS:sec}) }

In the remaining sections of the manuscript we deal with correct test sequences (CTS's). We begin in Section \ref{CTS:sec}  by establishing the notion of correct test sequence in Definition \ref{CTS-funciones-definicion:def}. We consider a set $X$ and a class of functions $\scrF(X)\subseteq K^X$ with values in a field $K$. Given $\Omega\subseteq \scrF(X)^m$ a distinguished class of functions in $\scrF(X)^m$ (i.e. maps from $X$ to $K^m$) and a subset $\Sigma \lestricto \Omega$, a \emph{correct test sequence of length $L$ for $\Omega$ with discriminant $\Sigma$} is a finite set of $L$ elements ${\bf Q}:=\{x_1,\ldots, x_L\} \subseteq X$ such that the following formula holds:
\begin{equation}
\forall f\in \Omega, \;\; f(x_1)=\cdots = f(x_L)=0\in K^m \Longrightarrow f\in \Sigma.
\end{equation}
In the case $\Sigma=\{0\}\lestricto \Omega$, we say that ${\bf Q}$ is a correct test sequence for $\Omega$.
The section is structured as an expository section where we show several equivalent notions of correct test sequences, depending on the context: As identity sequences (when viewed in $\Omega-\Omega$ and $\Sigma=\{0\}$),  as finite \emph{norming sets} in the case $K$ is a field with some absolute value (see Proposition \ref{cuestores-y-metrica:prop}), in the ring of continuous functions $\scrC(X)$ on a topological space (see Proposition \ref{CTS-funciones-continuas:prop}) or in Reproducing Kernel Hilbert Spaces (see Proposition \ref{CTS-RKHS:prop}).

In Subsection \ref{CTS-polynomials:subsec} we focus on multivariate polynomials as the main subject of these pages. For an algebraically closed field $K$ and a positive integer $d\in\N$, we consider the vector space $P_d^K:=P_d^K(X_1,\ldots, X_n)$ of all polynomials in $K[X_1,\ldots, X_n]$ of degree at most $d$. For a list of degrees $(d):=(d_1,\ldots, d_m)\in \N^m$, we consider the class $\scrP_{(d)}^K$ of all lists $f:=(f_1,\ldots, f_m)$ of $m$ polynomials such that $f_i\in P_{d_i}^K$ for every $i$, $1\leq i \leq m$. Thus, we discuss correct test sequences ${\bf Q}$ associated to constructible subsets $\Omega\subseteq \scrP_{(d)}$ with respect to some discriminant subset $\Sigma\lestricto \Omega$. Our main remark in this subsection is a kind of \emph{curse of dimensionality of correct test sequences} (cf. Proposition \ref{codimension-longitud-cuestores:prop}): If ${\bf Q}$ is a correct test sequence of length $L$ for a constructible set $\Omega\subseteq \scrP_{(d)}^K$ with respect to $\Sigma$, then:
$$L\geq \dim(\Omega)-\dim(\Sigma)= \codim_\Omega(\Sigma).$$
This lower bound establishes the limits to have correct test sequences: An \emph{optimal length correct test sequence for $\Omega$} will be a correct test sequence of length in $O(\dim(\Omega))$.

Then, we address the question of whether correct test sequences of optimal length exist in Section \ref{probability-CTS-zero-dimensional:sec}. Our main outcome first proves that correct test sequences of optimal length exist not only for parametrized constructible sets (as in \cite{HeintzSchnorr}) but for any constructible set. Secondly, we prove that length $L$ correct test sequences do exist not only in \emph{grid-like sets} (of the form $(Q_1\times \cdots \times Q_n)^L\subseteq (\A^n)^L$, where $Q_i\subseteq K$ is a finite set): Correct test sequences may be found in any constructible set of accurate dimension and degree. Last but not least, we prove that correct test sequences of optimal length are highly dense in probability terms in any set of the form $C^L$, where $C$ is locally closed and $L$ is in $O(\dim(\Omega))$.

All these properties are stated in Theorem \ref{teorema-principal-densidad-questores:teor}. Given $C\subseteq \A^n$ a locally closed set, $L\in \N$ a positive integer and $\Omega$ and $\Sigma$ as above, we denote by $R(\Omega, \Sigma, C, L)$ the constructible set of all sequences ${\bf Q}\in C^L$ of length $L$ which are correct test sequences for $\Omega$ with respect to $\Sigma$. Theorem \ref{teorema-principal-densidad-questores:teor} proves that, under some technical hypothesis on the degree and the dimension of $C$, for $L\geq 6\dim(\Omega)$, the set $R(\Omega, \Sigma, C, L)$ is highly dense in $C^L$.  We discuss the particular case when $C$ is a complete intersection variety in several Corollaries. We reproduce here Corollary \ref{densidad-questores-variedades-ic:corol} which exhibits the high density of correct test sequences in the case $C$ is a complete intersection variety of accurate dimension and degree:

\begin{main}\label{densidad-questores-variedades-ic:corol-main} Let $m,n \in\N$ be two positive integers, with $m\leq n$, and let $(d):=(d_1,\ldots, d_m)$ be a list of degrees and $d:=\max\{d_1,\ldots, d_m\}$.  Let $\Sigma\subseteq\Omega$ be two constructible subsets of $\scrP_{(d)}(X_1,\ldots, X_n)$ such that $\Sigma$ has co-dimension at least 1 in $\Omega$.  Assume that $\Omega\setminus \Sigma$ satisfies the following property:
  \begin{equation}\label{dimension-hipotesis-cts-corol:eqn}
  \forall f:=(f_1,\ldots, f_m)\in \Omega \setminus \Sigma, \; \dim(V_\A(f_1,\ldots, f_m))= n-m,
\end{equation}
Let $C:=V_\A(h_1,\ldots, h_r)\subseteq \A^n$ be a complete intersection algebraic variety of co-dimension $r\geq (n-m)+ m/2 + 1/2$ such that $\deg(C)\geq \delta^r$, where $\delta:=\min\{\deg(h_1),\ldots, \deg(h_r)\}$.
Let $L\in \N$ be a positive integer and assume that the following properties hold:
\begin{enumerate}
\item $L\geq 6\dim(\Omega)$,
\item $\log(\delta) \geq \max\{2(1+\log(d+1)), {{2\log(\deg_{\rm lci}(\Omega))}\over{\dim(\Omega)}}\}$,
\item $\max\{ \deg(h_1), \ldots, \deg(h_r)\}\leq(1+ {{1}\over{n-m}})\delta$,
\end{enumerate}
where $\log$ stands for the natural logarithm. Let $R:=R(\Omega, \Sigma, C, L)$ be the constructible set of all sequences ${\bf Q}\in C^L$ of length $L$ which are correct test sequences for $\Omega$ with respect to $\Sigma$. Then, there is a non-empty Zariski open subset $\G(C)$, in the space $\G(n, n-r)$ of all linear affine varieties of co-dimension $n-r$,  such that for every $A\in \G(C)$ the probability that a randomly chosen list ${\bf Q}\in (C\cap A)^L$ is in $R$ satisfies:
$${\rm Prob}_{(C\cap A)^L}[R]\geq 1- {{1}\over{\deg_{\rm lci}(\Omega)e^{\dim(\Omega) + (m-1)L}}},
$$
where $(A\cap C)^L$ is endowed with its uniform probability distribution.
\end{main}

Note that if we introduce in $C$ a probability distribution based on a Poincar\'e-like formula for algebraically closed fields (see Section 2 of \cite{BP07} for the complex case) involving $\sharp(C\cap A)$ with $A\in \G(n,n-r)$, the previous statement simply claims that the measure of $R(\Omega, \Sigma, C, L)$ in $C^L$ is close to $1$.

 Trying to answering some of the features that differ between DeMillo-Lipton-Schwartz-Zippel test and correct test sequences, we also prove that correct test sequences of optimal length are highly dense inside the sampling set of the test designed by these authors in \cite{DeMilloLipton}, \cite{Zippel} and \cite{Schwartz}. This is done in Corollary \ref{Zippel-Schwartz-repleto-CTS:corol}.

In Section \ref{suite-secante:sec} we exhibit an immediate application of our approach to correct test sequences: A ${\bf BBP}_K$ algorithm for the \emph{``Suite S\'ecante'' Problem}. Our goal is to design an algorithm that solves the following problem:
\begin{problem}[{\bf ``Suite S\'ecante'' Problem}]
Given $f\in \scrP_{(d)}$ decide whether $f$ is a ``Suite S\'ecante''. Namely, decide whether the following property holds:\\
The algebraic variety $V_\A(f_1,\ldots, f_m)\subseteq \A^n$ is a non-empty variety of dimension $n-m$.
\end{problem}
But we do not want to use any of the standard techniques in Computational Algebraic Geometry. In this field, the input is usually a list of polynomials and the goal is to determine, among other things, the dimension of the affine set of their common zeros. This is done by manipulating the input list in different forms (computing a Gr\"obner basis of the ideal, using variations of Bertini's Theorems to approximate the zero set by a complete intersection variety, etc.). \emph{Our method exhibits the ``strange'' power of correct test sequences: Our algorithm answers to this problem in polynomial time doing no manipulation of the input list of polynomials.} \emph{It just evaluates the input list of polynomial equations $f$  at some well-suited points, i.e., it computes $f(x_1),\ldots, f(x_L)\in K^m$} for some $x_1, \ldots, x_L$.

We restrict to the case of \emph{fields which are well-suited for CTS's guessing}. We say that a field $K$ is well-suited for CTS's guessing in zero-dimensional varieties if it satisfies the following property: \\
 For every $R\in \N$ and every positive integer $n\in \N$, there is a zero-dimensional variety $V_{R}\subseteq \A^n(K)$ of degree $R^n$ given by polynomial equations of degree at most $R$, and such that the following task may be performed with at most $O(n \log_2(R))$ arithmetic operations:
$${\hbox {\bf guess at random}}\;\; x\in V_R.$$

Thus, we prove Theorem \ref{suite-secante-BPP:teor}, which we reproduce here in order to improve accessibility for the reader:

\begin{main}\label{suite-secante-BPP:main}
If $K$ is a field well-suited for CTS's guessing, the problem ``Suite S\'ecante'' with restricted inputs is in ${\bf BPP}_K$. Namely, let $(d)$ be a degree list and let $\Omega\subseteq \scrP_{(d)}$ be a constructible subset of  lists of polynomials. Let $U\subseteq \scrP_{(d)}$ be the Zariski open subset described in Theorem \ref{variedad-incidencia-propiedades:teor} of lists that are ``Suites S\'ecantes''. Assume that $\Omega \setminus \left(\Omega\cap U\right)$ has dimension at most $\dim(\Omega)-1$. Then, there is an algorithm in ${\bf BPP}_K$ that solves ``Suite S\'ecante'' Problem with inputs in $\Omega$, i.e.:\\
\emph{Given as input $f\in \Omega$ and the data $\dim(\Omega)$ and $\deg_{\rm lci}(\Omega)$, the algorithm decides whether $f$ is a ``Suite S\'ecante'' or not.}\\
The running time of the algorithm in terms of arithmetic operations is at most:
$$O\left( \dim(\Omega)n\left((T +\log(\dim(\Omega))+ \log(d) \right) + Tn\log(\deg_{\rm lci}(\Omega)) \right).$$
where $T$ is the  maximum number of arithmetic operations required to evaluate at one point any list $f:=(f_1,\ldots, f_m)\in\Omega$ and $d:=\max\{ d_1,\ldots, d_m\}$. The error probability is bounded by:
$${{1}\over{\deg_{\rm lci}(\Omega)e^{6m\dim(\Omega)}}}.$$
\end{main}

In the dense input case (i.e. $\Omega=\scrP_{(d)}^K$) the total number of arithmetic operations is of order:
$$O(nN_{(d)}^2),$$
where $N_{(d)}=\dim_K(\scrP_{(d)})$ and the error probability is bounded by:
$${{1}\over{e^{6mN_{(d)}}}}.$$

We finally devote Appendix \ref{Kakeyaetal-CTS:sec} to reformulate two well-known statements of the \emph{Polynomial Method} in terms of correct test sequences:
\begin{itemize}
\item In Subsection \ref{Kakeya:subsec}  we discuss Dvir exponential lower bounds for Kakeya sets (cf. \cite{Dvir}) in terms of correct test sequences. We show that Kakeya sets are correct test sequences for certain constructible sets and, hence, Dvir's lower bound is a consequence of the \emph{curse of dimensionality} discussed above (in Corollary \ref{Dvir-CTS-curse:corol}). We also prove that most correct test sequences are not Kakeya sets when the degree of the involved polynomials is bounded by $q^{1-\varepsilon}-1$ for small $\varepsilon>0$ (see Corollary \ref{CTS-not-Kakeya:corol}).
\item In Subsection \ref{CTS-Alon-combinatorio:subsec} we also see that Alon's Combinatorial Nullstellensatz (cf. \cite{Alon}) is also related to our treatment of  correct test sequences. In particular, we show how duality in zero-dimensional reduced algebras is related to Alon's main outcome in \cite{Alon}.
\end{itemize}

\section[B\'ezout's Inequality]{B\'ezout's Inequality for constructible sets}
\label{Bezout-inequality:sec}
In the early  80's of the last century, three authors have independently arrived to three different proofs of some B\'ezout's Inequalities. They took different approaches of the notion of degree of an algebraic variety, and they used different techniques to achieve their outcomes, although the underlyign notion has some coincidences. These three authors were W. Vogel
(cf. \cite{Vogel}), J. Heintz (cf. \cite{Heintz83}) and W. Fulton (cf. \cite{Fulton}). Neither Fulton nor Vogel dealt with constructible sets, whereas Heintz developed \cite{Heintz83} an affine
version of the notion which was systematically called \emph{degree of constructible sets}.
Nevertheless, as already observed in \cite{Heintz85}, Heintz's statements do not hold for constructible sets and they only hold for locally closed subsets of some affine space. The question that remained open was to introduce a right notion of degree of constructible sets and prove that this notion satisfies a B\'ezout's Inequality. In this section we close that gap.

\subsection{Closed, locally closed and constructible sets in Zariski's topopology}
As in the introduction, $K$ is an algebraically closed  field and $K[X_1,\ldots, X_n]$ is the ring of polynomials in the set of variables $\{X_1,\ldots, X_n\}$ with coefficients in $K$. We shall denote by $\A^n(K)$ (or $\A^n$ when no confusion may arise) the affine space of dimension $n$ over $K$.
Given a finite set of polynomials $\{f_1,\ldots, f_s\}\subseteq K[X_1,\ldots, X_n]$,  we denote by $V_\A(f_1,\ldots, f_s)\subseteq \A^n(K)$ the algebraic variety of the common zeros in $\A^n(K)$ of these polynomials. Namely,
$$V_\A(f_1, \ldots, f_s):=\{ x\in \A^n(K) \; : \; f_1(x)=f_2(x)=\cdots = f_s(x)= 0 \}.$$
If ${\frak a}\subseteq K[X_1,\ldots, X_n]$ is an ideal, we also denote by $V_\A({\frak a})\subseteq \A^n(K)$ the set of common zeros of all polynomials in ${\frak a}$. Obviously, if ${\frak a}$ is the ideal generated by the finite set $\{f_1,\ldots, f_s\}$ (i.e. ${\frak a}=(f_1,\ldots, f_s)$), we then have $V_\A({\frak a})=V_\A(f_1, \ldots, f_s)$.

There is a unique topology in $\A^n(K)$ whose closed sets are affine algebraic varieties, which is usually known as the \emph{Zariski topology} on $\A^n(K)$. For each subset
$S\subseteq \A^n$ we denote by $\overline{S}^z$ the closure of $S$ with respect to the Zariski topology in $\A^n$. For every subset $X\subseteq \A^n$, the topology induced on $X$ by the Zariski topology  of $\A^n$ will be called the Zariski topology of $X$. We then use the terms  \emph{Zariski open in $X$}
or \emph{Zariski closed in $X$} to describe open and closed sets in $X$ with respect to its Zariski topology.

\begin{definition}[{\bf Locally closed and constructible sets in Zariski's topology of $\A^n$}]
The intersection of an open and a closed subsets of the Zariski topology of $\A^n$ will be called \emph{locally closed} subset. A locally closed subset $V\subseteq \A^n$ is called \emph{irreducible} if  $V$ is an open subset in a closed irreducible variety of $\A^n$. Finite unions of locally closed sets are called  \emph{constructible sets} (in the most classical tradition of Elimination Theory, see \cite{Chevalley}).
\end{definition}

Some authors prefer to use the term \emph{quasi-projective variety} instead of locally closed (cf. \cite{Shafarevich}, for instance). As the context of this  manuscript is essentially affine, we avoid to discuss the projective terminology, except if required.

The classical Chevalley Theorem explains the need of using constructible sets in Elimination Theory: They are  the projections of closed sets.

\begin{theorem}[\cite{Chevalley}]
A subset $C\subseteq \A^n(K)$ is constructible if and only if there is some $m\in \N$ and some algebraic variety $V\subseteq \A^{n+m}(K)$ such that $C:=\pi(V)$, where $\pi:\A^{n+m}(K)\longrightarrow \A^n(K)$ is the canonical projection.
\end{theorem}

As affine Zariski topology is Noetherian, locally closed sets admit a minimal decomposition as finite union of locally closed irreducible sets. Uniqueness up to permutation of these minimal decompositions of locally closed sets allow us to use the term \emph{locally closed irreducible components}. For a locally closed subset $V\subseteq \A^n$, its locally closed irreducible components are uniquely determined because they are in bijection with the irreducible components of its Zariski closure $\overline{V}^z$. This fact is easy to prove since every locally closed irreducible set is dense in its Zariski closure. We see later that this uniqueness cannot be immediately extended  to constructible sets.

For every locally closed subset $V\subseteq \A^n$ we denote by $K[V]$
the ring of polynomial  functions defined on $V$.  A polynomial map between two locally closed sets $\varphi:V\longrightarrow W$  is called \emph{dominant} if the image $\varphi(V)$ is Zariski dense in $W$ (which is equivalent to the fact that $K[W]$ is embedded as subring of $K[V]$). If $V$ is irreducible, $K[V]$ is an integral domain and we denote by $K(V)$ the field of rational functions defined on $V$ (i.e. the field of fractions of $K[V]$).

The \emph{dimension} of a constructible subset $C\subseteq \A^n$ will be its Krull dimension as topological space. Namely, the length of the longest chain of locally closed irreducible subsets of $\A^n$ included in $C$. Constructible sets also admit a decomposition into irreducible components (cf. \cite{Chevalley}, for instance). As we shall see in Example \ref{La-Croix-de-Berny:ej}, irreducible components of constructible sets are not always closed irreducible subsets.  We have the following decomposition:

\begin{lemma}\label{descomposicion-union-irreducibles-distintos:lema} Let $C\subseteq \A^n(K)$ be a constructible subset. Then, there is a finite set $\scrC:=\{ U_1\cap V_1, \ldots, U_s\cap V_s\}$ of locally closed irreducible sets such that the following properties hold:
\begin{equation}\label{descomposicion-union-irreducibles-distintos0:eqn}
C:=(U_1\cap V_1)\cup \cdots \cup (U_s\cap V_s),
\end{equation}
and
\begin{enumerate}
\item $V_i$ is an irreducible algebraic variety in $\A^n$,
\item $U_i$ is the maximum of the Zariski open subsets $O_i\subseteq \A^n$  such that $O_i\cap V_i\subseteq C$,
\item $V_i\not= V_j$.

\end{enumerate}
\end{lemma}

\begin{proof}As $C$ is a finite union of locally closed subsets $C=W_1\cup \cdots \cup W_s$, just taking the irreducible components of each $W_i$ we conclude that $C$ is a finite union of locally closed irreducible sets. As non-empty open subsets of irreducible varieties are dense in the Zariski topology, just  rearranging these locally closed irreducible sets, we may assume a decomposition as the following one:
\begin{equation}\label{descomposicion-union-irreducibles-distintos:eqn}
C:=(O_1\cap V_1)\cup \cdots \cup (O_r\cap V_r),
\end{equation}
where for each  $i$, $1\leq i \leq r$, $O_i\subseteq \A^n$ is a Zariski open subset and $V_i\subseteq  \A^n$ is a locally closed irreducible subset such that $V_i\not= V_j$ for each $i\not = j$. If there were some  $i$  and $j$  such that $V_i=V_j$, we just replace $(O_i\cap V_i) \cup (O_j\cap V_j)$ by $(O_i\cup O_j)\cap V_i$. Thus, we may assume that condition  $iii)$ holds. Finally, it suffices to take the maximum of the Zariski open subsets $U_i\subseteq \A^n$ such that $U_i\cap V_i\subseteq C$ to finish the proof. The existence of this maximum is guaranteed by the fact that the union of open sets is open in any topology. \end{proof}

\begin{example}[{\bf La Croix de Berny}]\label{La-Croix-de-Berny:ej}
We first observe that the decomposition of constructible sets into irreducible locally closed subsets is not unique. In addition, we see that even if polynomial images of algebraic varieties are always constructible sets (Chevalley's Theorem, cf. \cite{Chevalley}), they are not always locally closed. Let us consider the following cubic hyper-surface of $\A^3(\C)$:
\begin{equation}\label{hiper-superficie-Croix-Berny:eqn}
W:= \{ (x,y,z) \in \A(\C)^3\; : \; zxy + (x^2+y^2-1)=0\}.
\end{equation}
Note that this hyper-surface is irreducible since the polynomial $h(X,Y,Z)= Z XY + (X^2+Y^2-1)\in \C[X,Y][Z]$ is a primitive polynomial (as $gcd(XY, X^2+Y^2-1)=1$ in $\C[X,Y]$) of degree $1$. The degree of this hyper-surface is $3$ according to the notion of degree in \cite{Heintz83}.\\
Let $\pi: \A^3(\C)\longrightarrow \A^2(\C)$ be the projection that forgets the coordinate $z$ (i.e.
$\pi(x,y,z):= (x,y)$, for all $(x,y,z)\in \A^3(\C)$). We then observe that  $\pi(W)$ is the following constructible set:
$$C:=\pi(W)=\{ (x,y)\in \A^2(\C)\; :\; xy \not= 0\} \cup\{ (0,1), (0,-1), (1,0), (-1,0)\}.$$
This is a decomposition of $C$ into five locally closed irreducible sets which satisfies the properties of the precedent Lemma \ref{descomposicion-union-irreducibles-distintos:lema}:
$$C:= C_1 \cup \{(1,0)\}\cup \{(-1,0)\} \cup \{(0,1)\} \cup \{(0,-1)\},$$
where $C_1\subseteq \A^2(\C)$ is the following Zariski open subset:
$$C_1:=\{ (x,y)\in \A^2(\C)\; :\; xy\not=0\}.$$
We also have another decomposition of $C=\pi(W)$ as the union of two locally closed irreducible  subsets:
$$C:= C_ 1 \cup\{ (x,y)\in \A^2(\C) \; :\; x^2+y^2-1=0\}.$$
This happens because $C$ is not locally closed. Moreover, the irreducible components of the Zariski closure $\overline{C}^{z} $ of $C$, which is the irreducible variety $\A^2(\C)$, are not in bijection with the Zariski closures of the irreducible sets occuring in such decompositions. In particular,
 the irreducible components of $C$ as Noetherian topological space (in the sense of Definition 9, Proposition 9 and pages 22-23 of  \cite{Chevalley}) are not always locally closed sets.
\end{example}

Before to proceed, let us see a basic problem with dimension in constructible sets:

\begin{example}[{\bf Local versus global dimension in constructible sets}]
We reconsider the example of {La Croix de Berny} introduced in Example \ref{La-Croix-de-Berny:ej} above.
There, we considered the following constructible set:
$$C:=\{ (x,y)\in \A^2(\C)\; :\; xy \not= 0\} \cup\{ (0,1), (0,-1), (1,0), (-1,0)\}.$$
Let $C_x$ be the local germ of $C$ at the point $x=(1,0)$ with respect to the Zariski topology (the arguments also hold for the Euclidean topology in $\A^2(\C)$). It is obvious that the dimension of $C$ at $x$ is 2: Any neighborhood of $x$ contains a piece of dimension $2$ of $C$. In particular, $C$ has \emph{local} and \emph{global} (at any point $z\in C$)  dimension equal to $2$, however, it cannot  be represented as a finite union of locally closed irreducible varieties of dimension $2$. If  $C$ were the union of several locally closed irreducible subsets of dimension $2$, then $C$ would be a Zariski open subset of $\A^2(\C)$, which is not the case as we have seen in Example \ref{La-Croix-de-Berny:ej}.
\end{example}

Uniqueness of locally closed irreducible ``components'' is only granted for highest dimension components.

\begin{proposition}[{\bf Components of higher dimension of a constructible set}] \label{componentes-alta-dimension:def}
With the same notations and assumptions as above, if $C\subseteq \A^n$ is a constructible set, then $C$ admits a decomposition satisfying $i), ii)$ and $iii)$  of Lemma
 \ref{descomposicion-union-irreducibles-distintos:lema} above:
$$C:=(U_1\cap V_1)\cup \cdots \cup (U_s\cap V_s),$$
such that there exists $r\leq s$, verifying:
$$\dim(C)=\dim(V_i) \Longleftrightarrow 1\leq i\leq r.$$
The varieties  $V_1, \ldots, V_r$ are uniquely determined and they are the irreducible components of higher dimension of the Zariski closure of $C$. The varieties $V_{r+1}, \ldots, V_s$ are called embedded components of $C$.
\end{proposition}

\begin{definition}[{\bf Globally equi-dimensional constructible sets}]
\label{descomposicion-equidimensional-constructibles:def}
Let $C\subseteq \A^n$ be a constructible subset. We say that $C$ is \emph{globally equi-dimensional} if there is a decomposition of $C$ as finite union of locally closed irreducible subsets:
$$C:=(U_1\cap V_1)\cup \cdots \cup (U_s\cap V_s),$$
that satisfies the conditions of Lemma  \ref{descomposicion-union-irreducibles-distintos:lema} and such that the dimension also verifies:
$$\dim(C)=\dim(U_i\cap V_i) = \dim (V_i), \; 1\leq i \leq s.$$
\end{definition}

As we have seen, the constructible set described in Example \ref{La-Croix-de-Berny:ej} is not globally equi-dimensional, although it is locally equi-dimensional and it is the projection of a hyper-surface. In particular, locally closed subsets such that their Zariski closure is equi-dimensional are globally equi-dimensional.

\begin{remark}[{\bf Decompositions into equal dimension globally equi-dimensional pieces}]\label{decomposition-equal-dimension:rk}
Given a decomposition of a constructible set $C=W_1\cap\cdots \cap W_s\subseteq \A^n(K)$, with $W_i=U_i\cap V_i$,  as in Lemma \ref{descomposicion-union-irreducibles-distintos:lema}, we may also decompose $C$ as a finite union of globally equi-dimensional constructible sets:
$$C:=C_r\cup \dots \cup C_0,$$
where $r=\dim(C)$ and
$$C_k:= \bigcup_{\dim(W_i)=k} W_i.$$
We call such decomposition a \emph{equal dimension decomposition of $C$}.
As we have seen in Example \ref{La-Croix-de-Berny:ej}, equal dimension decompositions are not unique. The example described there admits several equal dimension decompositions:
$$C:=\{(x,y)\in \A^2(\C)\; : \; xy\not=0\}\cup \{ (\pm 1, 0), (0, \pm 1)\},$$
and
$$C:=\{(x,y)\in \A^2(\C)\; : \; xy\not=0\}\cup \{(x,y)\in \A^2(\C)\; : \; x^2+y^2-1=0\}.$$
\end{remark}

\subsection{Degree of locally closed sets according to \cite{Heintz83} and their B\'ezout's Inequality}
\label{Bezout-Inequality-locally-closed:subsec}

We now resume some of the statements of \cite{Heintz83} which hold for locally closed sets. Let $V\subseteq \A^n(K)$ be a locally closed irreducible subset of (Krull) dimension $r$. Let
 $\A^{nr}:=\A^{nr}(K)=\scrM_{r\times n} (K)$ be the space of $r\times n$ matrices with coordinates in $K$. Let  $\A^r:=\A^r(K)$ be  the affine space of dimension $r$ over $K$.  As in \cite{Heintz83}, we consider the following polynomial mapping:
 \begin{equation}\label{morfismo-dominante-grado:eqn}
\begin{matrix}
\Phi: & \A^{nr}\times V & \longrightarrow & \A^{nr}\times \A^r\\
& (M, x) & \longmapsto & (M, Mx)\end{matrix}
\end{equation}

The following statement was proved in \cite{Heintz83} (cf. also \cite{Res-Geom}):

\begin{proposition}\label{existencia-grado:prop} With the previous notations, we have:
\begin{enumerate}
\item The morphism $\Phi$ is a dominant morphism and the following is a finite and separable field extension:
$$\Phi^*: K(\A^{rn}\times \A^r)\longhookrightarrow K(\A^{rn}\times V).$$
\item For every point  $(M,b)\in \A^{rn}\times \A^r$, if the fiber $\Phi^{-1}(\{(M,b)\})$ is finite, then we have:
$$\sharp\left(\Phi^{-1}(\{(M,b)\})\right)\leq [K(\A^{rn}\times V)\; :\; K(\A^{rn}\times \A^r)].$$
 \item There is a Zariski open subset ${\mathcal U}\subseteq \A^{rn}\times \A^r$ such that for every $(M,b)\in {\mathcal U},$  the fiber $\Phi^{-1}(\{(M,b)\})$ is finite and satisfies:
$$\sharp\left(\Phi^{-1}(\{(M,b)\})\right)= [K(\A^{rn}\times V)\; :\; K(\A^{rn}\times \A^r)].$$
\end{enumerate}
\end{proposition}

Let us now consider $GL(n, r)\subseteq \A^{r(n+1)}$ the Zariski subset of all pairs $(M,b)\in \A^{r(n+1)}$ such that $\rank(M)=r$. Namely, given a pair $(M,b)\in \A^{r(n+1)}$ we consider the linear affine variety $\G(M,b)$ given by the following identity:
$$\G(M,b):=\{ x\in \A^n\; : \; Mx^t-b=0\},$$
where $x^t$ is the transpose of $x=(x_1,\ldots, x_n)$. Then, $GL(n,r)$ is the set of all pairs $(M,b)$ such that the co-dimension satisfies $\codim(\G(M,b))=r$. We then consider the ``Grassmannian'' $\G(n,r)$ of all linear affine varieties in $\A^n$ of co-dimension $r$ and we have the following onto mapping:
$$\begin{matrix}
\scrG: & \A^{rn}\times \A^r & \longrightarrow & \G(n,r)\\
& (M, b) & \longmapsto & \G(M,b).\end{matrix}$$
Note that \emph{the Zariski topology in $\G(n,r)$} is defined as
the final topology induced in $\G(n,r)$ by $\scrG$ and the Zariski topology in $GL(n,r)$. Let $V\subseteq \A^n(K)$ be a locally closed irreducible subset of dimension $r$ and let us introduce the class:
$$\G(V):=\{ A \in \G(r,n)\; :\; \sharp(V\cap A)< \infty\}.$$
The following properties immediately follow from \cite{Heintz83}:

\begin{proposition}\label{interpretacion-geometrica-grado:prop}
With the same notations and assumptions, we have:
\begin{enumerate}
\item The class  $\G(V)$ is non-empty and contains an open subset of  $\G(n,r)$.
\item The maximum $\max\{ \sharp\left(A \cap V\right)\; : \; A \in \G(V)\}$ is finite.
\item There is a Zariski open subset $\CU\subseteq \G(n,r)$ such that for all $A \in \CU$ we have:
$$\sharp\left(A\cap V \right)= \max\{\sharp\left(T\cap V\right)\; : \; T\in \G(V)\}.$$
\end{enumerate}
\end{proposition}

\begin{definition}[{\bf Degree  of a locally closed subset}]\label{grado-localmente-cerrados:def}
Let $V\subseteq \A^n(K)$ be a locally closed irreducible subset. We define the degree of $V$ as the following quantity:
$$\deg(V):=\max\{ \sharp\left( A \cap V\right) \; : \; A \in \G(n,r), \; \sharp(A\cap V)< \infty\}.$$
Let $W\subseteq \A^n(K)$ be any locally closed subset and let  $C_1,\ldots, C_s$ its locally closed irreducible components. The degree  of $W$ is defined as:
$$\deg(W):=\sum_{i=1}^s \deg(C_i).$$
\end{definition}

Some immediate properties are resumed on the following Proposition:

\begin{proposition}\label{propiedades-basicas-grado-lc:prop}
 With the same notations and assumptions as above, we have:
\begin{enumerate}
\item For every locally closed subset $V\subseteq \A^n$, its degree agrees with the degree of its Zariski closure, i.e.
$$\deg(V)=\deg(\overline{V}^z).$$
 \item The degree of a finite set $C\subseteq \A^n$  equals its cardinal:
 $$\deg(C)= \sharp(C).$$
\item The number of irreducible components of a locally closed set is bounded by its degree.
\item The degree of any linear affine variety is $1$.
\item The degree of locally closed sets is invariant by linear or affine isomorphisms.
\item  For every non-constant $f\in K[X_1,\ldots, X_n]$, the degree of the hyper-surface $V_\A(f)$ is at most the degree of the polynomial $\deg(f)$ and $\deg(V_\A(f))=\deg(f)$ if and only if $f$ is square-free.
\end{enumerate}
\end{proposition}

As in \cite {Heintz83}, the following Lemmata also hold:

\begin{lemma}\label{interseccion-con-lineal-no aumenta-grado:lema}
Let $V\subseteq \A^n(K)$ be a locally closed subset and  $A\subseteq \A^n(K)$ be a linear affine variety. Then, we have:
$$\deg(V\cap A)\leq \deg(V).$$
\end{lemma}

\begin{lemma}\label{abierto-zariski-sucesion-regular:proposition} Let $V\subseteq \A^n$  be an equi-dimensional  locally closed set of dimension $r$. Then, we have:
$$\deg(V)=\max\{ \sharp\left( A\cap V\right) \; : \; A\in \G(n,r), \; \sharp(A\cap V)< \infty\}.$$
Moreover, there is a non-empty Zariski open subset ${\CU}\subseteq \A^{nr+r}$ such that for all  $(M,b)\in {\mathcal U}$ the following holds:
$$\sharp\left( \G(M,b)\cap V\right)=\deg(V),$$
where, as above,  $\G(M,b):=\{ x\in \A^n \; : \; Mx^t=b\}$.
\end{lemma}

As in Remark 2 (2) of \cite{Heintz83}, the degree of the Zariski closure is preserved by taking images by linear transformations:
\begin{proposition}\label{grado-preservado-por-transformaciones-lineales:prop} Let $\ell: \A^n (K)\longrightarrow \A^m(K)$ be a linear mapping and  $V\subseteq \A^n(K)$ a locally closed subset. Then, we have:
$$\deg(\overline{\ell(V)}^z) \leq \deg(V).$$
Moreover, if $\ell(V)$ is locally closed we obtain:
$$\deg(\ell(V)) \leq \deg(V).$$
\end{proposition}

The heart of the new outcomes in \cite{Heintz83} is the fact that the degree of locally closed sets  perfectly behaves  with respect to Cartesian products:

\begin{theorem}\label{grado-producto-cartesiano:prop}
Let $V\subseteq \A^n$ and $W\subseteq \A^m$ be two locally closed sets. Let $V\times W\subseteq\A^{n+m}$ be its Cartesian product, which is also locally closed. Then, we have:
$$\deg(V\times W) = \deg(V) \deg(W).$$
\end{theorem}

The previous results allow us to exhibit  the \emph{B\'ezout's Inequality for locally closed sets} of \cite{Heintz83}. The original proof is correct and its main ingredient is to consider $V\cap W$ as the projection of $(V\times W)\cap \Delta^{(2n)}$, where $\Delta^{(2n)} \subseteq \A^n\times \A^n$ is the diagonal subvariety in the Cartesian product space. As the image $V\cap W$ is locally closed, using Proposition \ref{grado-preservado-por-transformaciones-lineales:prop} combined with Lemma \ref{interseccion-con-lineal-no aumenta-grado:lema} and Theorem \ref{grado-producto-cartesiano:prop}, the following statement naturally follows:

\begin{theorem}[{\bf B\'ezout's Inequality for locally closed sets, \cite{Heintz83}}]\label{Bezout-localmente-cerrados:teor}
Let $V, W\subseteq \A^n(K)$ be two locally closed subsets. Then, we have:
$$\deg(V\cap W)\leq \deg(V)\deg(W).$$

\end{theorem}

\subsection{Two notions of degree and two B\'ezout's Inequalities for constructible sets}\label{Bezout-Inequality-constrictibles:subsec}

A first notion of degree for constructible sets was introduced in \cite{Heintz83}: The ``degree'' of a constructible set $C\subseteq \A^n(K)$ was there defined as the degree of its Zariski closure. Let us call it the \emph{$Z-$degree of $C$}, which we denote by $\deg_z(C):=\deg(\overline{C}^z)$. That notion of degree does not satisfy a B\'ezout's Inequality  for constructible sets. The following example shows that fact.

\begin{example}[{\bf $Z-$degree does not satisfy a B\'ezout's Inequality for constructible sets}]\label{croix-de-berny-def:ej}
Let us consider again the Example \ref{La-Croix-de-Berny:ej} above. We had the cubic irreducible hyper-surface:
$$ W:= \{ (x,y,z) \in \A(\C)^3\; : \; zxy + (x^2+y^2-1)=0\}$$
and its image $C:=\pi(W)$ under the canonical projection $\pi: \A^3(\C) \longrightarrow \A^2(\C)$. We observed that $\overline{C}^z=\A^2(\C)$ and, hence, the degree of its Zariski closure satisfies $\deg_z(C)=1$. We may also consider the pair of lines $V:=\{ (x,y)\in \A^2(\C) \; : xy=0\}$, which is a degree $2$ algebraic variety. We consider  its intersection $C\cap V$ and we see that $C\cap V$ are the following four points in the plane:
$$C\cap V=  \{ (0,1), (0,-1), (1,0), (-1,0)\},$$
Therefore, $\deg_z(C\cap V)=\deg(C\cap V)=4$. Nevertheless, we have that:
$$4\not\leq \deg_z(C)  \deg_z(V) = 2.$$
\end{example}

This Example basically means that statements in \cite{Heintz83} concerning constructible sets are \emph{either wrong or incomplete while using $Z-$degree of a constructible set as notion of degree}. This was already observed in \cite{Heintz85}, where the author restricted the correctness of all statements in \cite{Heintz83} to the case of locally closed sets  (as we did in  Subsection \ref{Bezout-Inequality-locally-closed:subsec} above). We have not found in the literature a notion of degree of constructible sets that generalizes the notion of degree of locally closed sets and also satisfies a B\'ezout's Inequality. This pushed us to find a good notion of degree of constructible sets that satisfies a B\'ezout's Inequality and this lead us to the following two notions of degree of a constructible set: The minimal degree of a presentation as union of locally closed irreducible sets (\emph{$LCI$-degree}) and the minimal degree of a presentation as projection of a locally closed set (\emph{$\pi-$degree}).

\begin{definition}[{\bf $LCI$-degree of a constructible set}]\label{grado-constructibles:def} Let $C\subseteq \A^n$ be a constructible subset. Let $\scrC:=\{ U_1\cap V_1,\ldots, U_r\cap V_r\}$ be a decomposition of $C$  as finite union of locally closed irreducible sets that satisfies Lemma \ref{descomposicion-union-irreducibles-distintos:lema}. We define the degree of $C$ relative to this decomposition as:
 $$\deg(C,\scrC):=\sum_{i=1}^r \deg(V_i).$$
We finally define the $LCI$-degree of $C$ as the minimum of all the degrees of all decompositions of this kind:
 $$\deg_{\rm lci}(C):=\min\{ \deg(C,\scrC)\; :\; {\hbox {\rm $\scrC$ is a decomposition of $C$ that satisfies Lemma \ref{descomposicion-union-irreducibles-distintos:lema}}}\}.$$
We say that a decomposition $\scrC:=\{ U_1\cap V_1,\ldots, U_r\cap V_r\}$ of a constructible set $C\subseteq \A^n$  is a minimum $LCI$-degree decomposition of $C$ if $\scrC$ satisfies  Lemma \ref{descomposicion-union-irreducibles-distintos:lema} and also:
$$\deg_{\rm lci}(C)=\sum_{i=1}^r \deg(V_i).$$
 \end{definition}

\begin{definition}[{\bf $\pi-$degree of a constructible set}] \label{grado-como-proyeccion:def} As above, let $C\subseteq \A^n$ be a constructible set and we consider the class of all locally closed subsets that project onto $C$. Namely,
for every $m\in \N$, $m\geq n$, we define:
$$\Pi_m(C):=\{ V \subseteq \A^m\; :  {\hbox {\rm $V$ is locally closed and  $\pi(V)=C$}}\},$$
where $\pi: \A^m\longrightarrow \A^n$ is the canonical projection. We then define:
$$\Pi(C):=\bigcup_{m\geq n} \Pi_m(C),$$
and we define the projection degree (also $\pi-$degree) of $C$ as the following minimum:
$$\deg_\pi(C):=\min \{ \deg_z(V)\; : \; V\in \Pi(C)\}.$$
\end{definition}

We obviously observe, taking $m=n$, that if $C$ is locally closed, then using the ``projection'' $\pi:=Id_n:\A^n \longrightarrow \A^n$ we obtain:
\begin{equation}\label{deg-pi-deg-z:eqn}
\deg_\pi(C)\leq \deg(C)=\deg_z(C).
\end{equation}

The three notions of degree are sub-additive functions:

 \begin{proposition}\label{grado-constructibles-subaditivo:prop} The three notions $\deg_z$, $\deg_\pi$ and $\deg_{\rm lci}$ are sub-additive functions. Namely, given $C,C'\subseteq \A^n$ two constructible subsets. We have:
 $$\deg_z(C\cup C')\leq \deg_z(C) + \deg_z(C'),$$
  $$\deg_\pi(C\cup C')\leq \deg_\pi(C) + \deg_\pi(C')$$
and,
 $$\deg_{\rm lci}(C\cup C')\leq \deg_{\rm lci}(C) + \deg_{\rm lci}(C').$$
 \end{proposition}
 \begin{proof} As the closure of a finite union is the union of the closures, the first inequality obviously holds and $\deg_z$ is sub-additive. As for the second inequality, it also follows almost immediately. We proof it for completeness. Assume $V\subseteq \A^{m_1}$  and $V'\subseteq \A^{m_2}$ are two locally closed subsets such that $m_1\geq m_2\geq n$ and $C:=\pi_1(V)$, and $C':=\pi_2(V')$, where
 $\pi_1:\A^{m_1}\longrightarrow \A^n$ and $\pi_2: \A^{m_2}\longrightarrow \A^n$ are the two canonical projections. Assume also that $\deg_\pi(C)=\deg_z(V)$ and $\deg_\pi(C')= \deg_z(V')$. As $m_1\geq m_2$, we may also define $V'':=V'\times \A^{m_1-m_2}$ and by Theorem  \ref{grado-producto-cartesiano:prop} we know that $\deg(V'')= \deg(V')$. It is also clear that $\pi_1(V'')=\pi_2(V')=C'$ and that $\pi_1(V \cup V'')= C\cup C'$. Hence, we have:
 $$\deg_\pi(C\cup C') \leq\deg_z(V\cup V'')\leq \deg_z(V) + \deg_z(V'') = \deg_\pi(C) + \deg_\pi(C').$$
 The third inequality is an immediate consequence of Definition \ref{grado-constructibles:def}. From two minimum $LCI$-degree decompositions of  $C$ and $C'$ we may reconstruct a decomposition of $C\cup C'$ satisfying the properties described in Lemma \ref{descomposicion-union-irreducibles-distintos:lema}. The degree of such decomposition is obviously greater than the degree of the union $C\cup C'$.
 \end{proof}

\begin{example}[{\bf Three different ``degrees''}]\label{La-Croix-de-Berny-grados-varios:ej} The three notions $\deg_z$, $\deg_\pi$ and $\deg_{\rm lci}$ are different. In order to simplify the arguments, we slightly modify the constructible set exhibited in Example \ref{La-Croix-de-Berny:ej}. We first consider the following quadratic hyper-surface:
\begin{equation}\label{semi-croix-de-berny-hiper-superficie:eqn}
W':=\{(x,y,z)\in \A^3(\C) \; : \; xz+ (y^2-1)=0\},
\end{equation}
and the constructible subset  $C :=\pi(W')\subseteq \A^2(\C)$, where $\pi:\A^3(\C)\longrightarrow \A^2(\C)$ is the canonical projection that forgets the variable $z$. Since $C$ is Zariski dense in $\A^2(\C)$, we obviously have that $\deg_z(C)=1$.\\
As $C$ is the projection of $W'$ and $\deg(W')=2$, we have $\deg_\pi(C)\leq 2$. Moreover, $\deg_\pi(C)\not=1$ since, otherwise, $C$ would be the projection of a non-empty Zariski open subset of a linear affine variety, which cannot be possible because $C$ is not an open subset of $\A^2(\C)$.\\
Additionally, we claim that $\deg_{\rm lci}(C)=3$. We have the following decomposition of $C$:
\begin{equation}\label{semi-croix-de-Berny-decomposition:eqn}
C:= \{ (x,y)\in \A^2(\C)\; :\;  x\not=0\} \cup \{ (x,y)\in \A^2(\C)\; :\; x^2+y^2-1=0\}.
\end{equation}
Hence, we conclude that $\deg_{\rm lci}(C)\leq 3$.  Let us see now that $2<\deg_{\rm lci}(C)$. For if $\deg_{\rm lci}(C)=2$, as the open set $C_1=\A^2(\C)\setminus V_\A(X)$ is maximal ($C$ is not locally closed), then $C$ would decompose as a union of at most two locally closed irreducible sets:
$$C= \{ (x,y)\in \A^2(\C)\; :\;  x\not=0\} \cup B,$$
where $B$ is locally closed irreducible of degree  $1$ (namely, $B$ is a Zariski open subset of a linear affine subvariety of $\A^2(\C)$). We may now consider the following intersection:
$$C\cap \{(x,y)\in \A^2(\C)\; : \;  x=0\}=B\cap\{(x,y)\in \A^2(\C)\; : \;  x=0\}.$$
Since $C_1\cap \{(x,y)\in \A^2(\C)\; : \;  x=0\}=\emptyset$, $C\cap \{(x,y)\in \A^2(\C)\; : \;  x=0\} = \{ (0,1), (0,-1)\}$ and B\'ezout's Inequality for locally closed subsets holds, we would conclude:
$$2=\sharp\left(B\cap\{(x,y)\in \A^2(\C)\; : \; x=0\}\right)\leq \deg(B)\deg\left(\{(x,y)\in \A^2(\C)\; : \;  x=0\}\right)= 1\cdot 1= 1,$$
which is false. Thus, we have proven that $\deg(B)\geq 2$ and, hence,  $3\geq \deg_{\rm lci}(C)= 1 + \deg(B)\geq 3$ as wanted. Thus, there are constructible sets $C\subseteq \A^2 (\C)$ such that the three notions of degree above take different values:
$$1= \deg_z(C) < 2 = \deg_\pi(C) < \deg_{\rm lci}(C)=3.$$
\end{example}

There is no uniqueness in the embedded components of a decomposition of constructible sets  that minimize $LCI$-degree:
\begin{remark}[{\bf Non-uniqueness of embedded components minimizing $LCI-$degree}]\label{no-unicidad-de-la-descomposicion-minimal:rk}
 As in Example \ref{La-Croix-de-Berny-grados-varios:ej} above, we have a constructible set $C\subseteq \A^2(\C)$, dense in $\A^2(\C)$ for the Zariski topology and of $LCI$-degree $3$, which is given by the following decomposition as finite union of locally closed irreducible subsets that minimize the degree:
$$C=\{ (x,y) \in \C^2\; :\; x\not=0\}\cup\{ (x,y)\in \C^2\; : \; x^2+ y^2-1=0\}.$$
For every $a\in \C$, we may consider the polynomial $h_a (X,Y) :=X^2+Y^2 + aXY -1$. Then, we see that there are infinitely many decompositions of $C$ as finite union of locally closed sets that minimize the $LCI$-degree as in the following equality:
$$C=\{ (x,y) \in \C^2\; :\; x\not=0\}\cup\{ (x,y)\in \C^2\; : \; h_a(x,y)=0\}.$$
This is because the ideals $(X, X^2+Y^2-1)$ and $(X, h_a (X,Y) )$  are the same ideal in $\C[X,Y]$ for every $a\in \C$.
\end{remark}

Our $LCI$-degree generalizes the notion of degree for locally closed sets introduced in \cite{Heintz83}. We resume this fact in the next Proposition of straightforward proof:

\begin{proposition}\label{generaliza-gradoLCI-localmente-cerrados:prop} Let $C\subseteq \A^n(K)$ be  a constructible set which is either locally closed set or globally equi-dimensional. Then, $\deg_z(C)$ and $\deg_{\rm lci}(C)$ agree. Namely,
 $$\deg_z(C)=\deg(\overline{C}^z)=\deg_{\rm lci}(C).$$
\end{proposition}

\begin{remark}[{\bf LCI-degree and equal dimension decomposition}]\label{LCI-equal-dimension-components:rk}
We may also consider an equal dimensional decomposition of a constructible set $C\subseteq \A^n(K)$ of dimension $r$ as in Remark  \ref{decomposition-equal-dimension:rk}.
As in that Remark, let $\scrD:=\{ C_r,\ldots, C_0\}$ be a decomposition of $C$  as finite union of globally equi-dimensional constructible sets. We  consider the following quantity associated to this decomposition:
 $$\deg_z(C,\scrD):=\sum_{i=0}^r \deg_z(C_i).$$
As $\deg_{\rm lci}$ is sub-additive and because of Proposition \ref{generaliza-gradoLCI-localmente-cerrados:prop}, we have:
$$\deg_{\rm lci}(C)\leq \deg_z(C,\scrD).$$
Moreover, taking  a decomposition of $C$ into locally closed irreducible sets that minimize the degree:
$$C:=W_1\cup \cdots \cup W_s,$$
such that
$$\deg_{\rm lci}(C):=\sum_{i=1}^s \deg(W_i),$$
and reorganizing this decomposition into equal degree components, we obtain a decomposition $\scrW:=\{D_r,\ldots, D_0\}$ such that:
$$\deg_{\rm lci}(C):= \sum_{i=0}^r \deg_{\rm lci}(D_i)=\sum_{i=0}^r \deg_z(D_i).$$
Hence, the $LCI-$degree could have been defined in terms of equal dimension components as follows:
 $$\deg_{\rm lci}(C):=\min\{ \deg_z(C,\scrD)\; :\; {\hbox {\rm $\scrC$ is a globally equi-dimensional decomp. of $C$ as in Rk. \ref{decomposition-equal-dimension:rk}}}\}.$$
\end{remark}

\begin{proposition}\label{relations-between-three-degrees:prop}
For every constructible subset $C\subseteq \A^n(K)$ of dimension $r$, the following inequalities hold:
$$\deg_z(C) \leq \deg_\pi(C) \leq \deg_{\rm lci}(C).$$
Moreover, given an equal dimension decomposition of $C$ into globally equi-dimensional constructible sets
$C:=C_r\cup\cdots \cup C_0$,
we have:
$$\deg_{\rm lci}(C)\leq (\dim(C) + 1) \max \{ \deg_z(C_i)\; : \; 0\leq i \leq r\}.$$
\end{proposition}
\begin{proof}
Let $V\in \Pi_m(C)$ be a locally closed subset of $\A^m$ such that $\pi(V) :=C$, where $\pi:\A^m\longrightarrow \A^n $ is the canonical projection, and $\deg_\pi(C)=\deg_{z}(V)$. Because of Proposition \ref{grado-preservado-por-transformaciones-lineales:prop}, we have:
$$\deg_z(C) = \deg(\overline{\pi(V)}^z)\leq \deg(V) = \deg_\pi(C),$$
which proves the first inequality. As for the second inequality, assume a decomposition of $C$ as finite union of locally closed irreducible subsets:
$$C:=W_1\cup \cdots \cup W_r,$$
such that
$$\deg_{\rm lci}(C)=\sum_{i=1}^r \deg(W_i).$$
As $\deg_\pi$ is sub-additive, we obtain:
$$\deg_\pi(C) \leq \sum_{i=1}^r \deg_\pi(W_i).$$
Thus, as $W_i$ is locally closed, from Inequality \ref{deg-pi-deg-z:eqn}, we conclude $\deg_\pi(W_i)\leq \deg(W_i)$ and the second inequality holds. The last inequality easily follows from Remark \ref{LCI-equal-dimension-components:rk} above.
\end{proof}

\begin{remark} In general, it is not true  that $\deg_{\rm lci}(C) \leq (\dim(C)+1)\deg_z(C)$. Just consider the following example which extends the constructible set introduced in Example \ref{La-Croix-de-Berny-grados-varios:ej}:
$$C':=\{ (x,y)\in \A^2(\C) \; : \; x\not=0\}\cup \{ (0, \pm 1)\} \cup \{(0, \pm 2)\}.$$
One may prove that $\deg_{\rm lci}(C')= 5$ with similar arguments to those used in Example \ref{La-Croix-de-Berny-grados-varios:ej}. Hence, this example proves that
$$5 = \deg_{\rm lci}(C') \not\leq (\dim(C')+1) \deg_z(C') = 3 \cdot 1=3.$$
\end{remark}

Also our $\pi$-degree generalizes the notion of degree for locally closed sets (and globally equi-dimensional constructible sets):

\begin{corollary}\label{generaliza-grado-localmente-cerrados:prop}
Let $C\subseteq \A^n(K)$ be a constructible set. If $C$ is either locally closed or globally equi-dimensional, we have:
 $$\deg_z(C)=\deg(\overline{C}^z)=\deg_{\pi} (C)=\deg_{\rm lci}(C).$$
Moreover, if $C\subseteq \A^n (K)$ is a globally equi-dimensional constructible set of dimension $r$, then there is a non-empty Zariski open subset $\G(C)\subseteq \G(n,r)$ such that the following holds for every $A\in \G(C)$:
$$\sharp\left(A\cap C\right) =\deg_z(C) = \deg_\pi(F)=\deg_{\rm lci}(C).$$
\end{corollary}
\begin{proof}
Combining Propositions \ref{relations-between-three-degrees:prop} and Proposition \ref{generaliza-gradoLCI-localmente-cerrados:prop} the equality immediately follows.
The second claim also follows immediately from the definitions and Proposition \ref{interpretacion-geometrica-grado:prop}.
\end{proof}

In the general case, all what we have for intersections with linear affine varieties is the following statement:

\begin{corollary}\label{intereseccion-lineales-no-equi-dimensional:corol}
Let $C\subseteq \A^n (K)$ be a constructible set of dimension $r$ and let $V_{1}, \ldots, V_{s}$ be the irreducible components of higher dimension of the Zariski closure of $C$ (as in Proposition \ref{componentes-alta-dimension:def}). Then, there is  a non-empty Zariski open subset $\G(C)\subseteq \G(n,r)$ such that the following holds for every $A\in \G(C)$:
$$\sharp\left(A\cap C\right) = \sum_{i=1}^{s} \deg(V_i) \leq \deg_z(C) \leq \deg_{\pi}(C) \leq \deg_{\rm lci} (C).$$
\end{corollary}

The following statements correct Remark 2 (2) and explains Lemma 2 of \cite{Heintz83} about degree of linear images of constructible sets. Again, we begin with the $LCI$-degree:

\begin{proposition}\label{constructibles-grado-imagenes-por-lineal:prop} Let $C\subseteq \A^n$ be a constructible set and $\ell: \A^n \longrightarrow \A^m$ be a linear map. We have:

\begin{enumerate}
\item The degree of the Zariski closure of the image is bounded as follows:
$$\deg_z(\ell(C))=\deg\left( \overline{\ell(C)}^z\right)\leq \deg_{\rm lci}(C).$$
\item If $\ell(C)$ is locally closed or globally equi-dimensional, we also have:
$$\deg_z(\ell(C))=\deg_{\rm lci}\left( \ell(C)\right)\leq \deg_{\rm lci}(C).$$
\item In general, it is not true that given a linear affine variety $A\subseteq \A^m$ the following holds:
    $$\sharp(A\cap \ell(C))\leq \deg_z(C).$$
\item In general, it is not true that $\deg_{\rm lci}(\ell(C)) \leq \deg_{\rm lci}(C)$.
\end{enumerate}
\end{proposition}
\begin{proof}
The first Claim is an immediate consequence of Proposition \ref{grado-preservado-por-transformaciones-lineales:prop} combined with the sub-additivity of the $LCI$-degree (Proposition \ref{grado-constructibles-subaditivo:prop} above). \\
Claim $ii)$ follows from the fact that $\deg_{\rm lci}(C)= \deg_z(C)$ for every subset $C \subseteq \A^n$ that is either locally closed or globally equi-dimensional.\\
For Claim $iii)$, just consider the constructible set $C\subseteq \A^2(\C)$ described in  Example \ref{La-Croix-de-Berny-grados-varios:ej} above and the Identity $Id : C \longrightarrow \A^2(\C)$ as linear functional. Let $A:=\{(x,y)\in \A^2\; :\; x=0\}$ be the line in the plane. Then, we have that $A\cap Id(C)$ is made of two points, whereas $\deg_z(C)=1$. Hence,
$$2=\sharp(A\cap Id(C)) \not\leq \deg_z(C) = 1.$$
For  Claim $iv)$, we use the same constructible subset $C\subseteq \A^2(\C)$ introduced in Example  \ref{La-Croix-de-Berny-grados-varios:ej} above.  There, we proved that $C$ is a constructible set such that $\deg_{\rm lci}(C)=3$ and, at the same time, it is the projection of an algebraic variety $W'\subseteq \A^3(\C)$ of degree $2$. Hence, we have that the following holds and proves Claim $iv)$:
$$3= \deg_{\rm lci}( \pi(W')) \not\leq \deg_{\rm lci}(W')= \deg (W') =2.$$
\end{proof}

On the other hand, $\deg_\pi$ is well-behaved for linear images of constructible sets. This is resumed in the following statement:

\begin{proposition}\label{pi-degree-propiedades:prop}
Let $C\subseteq \A^n$ be a constructible subset. For every $m\geq n$, let us consider the class
$\scrL_m(C)$ of all pairs $(V,\ell)$, where $V\subseteq \A^m$ is a locally closed set, $\ell:\A^m\longrightarrow \A^n$ is a linear mapping and $\ell(V)=C$. Finally, let us consider the class:
$$\scrL(C):=\bigcup_{m\geq n}\scrL_m(C).$$
Then, we have:
\begin{enumerate}
\item The $\pi-$degree satisfies:
$$\deg_\pi(C)= \min\{ \deg_z(V)\; :\; \exists \ell:\A^m\longrightarrow \A^n, \; (V,\ell)\in \scrL(C)\}.$$
\item For every linear function $\lambda:\A^n \longrightarrow \A^r$, we have:
$$\deg_z(\lambda(C))\leq \deg_\pi(\lambda(C)) \leq \deg_\pi(C) \leq \deg_{\rm lci} (C).$$
\end{enumerate}
\end{proposition}
\begin{proof}
\leavevmode

\begin{itemize}
\item \emph{Proof of Claim $i)$:} Let $(V,\ell)\in \scrL(C)$ be one of such pairs. We may consider $G\subseteq \A^{m+n}$ the graph of the restriction $\restr{\ell}{V}$. We observe that $G$ is also a locally closed set since:
    $$G:=\{ (z, y)\in \A^m\times \A^n\; :\; z\in V, \; \ell(z)-y=0\}.$$
    Clearly, $G$ is the intersection of $V\times \A^n$ with the linear affine variety:
    $$A:=\A^m\times \{ (z,y)\in \A^{m+n}\; : \; \ell(z)-y=0\}.$$
 Hence, because of the B\'ezout's Inequality for locally closed sets (Theorem \ref{Bezout-localmente-cerrados:teor} above) we know that:
 $$\deg_z(G)=\deg(G)\leq \deg(V\times \A^n) \deg(A)= \deg(V)\cdot 1.$$
 Moreover, let $\pi: \A^{m+n}\longrightarrow \A^n$ be the canonical projection given by:
$$\pi(z,y):= y, \; \forall (z,y)\in \A^{m+n}.$$
We obviously have that $C:=\pi(G)$ and, hence, we conclude:
$$\deg_\pi(C) \leq \deg_z(V), \; \forall (V,\ell)\in \scrL(C).$$
The other inequality is obvious since projections are particular instances of linear mappings. Therefore, the first Claim is proved. \\
\item \emph{Proof of Claim $ii)$:} If $\lambda:\A^n \longrightarrow \A^r$ is a linear mapping, then we have:
    $$\scrL(\lambda(C))\supseteq \{ (V, \lambda\circ \ell)\; : \; (V,\ell)\in \scrL(C)\}.$$
    Hence, as:
    $$\deg_\pi(\lambda(C))=\min\{ \deg_{z}(V)\; : \; \exists \rho, \; (V,\rho)\in \scrL(\lambda(C))\},$$
    we immediately conclude:
    $$\deg_\pi(\lambda(C)) \leq\min\{\deg_{z}(V)\; :\; \exists \ell, \; (V,\ell)\in \scrL(C), \; (V, \lambda\circ \ell)\in \scrL(\lambda(C))\} = \deg_\pi(C),$$
    and Claim $ii)$ follows.
\end{itemize}
\end{proof}

In what respect to other properties, both the $LCI$-degree and the $\pi-$degree behave as expected for the natural operations between constructible sets:

\begin{proposition}\label{varias-propiedades-basicas-grado-construtibles:prop} Let $C\subseteq \A^n$ and $D\subseteq \A^m$ be two constructible sets.
\begin{enumerate}
\item For every linear affine variety $A\subseteq \A^n$ the following two inequalities hold:
$$\deg_{\rm lci}(A\cap C)\leq \deg_{\rm lci}(C), \;\;  \deg_\pi(A\cap C)\leq \deg_\pi(C).$$
\item The degrees of the Cartesian product satisfy the following inequalities:
$$\deg_{\rm lci}(C\times D) \leq \deg_{\rm lci}(C)\deg_{\rm lci}(D), \; \; \deg_\pi(C\times D) \leq \deg_\pi(C)\deg_\pi(D).$$
\item For every Zariski open subset $U\subseteq \A^n$, we have:
$$\deg_{\rm lci}(C\cap U) \leq \deg_{\rm lci}(C), \;\; \deg_\pi(C\cap U) \leq \deg_\pi(C).$$
\end{enumerate}
\end{proposition}
\begin{proof} We prove every inequality first for $\deg_{\rm lci}$ and then for $\deg_\pi$:
\begin{itemize}
\item \emph{Inequalities for $\deg_{\rm lci}$:} Claims $i)$ and $ii)$ easily follow from the sub-additivity of $\deg_{\rm lci}$  and the corresponding properties for locally closed sets. As for Claim $iii)$, suppose $C$ is decomposed as a finite union of locally closed irreducible sets:
$$C=W_1\cup \cdots \cup W_s,$$
such that $\deg_{\rm lci}(C)= \sum_{i=1}^s \deg(W_i)$. As $U$ is Zariski open, then for every $i$, $1\leq i \leq s$, we have just two options: Either $U\cap W_i=\emptyset$ or $U\cap W_i\not=\emptyset$, in which case $U\cap W_i$ is locally closed and Zariski dense in $\overline{W_i}^z$. Hence, we would have:
$$\deg_{\rm lci}(C\cap U) \leq \sum_{i=1}^s \deg(U\cap W_i) = \sum_{i=1}^s \deg(W_i)= \deg_{\rm lci}(C).$$
\item \emph{Inequalities for $\deg_\pi$:} For Claim $i)$, assume there is a locally closed subset $V\subseteq \A^m$ such that $\pi(V) := C$ and $\deg_\pi(C) = \deg_{z}(V)$, where $\pi:\A^m \longrightarrow \A^n$ is the canonical projection. We thus consider the linear affine subvariety $B:=\pi^{-1}(A)\subseteq \A^m$. We have that $\pi(V\cap B)= C\cap A$ and, hence,
    $$\deg_\pi(C\cap A) \leq \deg_{z}(V\cap B) \leq \deg_{z}(V).$$
    For Claim $ii)$, assume that
    there exist two locally closed sets $V\subseteq \A^{m_1}$ and $W\subseteq \A^{m_2}$ such that $m_{1} \geq n$ and $m_{2} \geq m$, and  $\pi_1(V): =C$, $\pi_2(W): =D$, where $\pi_1 : \A^{m_1} \longrightarrow \A^n$ and $\pi_2 : \A^{m_2} \longrightarrow \A^m$ are the canonical projections, $\deg_\pi(C)=\deg_{z}(V)$ and $\deg_\pi(D)=\deg_{z}(W)$. Therefore, we consider the affine space $\A^{m_1+m_2}:=\A^{m_1}\times \A^{m_2}$ and the product projection:
    $$\begin{matrix}
    \overline{\pi}:=\pi_1\times \pi_2: &\A^{m_1}\times\A^{m_2}& \longrightarrow & \A^{n}\times \A^m\\
    & (z_1, z_2) & \longmapsto & (\pi_1(z_1), \pi_2(z_2)).
    \end{matrix}$$
    Then, $\overline{\pi}(V\times W)= C\times D$ and, as $V\times W$ is locally closed, we conclude:
    $$\deg_\pi(C\times D) \leq \deg(V\times W) = \deg_{z}(V)\deg_{z}(W) = \deg_{\pi} (C) \deg_{\pi} (D).$$
    Finally, for Claim $iii)$ we have that there exists $V\subseteq \A^m$ locally closed such that $\pi(V) := C$ and $\deg(V)=\deg_\pi(C)$, where $\pi:\A^m\longrightarrow \A^n$ is the canonical projection.
	Taking the open subset $\widetilde{U}:=\pi^{-1}(U)\subseteq \A^m$, we have that
    $C\cap U= \pi(V\cap \widetilde{U})$ and  the following inequalities hold because $V$ is locally closed:
    $$\deg_\pi(C\cap U) \leq \deg_{z}(V\cap \widetilde{U}) \leq \deg_{z}(V)=\deg_\pi(C).$$
\end{itemize}
\end{proof}

Finally, both $\deg_{\rm lci}$ and $\deg_\pi$ satisfy a B\'ezout's Inequality.

\begin{theorem}[{\bf B\'ezout's Inequalities for constructible sets}]\label{Desigualdad-Bezout-constructibles:teor}
Let $C, D\subseteq \A^n$ be two constructible sets. We have:
$$\deg_{\rm lci}(C\cap D) \leq \deg_{\rm lci}(C)\deg_{\rm lci}(D), \;\; \deg_\pi(C\cap D) \leq \deg_\pi(C)\deg_\pi(D).$$
\end{theorem}
\begin{proof}
We first prove the inequality for $\deg_\pi$ and, then, for $\deg_{\rm lci}$. For $\deg_\pi$ the statement is almost obvious.  Let us first consider the Cartesian product $C\times D\subseteq \A^{2n}$. Next, let $\Delta^{(2n)}\subseteq \A^{2n}$  be the linear affine subvariety given by the following identity:
$$\Delta^{(2n)}:=\{(x_1,\ldots,x_n,y_1,\ldots,y_n)\in\A^{2n}\; :\; x_i-y_i=0,1\leq i\leq n\}.$$
Let $\pi:\A^{2n}\longrightarrow \A^n$ be the projection onto the first $n$ variables. We have:
$$\pi\left(\left( C\times D\right) \cap \Delta^{(2n)}\right) = C\cap D.$$
From Claim $ii)$ of Proposition \ref{pi-degree-propiedades:prop}, we have that:
$$\deg_\pi(C\cap D) \leq \deg_\pi\left(\left( C\times D\right) \cap \Delta^{(2n)}\right).$$
As $\Delta^{(2n)}$ is  a linear affine subvariety, we deduce from Claim $i)$ of Proposition \ref{varias-propiedades-basicas-grado-construtibles:prop} that:
$$\deg_\pi(C\cap D) \leq \deg_\pi\left( C\times D\right) .$$
And, finally, Claim $ii)$ of Proposition \ref{varias-propiedades-basicas-grado-construtibles:prop} yields the final inequality:
$$\deg_\pi(C\cap D) \leq \deg_\pi\left( C\times D\right)\leq \deg_\pi(C)\deg_\pi(D) .$$

In the case of $\deg_{\rm lci}$, the trouble is that degree is not preserved by linear transformations (as observed in Claim $iii)$ of Proposition \ref{constructibles-grado-imagenes-por-lineal:prop}). Therefore, we have to make some subtle changes into the proof scheme used for $\deg_\pi$. As above, we consider the diagonal subvariety $\Delta^{(2n)}\subseteq \A^{2n}$ and its intersection with the Cartesian product:
$$C':=(C\times D)\cap \Delta^{(2n)}.$$
By Proposition \ref{varias-propiedades-basicas-grado-construtibles:prop}, we conclude:
$$\deg_{\rm lci}(C')\leq\deg_{\rm lci}(C\times D)\leq\deg_{\rm lci}(C)\deg_{\rm lci}(D).$$
Then, we consider a minimum $LCI$-degree decomposition of $C'$ in locally closed subsets as in Definition \ref{grado-constructibles:def}. Namely, a decomposition:
$$C'=W_1\cup\ldots\cup W_s$$
where $W_i=U_i\cap V_i\not=\emptyset$, $U_i\subseteq \Delta^{(2n)}$ is a Zariski open set in $\Delta^{(2n)}$, $V_i\subseteq\Delta^{(2n)}$  is an irreducible subvariety for the Zariski topology and such that:
$$\deg_{\rm lci}(C')=\sum_{i=1}^s\deg(W_i)=\sum_{i=1}^s \deg(V_i).$$
For every $i$, $1\leq i \leq s$, we consider:
\begin{itemize}
\item $O_i\subseteq\A^n$ given by:
 $$O_i:=\{\underline{x}\in\A^n \; : \; (\underline{x},\underline{x})\in U_i\},$$
\item $Q_i\subseteq\A^n$ given by:
 $$Q_i:=\{\underline{x}\in\A^n\; : \; (\underline{x},\underline{x})\in V_i\}.$$
\end{itemize}

We immediately see that if  $\pi: \A^{2n}\longrightarrow \A^n$ is the projection onto the first $n$ coordinates, then  $\restr{\pi}{\Delta^{(2n)}}: \Delta^{(2n)} \longrightarrow \A^n$ is a biregular isomorphism. Hence, we conclude that $O_i$ is a Zariski open set in $\A^n$,  $Q_i$ is a Zariski closed subset of $\A^n$ and we have:
$$\pi(W_i)=O_i\cap Q_i.$$
As $\pi(W_i)$ is locally closed, by Proposition \ref{grado-preservado-por-transformaciones-lineales:prop} we conclude:
$$\deg(\pi(W_i))=\deg(\overline{\pi(W_i)}^z)= \deg(Q_i)=\deg(O_i\cap Q_i)\leq \deg(W_i).$$
Additionally, we have:
$$C\cap D= \pi( (C\times D)\cap \Delta^{(2n)})= (O_1\cap Q_1)\cup\ldots\cup(O_s\cap Q_s).$$
Finally, because of Proposition \ref{grado-constructibles-subaditivo:prop} we conclude:
$$\deg_{\rm lci}(C\cap D)\leq\sum_{i=1}^s\deg(O_i\cap Q_i)\leq\sum_{i=1}^s\deg(W_i)=\deg_{\rm lci}(C')\leq\deg_{\rm lci}(C)\deg_{\rm lci}(D).$$
\end{proof}

\section{Upper Bounds for the intersection of several constructible sets}\label{upper-intersection-several:sec}

In this Section we exhibit a generalization of Proposition 2.3 in \cite{HeintzSchnorr} for constructible sets, as promised at the Introduction. Observe that a literal application of B\'ezout's Inequality puts the co-dimension at the exponent of bounds for the degree. Namely, given $f_1,\ldots, f_m\in K[X_1,\ldots, X_n]$ such that $V:=V_\A(f_1,\ldots, f_m)$ has dimension $n-m$, and let $W\subseteq \A^n$ be an algebraic variety. B\'ezout's Inequality yields the following upper bound:
$$\deg(W\cap V_\A(f_1,\ldots, f_m))\leq \deg(W)\prod_{i=1}^m \deg(f_i) \leq \deg(W)\left( \max\{\deg(f_1), \ldots, \deg(f_m)\}\right)^{\rm codim(V)}.$$
Nevertheless, Proposition 2.3 in  \cite{HeintzSchnorr} replaces co-dimension of $V$ by dimension of $W$, emphasizing the role of dimension of $W$ in this kind of upper bounds. Namely, Proposition 2.3 in \cite{HeintzSchnorr} yields:
$$\deg(W\cap V_\A(f_1,\ldots, f_m))\leq  \deg(W)\left( \max\{\deg(f_1), \ldots, \deg(f_m)\}\right)^{\rm dim(W)}.$$
The trade-off between these two upper bounds is the key ingredient to prove the existence of short correct test sequences (see Theorem \ref{teorema-principal-densidad-questores:teor} in Section \ref{probability-CTS-zero-dimensional:sec}).
\subsection{Main statement of this Section}

Proposition 2.3 in \cite{HeintzSchnorr} provides an elegant upper bound for the intersection of several algebraic varieties.  We want to extend it for constructible sets but, obviously, it cannot be possible if we use $\deg_z$ by the same reason why B\'ezout's Inequality is not true for $\deg_z$. Hence, we first extend that result, using similar inductive arguments as in the original proof, for the intersection of a constructible set with several locally closed sets for both $\deg_{\rm lci}$ and $\deg_{\pi}$. Then, we try to exhibit a similar upper bound for the general case for $\deg_{\rm lci}$. We have not been able to stablish a similar result for $\deg_{\pi}$.

\begin{proposition}[{\bf Extension of Proposition 2.3 of \cite{HeintzSchnorr}}] \label{grado-constructible-cerrados-multiple:prop} Let $\Omega\subseteq \A^n$ be a constructible set and let $C_1,\ldots, C_s\subseteq \A^n$ be a finite sequence of locally closed sets. Then, we have:
$$\deg_{\rm lci} \left(\Omega \cap\left(\bigcap_{i=1}^s C_i\right)\right)\leq \deg_{\rm lci}(\Omega) \left(\max\{ \deg_{\rm lci}(C_i)\; : \; 1\leq i \leq s\}\right)^{\dim(\Omega)}.$$
\end{proposition}
\begin{proof} The argument goes by induction on $s$ as in \cite{HeintzSchnorr}. The case $s=2$ is just the B\'{e}zout's Inequality for constructible sets, then we assume $s\geq 3$. First of all, suppose that $\Omega$ is a locally closed irreducible set (i.e. a non-empty open set in some irreducible algebraic variety). Then, we decompose the intersection as:
$$\Omega \cap\left(\bigcap_{i=1}^s C_i\right)= \left( \Omega \cap C_1\right) \cap \left( \bigcap_{i=2}^s C_i\right).$$
We observe that $\Omega\cap C_1$ is locally closed and we may consider two options:
\begin{itemize}
\item Either $\dim(\Omega\cap C_1)= \dim(\Omega)$. In this case, as $\Omega$ is locally closed irreducible, then $\Omega\cap C_1$ is a locally closed set dense in $\overline{\Omega}^z$ and, we have:
    $$\deg(\Omega\cap C_1) = \deg(\overline{(\Omega\cap C_1)}^z) \deg(\overline{\Omega}^z)=\deg(\Omega),$$
    where the last equality is true since $\Omega$ is locally closed. Applying the inductive hypothesis, we conclude:
    $$\deg\left( \Omega \cap \left(\bigcap_{j=1}^s C_j\right)\right)\leq \deg(\Omega\cap C_1) \left(\max\{ \deg(C_i) \; : \; 2\leq i \leq s\}\right)^{\dim(\Omega)},$$
    which implies the statement.
\item Or $\dim(\Omega\cap C_1)< \dim(\Omega)$. In this case, just using the inductive hypothesis we have:
    $$\deg\left( \Omega \cap \left(\bigcap_{j=1}^s C_j \right)\right)\leq
\deg(\Omega\cap C_1) \left(\max\{ \deg(C_i) \; : \; 2\leq i \leq s\}\right)^{\dim(\Omega \cap C_1)}.$$
Using the B\'ezout's Inequality for constructible sets we stated in previous sections, we obtain:
    $$\deg\left( \Omega \cap \left(\bigcap_{j=1}^s C_j\right)\right)\leq
\deg(\Omega)\deg(C_1) \left(\max\{ \deg(C_i) \; : \; 2\leq i \leq s\}\right)^{\dim(\Omega)-1},$$
and this also yields the wanted inequality.
\end{itemize}

As for the general case, let us consider a minimum $LCI$-degree decomposition of  $\Omega$ into locally closed irreducible subsets. Namely, we consider a decomposition:
$$\Omega:=(U_1\cap V_1) \cup \cdots \cup (U_t\cap V_t),$$
where $U_i\cap V_i\not=\emptyset$, $U_i\subseteq \A^n$ is a Zariski open set, $V_i\subseteq \A^n$ is an irreducible closed subset for the Zariski topology and such that the following holds:
$$\deg_{\rm lci}(\Omega)=\sum_{i=1}^t \deg(V_i).$$
As the degree is sub-additive we have:
$$\deg_{\rm lci} \left( \Omega \cap \left(\bigcap_{j=1}^s C_j\right)\right)\leq\sum_{i=1}^t \deg\left((U_i\cap V_i)\cap \left(\bigcap_{j=1}^s C_j\right)\right).$$
As $(U_i\cap V_i)$ is locally closed irreducible, the previous arguments apply and we conclude:

$$\deg_{\rm lci} \left( \Omega \cap \left(\bigcap_{j=1}^s C_j\right)\right)\leq\sum_{i=1}^t \deg(V_i)\left(\max\{ \deg(C_j) \; \; 1\leq j \leq s\}\right)^{\dim(U_i\cap V_i)}.$$
As $\dim(\Omega)= \max\{ \dim(U_i\cap V_i) \; :\; 1\leq i \leq t\}$, we conclude the proof of the statement.
\end{proof}

For the $\pi$-degree, we have a similar upper bound for globally equi-dimensional constructible sets:

\begin{corollary}\label{grado-interseccion-equidimensional-lc:corol}
 Let $\Omega\subseteq \A^n$ be a globally equi-dimensional constructible set and let $C_1,\ldots, C_s\subseteq \A^n$ be a finite sequence of locally closed sets. Then, we have:
$$ \deg_{\pi} \left(\Omega \cap \ \left( \bigcap_{i=2}^s C_i\right) \right) \leq \deg_{\pi}(\Omega) \left( \max\{ \deg_{\pi}(C_i)\; :\; 2\leq i \leq s\}\right)^{\dim(\Omega)}$$
\end{corollary}
\begin{proof}
Essentially the same proof of the previous Proposition, just note that since $\Omega$ is globally equi-dimensional it admits a decomposition into locally closed irreducible sets $\Omega = W_1\cup \cdots \cup W_s$ such that the following property holds:
$$\deg_{\pi} (\Omega) = \deg_{\rm lci} (\Omega) = \sum_{i=1}^{s} \deg(W_{i}).$$
\end{proof}

Proposition \ref{grado-constructible-cerrados-multiple:prop} also applies to have estimates on the degree of images of constructible sets under polynomial mappings as shown in the following Corollary.

\begin{corollary}\label{grado-imagen-constructible:corol}
Let $\varphi:=(\varphi_1,\ldots, \varphi_m): \A^n(K)\longrightarrow \A^m(K)$ be a polynomial mapping, where $K$ is an algebraically closed field. Assume that for every $i$, $1\leq i\leq m$, $\varphi_i\in K[X_1,\ldots, X_n]$ is a polynomial of degree at most $d$, where $d\geq 1$. Let $C\subseteq \A^n(K)$ be a constructible subset. Then, we have:
$$\deg_z\left( \varphi(C)\right)\leq  \deg_{\rm lci}(C) d^{\dim(C)},$$
$$\deg_z\left( \varphi(C)\right)\leq \deg_{\pi}(C) d^{m} \leq  \deg_{\rm lci}(C) d^{m}.$$
The bounds are not always true if we replace $\deg_{\rm lci}$ by $\deg_z$ on the right hand side of the inequality. Moreover, if either $\varphi(C)$ is locally closed or globally equi-dimensional, we also have:
$$\deg_z\left( \varphi(C)\right)=\deg_\pi(\varphi(C)) = \deg_{\rm lci}(\varphi(C)) \leq \deg_{\rm lci}(C) d^{\dim(C)},$$
$$\deg_z\left( \varphi(C)\right)=\deg_\pi(\varphi(C)) = \deg_{\rm lci}(\varphi(C))\leq \deg_{\pi}(C) d^{m} \leq  \deg_{\rm lci}(C) d^{m}.$$
\end{corollary}
\begin{proof}
\leavevmode

\begin{itemize}
\item Inequality $\deg_z\left( \varphi(C)\right)\leq  \deg_{\rm lci}(C) d^{\dim(C)}$: Let us first assume that  $C$ is locally closed and irreducible. Then, the Zariski closure $W:=\overline{\varphi(C)}^z\subseteq \A^m(K)$ of $\varphi(C)$  is an irreducible variety. Let $D:=\deg(W)$ be the degree of $W$ and let  $s:=\dim(W)\leq \dim(V)$. Then, there is a linear affine subvariety  $A\subseteq \A^m(K)$ of co-dimension $s$ such that:
$$\sharp\left(A \cap W\right) = \sharp\left(A \cap \varphi(C)\right) = \deg(W)= D.$$
Let us now consider the algebraic variety $\varphi^{-1}(A)\subseteq \A^n(K)$. As $A$ is given by $s$ linear equations, then $\varphi^{-1}(A)$ is also given by $s$ polynomial equations of degree at most $d$ (combining the coordinates $\varphi_1, \ldots, \varphi_s$ of $\varphi$ with the equations defining $A$ in $\A^m$). From Proposition \ref{grado-constructible-cerrados-multiple:prop}, we conclude:
$$\deg_{\rm lci}\left(C \cap \varphi^{-1}(A)\right)\leq \deg_{\rm lci}(C) d^{\dim(C)}.$$
Let $\scrC$ be the class of locally closed irreducible components of $C\cap \varphi^{-1}(A)$. As the number of irreducible components of a locally closed set is bounded by its degree, we conclude:
$$\sharp(\scrC)\leq \deg_{\rm lci}\left(C \cap \varphi^{-1}(A)\right)\leq \deg_{\rm lci}(C) d^{\dim(C)}.$$
However, let $V\in \scrC$  be one of such irreducible components of $C\cap \varphi^{-1}(A)$. We know that its Zariski closure  $\overline{\varphi(V)}^z$ is irreducible and, at the same time, we have:
$$\varphi(V)\subseteq A\cap \varphi(C).$$
As $A\cap \varphi(C)$ is a finite number of points,  $\varphi(V)$ has to be a point in  $\A^m$. Additionally, as
$\varphi(C\cap \varphi^{-1}(A))= \varphi(C)\cap A$, there is an onto mapping between the following two finite sets:
$$\begin{matrix}
\Phi: & \scrC & \longrightarrow & \varphi(C)\cap A\\
& V& \longmapsto & \Phi(V) := \varphi(V).\end{matrix}$$
Then, we conclude that if  $C$ is a locally closed irreducible set we have:
$$\deg(\overline{\varphi(C)}^z)=\deg(W)= \sharp\left( \varphi(C)\cap A\right) \leq \sharp\left(\scrC\right) \leq \deg(C) d^{\dim(C)}.$$
For any constructible set $C\subseteq \A^n$, we just have to consider a minimum degree decomposition of $C$ as in Definition \ref{grado-constructibles:def}, i.e. a decomposition of $C$ as finite union of  locally closed irreducible sets satisfying Lemma \ref{descomposicion-union-irreducibles-distintos:lema}:
$$
C:=(U_1\cap V_1) \cup \cdots \cup (U_s\cap V_s),$$
where $U_i\cap V_i\not=\emptyset$, $U_i\subseteq \A^n$ is open, $V_i\subseteq \A^n$ is closed irreducible in the Zariski topology and they satisfy:
\begin{equation}\label{descomposicion-de-constructible-para-cota-grado-imagen:eqn}
\deg(C)=\sum_{i=1}^s \deg(V_i).
\end{equation}
Then, we have:
$$\overline{\varphi(C)}^z= \overline{\left(\bigcup_{i=1}^s \varphi\left( U_i\cap V_i\right) \right)}^z=
\bigcup_{i=1}^s \overline{\varphi\left( U_i\cap V_i\right)}^z .$$
As $\deg$ is sub-additive when applied to algebraic varieties, we would have:
$$\deg\left(\overline{\varphi(C)}^z\right)\leq \sum_{i=1}^s \deg\left(\overline{\varphi\left( U_i\cap V_i\right)}^z\right) \leq  $$
$$\leq \sum_{i=1}^s \deg_{\rm lci}(V_i\cap U_i) d^{\dim(V_i)}= \sum_{i=1}^s \deg(V_i) d^{\dim(V_i)}.$$
As $\dim(C):=\max\{\dim(V_1), \ldots, \dim(V_s)\}$, from Identity (\ref{descomposicion-de-constructible-para-cota-grado-imagen:eqn}) we conclude:
$$\deg_z\left(\varphi(C)\right) \leq\left( \sum_{i=1}^s \deg(V_i)\right) d^{\dim(C)}=
\deg_{\rm lci}(C)d^{\dim(C)} .$$
\item Inequality $\deg_z\left( \varphi(C)\right)\leq \deg_{\pi}(C) d^{m} \leq  \deg_{\rm lci}(C) d^{m}$: First, we consider the graph of the polynomial map $\varphi$:
$$ Gr(\varphi) :=   \{ (x,y) \in C \times \A^m \: : \: y_{i} - \varphi_i (x) = 0, \: i=1, \ldots, m \} = $$
$$  = (C \times \A^{m}) \cap (\bigcap_{i=1}^{m} \{ Y_{i} - \varphi_{i}(X) = 0 \}). $$
Then, let $\Phi$ be the following polynomial map:
$$\begin{matrix}
\Phi: & C & \longrightarrow & Gr(\varphi)\\
& x& \longmapsto & (x, \varphi(x)).\end{matrix}$$
Obviously we have $\Phi(C) = Gr(\varphi)$, therefore $Gr(\varphi)$ is a constructible set. Let $\pi: \A^{n} \times \A^{m} \rightarrow \A^{m}$ be the canonical projection that forgets the first $n$ variables. Thus, we have that $\varphi(C) = \pi(Gr(\varphi))$ and $\deg_{z}(\varphi(C))= \deg_{z}(\pi(Gr(\varphi)))$. From Proposition \ref{pi-degree-propiedades:prop}, we conclude:
$$\deg_{z}(\pi(Gr(\varphi))) \leq \deg_{\pi} (\pi(Gr(\varphi))) \leq \deg_{\pi} (Gr(\varphi)).$$

Applying B\'ezout's inequality (for the $\pi$-degree), we deduce:
$$\deg_{\pi}(Gr(\varphi)) = \deg_{\pi} \left((C \times \A^{m}) \cap (\bigcap_{i=1}^{m} \lbrace Y_{i} - \varphi_{i}(X) = 0 \rbrace ) \right)   \leq $$
$$ \leq \deg_{\pi} (C \times \A^{m})\prod_{i=1}^{m} \deg_{\pi} (\lbrace Y_{i} - \varphi_{i}(X) = 0 \rbrace).$$
As $\deg_{\pi} (C \times \A^{m}) \leq \deg_{\pi} (C) \deg_{\pi} (\A^{m}) = \deg_{\pi}(C)$ and for all $i$, $1 \leq i \leq m$, we have that $\deg_{\pi} ( \lbrace Y_{i} - \varphi_{i} (X) = 0 \rbrace) = \deg(Y_{i} - \varphi_{i} (X)) = \max \lbrace deg(\varphi_{i}), 1 \rbrace \leq d$, we obtain:
$$deg_{\pi} (C \times \A^{m})\prod_{i=1}^{m} \deg_{\pi} (\lbrace Y_{i} - \varphi_{i}(X) = 0 \rbrace) \leq \deg_{\pi} (C) d^{m}.$$
Finally, from Proposition \ref{relations-between-three-degrees:prop} we conclude:
$$\deg_z\left( \varphi(C)\right)\leq \deg_{\pi}(C) d^{m} \leq  \deg_{\rm lci}(C) d^{m}.$$

\end{itemize}
As for the counterexample, we just need to recall the one exhibited in Claim $iii)$ of Proposition \ref{constructibles-grado-imagenes-por-lineal:prop}. The rest of the claims are obvious in the case $\varphi(C)$ is either locally closed or globally equi-dimensional.
\end{proof}

The remaining pages of this section are devoted to prove its main outcome:

\begin{theorem}[{\bf Bounds for the $LCI$-degree of the intersection of several constructible sets}]\label{cotas-grado-varias-intersecciones:teor}
Let $C_1, \ldots, C_s\subseteq \A^n$ be a sequence of constructible sets. Let $r:=\dim(C_1)$ be the dimension of the constructible set $C_1$. The following inequalities hold:

\begin{enumerate}
\item  \emph{First upper bound:}
$$\deg_\pi\left(\bigcap_{i=1}^s C_i\right)\leq \deg_{\rm lci} \left( \bigcap_{i=1}^s C_i\right)\leq {{s+r-1}\choose {r}}\deg_{\rm lci}(C_1) \left( \max\{\deg_{\rm lci}(C_j)\; :\; 2\leq j \leq s\}\right)^{r}.$$
\item \emph{Second upper bound:}
$$\deg_\pi\left(\bigcap_{i=1}^s C_i\right)\leq \deg_{\rm lci}\left(\bigcap_{i=1}^s C_i\right)\leq \deg_{\rm lci}(C_1) \left(1+ \sum_{i=2}^s \deg_{\rm lci}(C_i) \right)^{\dim(C_1)}.$$
\item \emph{Upper bound in terms of the average degree:} Given a family of constructible sets $C_1,\ldots,C_s\subseteq \A^n$, we define its average degree as:
    $$\deg_{\rm av}^{(E)}(C_1,\ldots, C_s):= {{1}\over{s}}\sum_{i=1}^s \deg_{\rm lci}(C_i).$$
    Then, we also have:
$$\deg_\pi\left(\bigcap_{i=1}^s C_i\right)\leq \deg_{\rm lci}\left(\bigcap_{i=1}^s C_i\right)\leq \deg_{\rm lci}(C_1) s^{\dim(C_1)}\left(\deg_{\rm av}^{(E)}(C_1,\ldots, C_s)\right)^{\dim(C_1)}.$$
\end{enumerate}

\end{theorem}

First upper bound in the previous statement generalizes the main bound of Proposition 2.3 in \cite{HeintzSchnorr}, but the constant coefficient  is exponential in $s-1$ because
$$\left( {{r+s-1}\over{s-1}}\right)^{s-1} \leq {{r+s-1}\choose{r}}\leq \left( {{e(r+s-1)}\over{s-1}}\right)^{s-1},$$
where $e$ is the basis of the natural logarithm. The reader may also think that the constant is also exponential in $r$ since
$$\left( {{r+s-1}\over{r}}\right)^{r} \leq {{r+s-1}\choose{r}}\leq \left( {{e(r+s-1)}\over{r}}\right)^{r}.$$
In any case it is not the constant 1 one would expected according to Proposition 2.3. of \cite{HeintzSchnorr}.
Third upper bound is simply obtained by rewriting second upper bound since $\deg_{\rm lci}(C_1)\geq 1$ always hold. Hence, second and third upper bounds are exponential in $r$, but polynomial in $s$ which may be helpful in some applications.
Note that the left hand side inequalities for the  $\pi$-degree are an immediate consequence of Proposition \ref{relations-between-three-degrees:prop}.

The rest of the Section is devoted to prove this Theorem: First we prove Claim $i)$ in Subsection \ref{cota-primera-grado-varias-intersecciones:subsec} and, then, we prove Claim $ii)$ in Subsection \ref{cota-segunda-grado-varias-intersecciones:subsec}. Finally, some immediate applications of these discussions are exhibited in Subsection \ref{cotas-intersecciones-multiples-varios:subsec}.

\subsection{Proof of the first bound of Theorem \ref{cotas-grado-varias-intersecciones:teor}}\label{cota-primera-grado-varias-intersecciones:subsec}
We begin with the case $s=2$. Then, B\'ezout's Inequality for constructible sets yields:
$$\deg_{\rm lci}(C_1\cap C_2) \leq \deg_{\rm lci}(C_1) \deg_{\rm lci}(C_2) \leq {{r+1}\choose{r}}\deg_{\rm lci}(C_1) \deg_{\rm lci}(C_2).$$
We thus apply induction and assume that $C_1=W$ is a locally closed irreducible subset. Then, we have:
$$W\cap C_2 \cap \left( \bigcap_{i=3}^s C_i\right)= W \cap \left( \bigcap_{i=2}^s C_i\right).$$
We distinguish two cases:
\begin{itemize}
\item Assume that $\dim(W\cap C_2) <\dim (W)$. Then, applying the inductive hypothesis we have:
\begin{equation*}
\begin{split}
 \deg_{\rm lci}\left(W \cap \left( \bigcap_{i=2}^s C_i\right)\right) & \leq {{\dim(W\cap C_2) + s-2}\choose {\dim(W\cap C_2)}} \\
& \deg_{\rm lci}(W\cap C_2)  \left( \max\{\deg_{\rm lci}(C_j)\; :\; 3\leq j \leq s \}\right)^{\dim(W\cap C_2)}.
\end{split}
\end{equation*}
Now, observe that if $a\leq b$ are two positive integers, we have:
$${{a+(s-2)}\choose{a}}\leq {{b+(s-2)}\choose {b}}\leq {{b+(s-1)}\choose{b}}.$$
As $\dim(W\cap C_2) \leq \dim(W)-1\leq \dim(W)$ we have:
\begin{equation*}
\begin{split}
 \deg_{\rm lci}\left(W \cap \left( \bigcap_{i=2}^s C_i\right)\right) & \leq {{\dim(W) + s-2}\choose {\dim(W)}}\\
&\deg_{\rm lci}(W\cap C_2)  \left( \max\{\deg_{\rm lci}(C_j)\; :\; 3\leq j \leq s \}\right)^{\dim(W)-1}.
\end{split}
\end{equation*}
Moreover, by B\'ezout's Inequality for constructible sets we would have:
\begin{equation*}
\begin{split}
\deg_{\rm lci}\left(W \cap \left( \bigcap_{i=2}^s C_i\right)\right) &  \leq {{\dim(W) + s-2}\choose {\dim(W)}}\\
&\deg(W) \deg_{\rm lci}(C_2)  \left( \max\{\deg_{\rm lci}(C_j)\; :\; 3\leq j \leq s \}\right)^{\dim(W)-1}.
\end{split}
\end{equation*}
Putting everything together we conclude the wanted inequality:
\begin{equation*}
\begin{split}
\deg_{\rm lci}\left(W \cap \left( \bigcap_{i=2}^s C_i\right)\right) & \leq {{\dim(W) + s-1}\choose {\dim(W)}}\\
&\deg(W) \left( \max\{\deg_{\rm lci}(C_j)\; :\; 2\leq j \leq s \}\right)^{\dim(W)}.
\end{split}
\end{equation*}

\item Otherwise, assume that $\dim(W\cap C_2)= \dim(W)$. Let us consider a decomposition of $C_2$ as finite union of locally closed irreducible subsets:
    $$C_2= W_1 \cup W_2 \cup \cdots \cup W_m,$$
    such that the degree of $C_2$ satisfies:
    $$\deg_{\rm lci}(C_2)=\sum_{i=1}^m \deg(W_i).$$
    Up to some reordering of the $W_i$'s, as $\dim(W\cap C_2)=\dim(W)$, there must exist some integer $r$, with $1\leq r \leq m$ such that the following holds:
    \begin{itemize}
    \item For all $i$, $1\leq i \leq r$, $\dim(W\cap W_i)=\dim(W)$.
    \item For all $i$, $ r+1 \leq i \leq s$, $\dim(W\cap W_i)< \dim(W)$.
    \end{itemize}
Next, we observe that the following set is a Zariski open subset in $W$ and, hence, a locally closed irreducible set:
$$W\cap \left(\bigcup_{i=1}^r W_i\right).$$
Recall that both $W$ and $W_i$ are locally closed irreducible. Then, there are Zariski open subsets $U_1,U_2\subseteq \A^n$ and irreducible algebraic varieties $V, V_i\subseteq \A^n$ such that:
$$ W=U_1\cap V, \; W_i=U_2\cap V_i.$$
As $\dim(W)=\dim(W\cap W_i)$, then $W\cap W_i\not=\emptyset$ and, hence,
$$W\cap W_i= U_1\cap U_2 \cap V \cap V_i.$$
In particular, $U_1\cap U_2\cap V$ is a non-empty Zariski open subset in $V$ of dimension equal to $\dim(V)$. Moreover, the dimension of the intersection $U_1\cap U_2\cap V \cap V_i$ is also $\dim(V)$. Thus,
$$\dim(W\cap W_i)= \dim (U_1\cap U_2\cap V \cap V_i) = \dim(U_1\cap U_2\cap V) = \dim(V).$$
The Zariski closure of $W\cap W_i$ is then a closed subvariety of $V$ of the same dimension than $V$. As $V$ is irreducible:
$$\overline{W\cap W_i}^z = V.$$
On the other hand, it is obvious that $\overline{W\cap W_i}^z\subseteq V_i$ and, hence, we conclude
$V\subseteq V_i$. Thus, we have:
$$W\cap W_i= U_i\cap V,$$
for some open subset $U_i\subseteq \A^n$. In conclusion, we have that:
$$W\cap \left(\bigcup_{i=1}^r W_i\right)= \left(\bigcup_{i=1}^r U_i\right) \cap V,$$
and this is a locally closed set.\\
Let us then define the following two constructible subsets:
$$C_2':=\bigcup_{i=1}^r W_i, \;\; \widetilde{C_2}:=\bigcup_{j=r+1}^s W_j.$$
As $C_2'\cap W$ is a Zariski open subset of the Zariski closure of $W$, we conclude:
$$\deg(W\cap C_2') = \deg(W).$$
On the other hand, we have that:
$$\deg_{\rm lci}\left(\widetilde{C_2}\right)=\deg_{\rm lci}(C_2)-\sum_{i=1}^r \deg(W_i)\leq \deg_{\rm lci}(C_2).$$
Hence, B\'ezout's Inequality for constructible sets yields:
$$\deg_{\rm lci}\left(W\cap \widetilde{C_2}\right)\leq\deg(W)(\deg_{\rm lci}(C_2)-\sum_{i=1}^r \deg(W_i))\leq \deg(W)\deg_{\rm lci}(C_2).$$
Next, let us consider:
$$W\cap \left(\bigcap_{i=2}^sC_i\right)= \left((W\cap C_2') \cap \left(\bigcap_{i=3}^s C_i\right)\right)
\cup \left((W\cap \widetilde{C_2}) \cap \left(\bigcap_{i=3}^s C_i\right)\right).$$
So, as $\deg_{\rm lci}$ is sub-additive for constructible sets, we conclude:
$$\deg_{\rm lci}\left(W\cap \left(\bigcap_{i=2}^sC_i\right)\right) \leq I_1 + I_2,$$
where
$$I_1:=\deg_{\rm lci}\left((W\cap C_2') \cap \left(\bigcap_{i=3}^s C_i\right)\right),$$
and
$$I_2:=\deg_{\rm lci}\left((W\cap \widetilde{C_2}) \cap \left(\bigcap_{i=3}^s C_i\right)\right).$$
We bound both quantities separately:
\begin{itemize}
\item Taking $r:=\dim(W)=\dim(W\cap C_2')$ and, knowing that $\deg(W)=\deg(W\cap C_2')$,
the application of the induction hypothesis yields:
    $$I_1\leq {{r+s-2}\choose{r}} \deg(W) \left( \max\{\deg_{\rm lci}(C_j)\; :\; 3\leq j \leq s\}\right)^{r}.$$
\item On the other hand, let us consider $t:=\dim(W\cap \widetilde{C_2})\leq r-1$, knowing that
$\deg_{\rm lci}(W\cap \widetilde{C_2}) \leq \deg(W) \deg_{\rm lci}(C_2)$,
the application of the induction hypothesis yields:
$$I_2 \leq {{t+s-2}\choose{t}} \deg(W) \deg_{\rm lci}(C_2) \left( \max\{\deg_{\rm lci}(C_j)\; :\; 3\leq j \leq s\}\right)^{t}.$$
As $t\leq r-1$, we also have:
$${{t+s-2}\choose{t}} \leq {{r-1+s-2}\choose{r-1}}\leq {{r-1+ s-1}\choose{r-1}},$$
and, hence, we conclude:
$$I_2 \leq {{r+s-2}\choose{r-1}} \deg(W) \deg_{\rm lci}(C_2) \left( \max\{\deg_{\rm lci}(C_j)\; :\; 3\leq j \leq s\}\right)^{r-1}.$$

\end{itemize}
Putting the two previous inequalities in a single one, we have:
\begin{equation*}
\begin{split}
I_1+I_2 \leq &  \deg(W)\left( \max\{\deg(C_j)\; :\; 3\leq j \leq s\}\right)^{r-1} \\
& \left({{r+s-2}\choose{r}} \max\{\deg_{\rm lci}(C_j) \; :\; 3\leq j \leq s\}+{{r+s-2}\choose{r-1}} \deg_{\rm lci}(C_2)\right).
\end{split}
\end{equation*}

Thus,

\begin{equation*}
\begin{split}
I_1+I_2 \leq & \deg(W)\left( \max\{\deg_{\rm lci}(C_j)\; :\; 3\leq j \leq s\}\right)^{r-1}  \\
& \max\{\deg_{\rm lci}(C_i)\; :\; 2\leq i \leq s \}  \left({{r+s-2}\choose{r}} +{{r+s-2}\choose{r-1}} \right).
\end{split}
\end{equation*}

Using Pascal's Triangle equality we obtain:
$$I_1+I_2 \leq \deg(W)\left( \max\{\deg_{\rm lci}(C_j)\; :\; 2\leq j \leq s\}\right)^{r}
{{r+s-1}\choose{r}}, $$
as wanted.
\end{itemize}

In order to conclude the general inequality, assume that $C_1$ admits a decomposition into locally closed irreducible varieties:
$$C_1:=V_1\cup V_2 \cup \cdots \cup V_t,$$
such that
\begin{equation}\label{grado-componentes-interseccion-multiple:eqn}
\deg_{\rm lci}(C_1)= \sum_{j=1}^t \deg(V_j).
\end{equation}
Since $\deg_{\rm lci}$ is sub-additive, we obtain:
$$\deg_{\rm lci}\left( C_1 \cap \left( \bigcap_{i=2}^s C_i \right)\right)
\leq \sum_{j=1}^t \deg_{\rm lci}\left( V_j  \cap \left( \bigcap_{i=2}^s C_i \right)\right).$$
Applying the previous discussion, taking $r_j=\dim(V_j)\leq \dim(V)$ we conclude:
$$\deg_{\rm lci}\left( C_1 \cap \left( \bigcap_{i=2}^s C_i \right)\right)
\leq \sum_{j=1}^t {{r_j+s-1}\choose {r_j}} \deg(V_j) \left(\max\{ \deg_{\rm lci}(C_i)\; : \; 2\leq i \leq s\}\right)^{r_j}.$$
As $r_j\leq r$ for each $j$, $1\leq j \leq t$, we have:
$${{r_j+s-1}\choose {r_j}}\leq {{r+s-1}\choose {r}},$$
and, hence, we deduce:
$$\deg_{\rm lci}\left( C_1 \cap \left( \bigcap_{i=2}^s C_i \right)\right)
\leq {{r+s-1}\choose {r}} \left(\sum_{j=1}^t \deg(V_j)\right) \left(\max\{ \deg_{\rm lci}(C_i)\; : \; 2\leq i \leq s\}\right)^{r}.$$
Using Identity (\ref{grado-componentes-interseccion-multiple:eqn}) we conclude the statement.

\subsection{Proof of the second bound of Theorem \ref{cotas-grado-varias-intersecciones:teor}}\label{cota-segunda-grado-varias-intersecciones:subsec}

Here, we adapt some of the proof strategies of Section 3 of \cite{Heintz83}, to produce a proof for the degree of the intersection of several constructible sets. Before going into the proof, we need of some preparatory results.

\subsubsection{Preparatory results}

\begin{definition} Let $C\subseteq \A^n$ be a constructible set and $V\subseteq \A^n$ an irreducible variety. We say that $V$ is an irreducible component of $C$ with respect to the $LCI$-degree of $C$ if there is some minimum $LCI$-degree decomposition of $C$ into locally closed irreducible sets:
$$C:=W_1\cup \cdots \cup W_s,$$
such that there exists $i$ with $V=\overline{W_i}^z$.
\end{definition}

Observe that if $C$ is locally closed or globally equi-dimensional, then the class $\scrD(C)$ of all irreducible components of $C$ with respect to the degree of $C$ is a finite set and it is completely determined by $C$. This is not true for general constructible sets  as we have shown in Remark \ref{no-unicidad-de-la-descomposicion-minimal:rk}.

\begin{lemma}\label{control-celdas:lema} Let $D\subseteq \A^n$ be a locally closed subset.  Let $\scrV_1,\ldots, \scrV_s$ be several finite families of locally closed sets. Let $\scrW$ be the set of all
locally closed subsets of $\A^n$ that may be defined as intersections with $D$ of some locally closed sets chosen according to the list $\scrV_1,\ldots, \scrV_s$. Namely, let $\scrW$ be the set of locally closed sets given by the following identity:
$$\scrW:=\{D\cap\left(\bigcap_{i\in S} V_{i}\right)\; : \; S\subseteq\{1,\ldots, s\}, \; V_{i}\in \scrV_i\}$$
Let $\scrC$ be the set of algebraic varieties defined by the following identity:
$$\scrC:=\{ V \; : \;  {\hbox {\rm $V$ irreducible variety and}}\;  \exists W\in \scrW, {\hbox {\rm a non-empty open subset of $V$  is a component of $W$}}\}.$$
Then, we have:
\begin{equation}\label{cota-superior-grado-celdas:eqn}
\sum_{V \in \scrC}\deg(V) \leq  \deg(D) \left(1+\sum_{V\in \bigcup_{i=1}^s \scrV_i} \deg(V) \right)^{\dim(D)}\leq \deg(D)\left(1+ \sum_{i=1}^s \sum_{V\in \scrV_i} \deg(V) \right)^{\dim(D)}.
\end{equation}
Moreover, there is a degree preserving bijection between $\scrC$ and the following set:
\begin{equation}\label{me-quedo-con-componentes:eqn}
\scrD:=\{ C \; : \;  \exists W\in \scrW, {\hbox {\rm  $C$ is an irreducible component $\overline{W}^z$}}\}.
\end{equation}
and the same upper bound holds:
\begin{equation}\label{cota-superior-grado-celdas:eqn}
\sum_{C\in\scrD} \deg(C)  \leq  \deg(D) \left(1+\sum_{V\in \bigcup_{i=1}^s \scrV_i} \deg(V) \right)^{\dim(D)}\leq \deg(D)\left(1+ \sum_{i=1}^s \sum_{V\in \scrV_i} \deg(V) \right)^{\dim(D)}.
\end{equation}
\end{lemma}

\begin{proof} First of all, as the degree is a sub-additive function, it will be enough to prove the Lemma for $D$ locally closed irreducible. We assume it from now on.

Secondly, observe that all elements in $\scrW$ are locally closed sets. Hence, for every $W\in \scrW$, its locally closed irreducible components are in bijection with the irreducible components of $\overline{W}^z$. This bijection is given by the fact that locally closed irreducible components are simply Zariski open subsets on irreducible components of $\overline{W}^z$ and the degree is preserved. Conversely, irreducible components of $\overline{W}^z$ are simply the Zariski closures of the locally closed irreducible components of $W$. We proceed by proving the upper bound of Identity (\ref{cota-superior-grado-celdas:eqn}).

  Let $d=\dim(D)$ be the dimension of $D$. Each  $W\in\scrW$ and each $C\in\scrD$ also have dimension smaller than $d$.  For each $k$, $0\leq k \leq n$, let us define $\scrD(k)$ as the set of elements in  $\scrD$ of dimension $k$, i.e.
$$\scrD(k):=\{ C \in \scrD\; : \; \dim(C)= k\},$$
where $\scrD(r)= \emptyset$  for $d+1\leq r \leq n$. As $D$ is locally closed irreducible,
we have $\scrD(d)=\{ \overline{D}^z\}$.

For  $k< d$  we have the following Claim:
\begin{claim}
For each $C\in \scrD(k)$ there exist $C^*\in\scrD$ and $W\in \scrV_1\cup \cdots \cup \scrV_s$ such that the following properties hold:
\begin{enumerate}
\item The dimension of  $C^*$ is greater than $k+1$ (i.e. $C^*\in \bigcup_{r=k+1}^d \scrD(r)$).
\item The irreducible variety $C$ is an irreducible component of the Zariski closure of the intersection $C^*\cap W$.
\end{enumerate}
\end{claim}

\begin{proof-claim}
Let $C\in \scrD(k)$ be an irreducible component and let $S\subseteq \{ 1,\ldots, s\}$ be a subset of minimal cardinal such that a Zariski open subset $U\cap C$ of $C$ is a locally closed irreducible component of some intersection:
$$D\cap \left( \bigcap_{i\in S} V_i\right),$$
where $V_i\in \scrV_i$. As $k<d$, $S$ cannot be the empty set. Let us define
 $S':=S\setminus \{ j\}$ for some $j\in S$. Then,  $U\cap C$ must be a locally closed irreducible component of
 $$(D\cap W')\cap V_j,$$
where $W':=\left( \bigcap_{i\in S'} V_i\right)$.

Then, there exists a locally closed irreducible component  $C^*$ of $W'$ such that some Zariski open subset of $C$  is a locally closed irreducible component of  $C^*\cap V_j$. Moreover, the dimension of $C^*$ must be bigger than $k+1$ since, otherwise, $C^*$ and $C$ would agree on a Zariski open subset and, hence, $S$ would not be of minimal cardinal. This proves the Claim. \end{proof-claim}

This Claim allows us to define the following mapping:
$$\begin{matrix}
\Phi_k: &\scrD(k)& \longrightarrow & \left(\bigcup_{r=k+1}^d\scrD(r)\right)\times \left( \bigcup_{i=1}^s \scrV_i\right)\\
& C & \longmapsto & (C^*, V)
\end{matrix},$$
given by the following rule:\\
To every $C\in\scrD(k)$ we associated a pair $(C^*, V)$ such that a non-empty Zariski open subset of $C$ is a locally closed irreducible component of $C^*\cap V$.

Let us define:
$$D(k):=\sum_{r=k}^d \sum_{C\in \scrD(r)}\deg(C),$$
we then prove by induction on $m=d-k$ the following inequality:
\begin{equation}\label{paso-final-cota-grado-intersecciones-menos:eqn}
D(k)=D(d-m)\leq \deg(D) \left(1+\sum_{V\in \bigcup_{i=1}^s \scrV_i} \deg(V) \right)^{m}.
\end{equation}
The case $m=0$ is obviously true since $\scrD(d)=\{\overline{D}^z\}$ and $D$ is locally closed. We have $D(d)=\deg(D)$.\\
Assume now that $m\geq 1$ and we have the following inequality:
$$\sum_{C\in \scrD(k)} \deg(C) \leq \sum_{(C^*, V)\in \left(\bigcup_{r=k+1}^d \scrD(r)\right)\times \left( \bigcup_{i=1}^s \scrV_i\right)} \; \sum_{C\in \Phi_k^{-1}(C^*, V)} \deg(C).$$
As all involved subsets are locally closed, for each $C^* \in \left(\bigcup_{r=k+1}^d \scrD(r)\right)$ and for each
$V\in \left( \bigcup_{i=1}^s \scrV_i\right)$, from Theorem \ref{Bezout-localmente-cerrados:teor} we obtain:
$$\sum_{C\in \Phi_k^{-1}(C^*, V)} \deg(C)\leq \deg(C^*)\deg(V).$$
Therefore, we have:
$$\sum_{C\in \scrD(k)} \deg(C) \leq \sum_{(C^*, V)\in \left(\bigcup_{r=k+1}^d \scrD(r)\right)\times \left( \bigcup_{i=1}^s \scrV_i\right)} \deg(C^*)\deg(V).$$
Rearranging the sums, we get:
$$\sum_{C\in \scrD(k)} \deg(C) \leq \left(\sum_{C^*\in \left(\bigcup_{r=k+1}^d \scrD(r)\right)} \deg(C^*)\right)
\left(\sum_{ V\in  \left( \bigcup_{i=1}^s \scrV_i\right)}\deg(V)\right)=D(k+1) R,$$
where $R=\sum_{ V\in  \left( \bigcup_{i=1}^s \scrV_i\right)}\deg(V)$.
As $D(k)= D(k+1) + \sum_{C\in \scrD(k)} \deg(C)$, we finally conclude:
$$D(k)= D(d-m) \leq D(d-(m-1))R+ D(d-(m-1))= D(k+1) (R+1).$$
Applying the inductive hypothesis one concludes Inequality (\ref{paso-final-cota-grado-intersecciones-menos:eqn}) and hence the proof of the Lemma.\end{proof}

\subsubsection{The wanted proof of the second upper bound for the degree in Theorem \ref{cotas-grado-varias-intersecciones:teor}.}

We may assume that $C_1$ is a constructible set. If the result is true in the locally closed irreducible case, we can decompose $C_1$ as a finite union of locally closed irreducible sets as in Lemma \ref{descomposicion-union-irreducibles-distintos:lema}:
$$C_1:=W_{1} \cup \cdots \cup W_{s},$$
that minimize the degree, i.e.:
$$\deg_{\rm lci}(C_1):= \sum_{i=1}^s \deg(W_i).$$
As the degree is sub-additive (see Proposition \ref{grado-constructibles-subaditivo:prop}) and $\dim(C_1)=\max\{\dim(W_i)\; : \; 1\leq i \leq s\}$, we would conclude the second upper bound from the case of locally closed irreducible sets by the following chain of inequalities:
$$\deg_{\rm lci}\left(C_1\cap \left(\bigcap_{i=2}^s C_i\right)\right)\leq \sum_{i=1}^s \deg_{\rm lci}\left(W_i\cap \left(\bigcap_{i=2}^s C_i\right)\right)\leq$$
$$\leq \sum_{i=1}^s \deg\left(W_i\right)\left(1+\sum_{i=1}^s\deg_{\rm lci}(C_i)\right)^{\dim(W_i)}\leq$$
 $$\leq \left(\sum_{i=1}^s \deg(W_i)\right) \left(1+\sum_{i=1}^s\deg_{\rm lci}(C_i)\right)^{\dim(C_1)} = \deg_{\rm lci}(C_1) \left(1+\sum_{i=1}^s\deg_{\rm lci}(C_i)\right)^{\dim(C_1)},$$
which proves the Proposition for any constructible set $C_1$. \\

We denote by $d:=\dim(C_1)$ the dimension of $C_1$ and we assume that $C_1$ is locally closed irreducible.
We also consider for each constructible set $C_i$, $2\leq i \leq s$, a decomposition according to Lemma
\ref{descomposicion-union-irreducibles-distintos:lema}:
$$C_i:=W_{i,1}\cup\ldots \cup W_{i,s(i)},$$
where each $W_{i, j}$ is a Zariski open subset of an irreducible variety that we denote by $V_{i,j}$. Assume also that these decompositions minimize the degree of the $C_i$'s, i.e.
$$\deg_{\rm lci}(C_i):=\sum_{i=1}^{s(i)} \deg(W_{i,j}) = \sum_{i=1}^{s(i)} \deg(V_{i,j}), \: 2 \leq i \leq s.$$
We now introduce classes of locally closed irreducible sets $\scrV_1, \ldots, \scrV_s$ given by the following identities:
$$\scrV_i:=\{ W_{i,1}, \ldots, W_{i,s(i)}\}, \; 2\leq i \leq s.$$
We define the class $\scrV$ of locally closed sets:
$$\scrV:=\{ W\; :\; \forall i, \; 2\leq i \leq s,\;  \exists W_i\in \scrV_i, \;
W=C_1\cap\left(\bigcap_{i=2}^s W_i\right)\}.$$
And, finally, we define the class $\scrC$  of the irreducible algebraic varieties given by the following equality:
$$\scrC:=\{ C \; : \; {\hbox {\rm $C$ is irreducible}},\; \exists W\in \scrV, {\hbox {\rm a non-empty open subset of $C$ is a component of $W$}}\}.$$
For each $C\in \scrC$, let $U_C\subseteq \A^n$ be a Zariski open subset, maximal with the property  $\emptyset \not= U_C\cap C \subseteq \left(C_1\cap \left(\bigcap_{i=2}^s C_i\right)\right)$.
\begin{claim}
With these notations, we have:
$$C_1\cap \left(\bigcap_{i=2}^s C_i\right)=  \bigcup_{C\in \scrC} U_C\cap C.$$
\end{claim}

\begin{proof-claim}
For each $C\in \scrC$, there exists a Zariski open subset $U_0\subseteq \A^n$ such that $U_0\cap C\not=\emptyset$ and $U_0\cap C$ is a locally closed irreducible component of a locally closed subset of the following form:
$$C_1\cap \left(W_2 \cap \cdots \cap W_s\right),$$
where $W_i\in \scrV_i$. Thus, $U_0\cap C\subseteq \left(C_1\cap \left(\bigcap_{i=2}^s C_i\right)\right)$ and, by the definition of  $U_C$ we conclude $U_C\cap C\subseteq  \left(C_1\cap \left(\bigcap_{i=2}^s C_i\right)\right)$.\\
On the other hand, fix some point $x\in \left(C_1\cap \left(\bigcap_{i=2}^s C_i\right)\right)$. Since the equality
$$C_1\cap \left(\bigcap_{i=2}^s C_i\right) = \bigcup_{(W_2,\ldots, W_s)\in \scrV_1\times \cdots \times \scrV_s} C_1\cap\left(\bigcap_{i=2}^s W_i\right),$$
holds, we conclude that there exist $W_2\in \scrV_2, \ldots, W_s\in \scrV_s$ such that $x\in \left(C_1\cap \left(\bigcap_{i=2}^s W_i\right)\right)$. In particular, $x$ must belong to some locally closed irreducible component of  $\left(C_1\cap \left(\bigcap_{i=1}^s W_i\right)\right)$ and, hence, there must exist some  $C\in \scrC$ such that $x\in U_C\cap C$.\end{proof-claim} \\
By Proposition \ref{grado-constructibles-subaditivo:prop} ($LCI$-degree is sub-additive), we conclude:
\begin{equation}\label{primer-paso-hacia-componentes:eqn}
\deg_{\rm lci}\left(C_1\cap \left(\bigcap_{i=2}^s C_i\right)\right)\leq \sum_{C\in \scrC}\deg(C).
\end{equation}
Now we consider the family $\scrW$ of locally closed sets (as that of Lemma \ref{control-celdas:lema} above) given by the following identity:
$$\scrW:=\{ W\; :\; \exists S\subseteq \{ 2, \ldots, s\},  \; \forall i\in S, \; \exists W_i\in \scrV_i, \; W=C_1\cap\left(\bigcap_{i\in S} W_i\right)\}.$$
We also define a new family of irreducible varieties:
$$\scrD:=\{ C \; : \; {\hbox {\rm $C$ is irreducible and}},\; \exists W\in \scrW, {\hbox {\rm a non-empty open subset of $C$ is a component of $W$}}\}.$$
Obviously $\scrC\subseteq \scrD$ and, because of Lemma \ref{control-celdas:lema}, we conclude:
$$\sum_{C\in \scrC}\deg(C)\leq \sum_{C\in \scrD}\deg(C)\leq \deg(C_1)\left(1+ \sum_{i=2}^s \sum_{V\in \scrV_i} \deg(V) \right)^{\dim(C_1)}.$$
As we have
$\deg_{\rm lci}(C_i):=\sum_{V\in \scrV_i} \deg(V)$, according to the Inequality (\ref{primer-paso-hacia-componentes:eqn}) we finally conclude the second upper bound of Theorem \ref{cotas-grado-varias-intersecciones:teor}:
$$\deg_{\rm lci}\left(\bigcap_{i=1}^s C_i\right)\leq \deg(C_1) \left(1+ \sum_{i=2}^s \deg_{\rm lci}(C_i)\right)^{\dim(C_1)}.$$

\subsection{A couple of immediate applications to show bounds for multiple intersections of constructible sets}\label{cotas-intersecciones-multiples-varios:subsec}
\begin{example}[{\bf Counting $\F_q-$rational points in constructible sets}]
Let $\F_q$ be a finite field and $\overline{\F_q}$ be its  algebraic closure. For each constructible subset $C\subseteq \A^n (\overline{\F_q})$, we denote by $C_{\F_q}:=C\cap \F_q^n$ the set of its $\F_q-$rational points. The following generalizes to constructible sets a classical result for hyper-surfaces due to \O. Ore (cf. \cite{Ore}).

\begin{corollary}\label{comptage-des-points-rationels:corol}
With the above notations, for every constructible subset  $C\subseteq \A^n (\overline{\F_q})$, the number of $\F_q-$rational points satisfies:
$$\sharp\left(C_{\F_q}\right)= \sharp\left(C\cap \F_q^n\right)\leq \deg_{\rm lci}(C) q^{\dim(C)}.$$
In particular, for every non-zero polynomial $f\in \overline{\F_q}[X_1,\ldots, X_n]$, whose degree satisfies $\deg(f)\leq q-1$, there is some point $x\in \F_q^n$ such that $f(x)\not=0$. Moreover, the number of non-zeros of $f$ in $\F_q^n$ satisfies:
$$\sharp\{ x\in \F_q^n\; :\; f(x)\not=0\} \geq (q-\deg(V(f))q^{n-1}.$$

\end{corollary}
\begin{proof} Firstly, let us consider for every  $j$, $1\leq j \leq n$, the algebraic varieties:
$$W_j:=\overline{\F_q}^{j-1}\times \{ x_j\in \overline{\F_q}\; : \; x_j^{q}-x_j=0\}\times \overline{\F_q}^{n-j}\subseteq \A^n(\overline{\F_q}).$$
Each $W_j\subseteq \A^n(\overline{\F_q})$ is a hyper-surface of degree  $q$ and we obviously have:
$$\F_q^n= \bigcap_{j=1}^n W_j.$$
As $W_1,\ldots, W_n$ are algebraic varieties of degree at most $q$, from Proposition \ref{grado-constructible-cerrados-multiple:prop} we conclude:
$$\deg_{\rm lci}(C\cap \F_q^n) = \deg_{\rm lci}(C\cap \left(\bigcap_{j=1}^n W_j\right) )\leq \deg_{\rm lci}(C) q^{\dim(C)}.$$
Since $C\cap \F_q^n$ is a zero-dimensional variety (see Proposition \ref{propiedades-basicas-grado-lc:prop}), we have:
$$\sharp\left(C_{\F_q}\right)= \sharp\left( C\cap \F_q^n\right)= \deg_{\rm lci}(C\cap \F_q^n).$$
As for the second claim, since $V(f)$ is a non-empty hyper-surface (of dimension $n-1$) in $\A^(\overline{\F_q})$, Proposition \ref{propiedades-basicas-grado-lc:prop} implies that the degree of  $V(f)$ is at most $\deg(f)$. Therefore, we have:
$$\sharp\left(V(f)\cap \F_q^n\right) \leq \deg(V(f)) q^{n-1}\leq \deg(f) q^{n-1}\leq (q-1) q^{n-1}.$$
Hence, we obviously conclude:
$$\sharp\{ x\in \F_q^n\; :\; f(x)\not=0\} \geq (q-\deg(V(f))q^{n-1}.$$
\end{proof}

The results of Weil, Lang and Stepanov (cf. \cite{Weil}, \cite{LangWeil} and \cite{Stepanov}) show that the bound of the previous Corollary is essentially optimal if $C$ is an algebraic variety.

\end{example}
\begin{example}[{\bf Evasive zero-dimensional varieties of constructible sets}]
\begin{corollary}\label{DeMillo-Lipton-Schwartz-Zippel:corol}
Let $\kappa$ be a field, $Q\subseteq \kappa$ a finite subset and $K$ the algebraic closure of $\kappa$.
Let $C\subseteq \A^n (K)$ be a constructible subset. Then, the number of points of the intersection of  $C$ and $Q^n$ satisfies:
$$\sharp\left(C\cap Q^n\right)\leq \deg_{\rm lci}(C) \sharp(Q)^{\dim(C)}.$$
In other terms, the probability that a random point $x\in Q^n$ satisfies $x\not\in C$ is greater than
$$1- {{\deg_{\rm lci}(C)}\over{\sharp(Q)^{\codim(C)}}},$$
where $\codim(C)= n- \dim(C)$ is the co-dimension  of $C$ in $\A^n (K)$.
\end{corollary}
\begin{proof} The proof is similar to that of the preceding Corollary by replacing $\F_q$ by $Q$, and replacing the equation of degree $q$ that defines $\F_q$ in $\overline{\F_q}$ by an univariate polynomial of degree $\sharp(Q)$ whose zeros in $K$ are exactly $Q$. In what concerns to the lower bound for the probability, it is a consequence of the following chain of inequalities:
$${\rm Prob}_{Q^n}[ x\in Q^n \; :\; x\not\in C]\geq 1- {{\deg_{\rm lci}(C)\sharp(Q)^{\dim(C)}}\over{\sharp(Q)^{n}}}=
1- {{\deg_{\rm lci}(C)}\over{\sharp(Q)^{\codim(C)}}}.$$
\end{proof}

Observe that in the case $C=V(f)$, where $C$ is a hyper-surface defined by  a non-zero polynomial, this last Corollary is the foundation of the classical \emph{Demillo-Lipton-Schwartz-Zippel Lemma} (cf. \cite{DeMilloLipton}, \cite{Zippel} and \cite{Schwartz}).
\end{example}

\section{Correct Test Sequences}\label{CTS:sec}

In this Section we revise the notion of correct test sequence (introduced in \cite{HeintzSchnorr}) and see that it is almost omnipresent, by exhibiting its presence in many scientific contexts. We do not want to be exhaustive here, but to exhibit different equivalent terms in the mathematical literature. We also prove some elementary and immediate properties.

\subsection{The notion, some equivalent notions and first properties}

\begin{definition}[{\bf Correct Test Sequences for vector functions}]\label{CTS-funciones-definicion:def}
Let $X$ be a set, $K$ a field and let $\scrF(X)\subseteq K^X$ be a subgroup of the abelian group $(K^X, +)$. Let $m\in\N$ be a positive integer and let $\Omega\subseteq \scrF(X)^m$ be a set of lists of such functions. Let $\Sigma\lestricto \Omega$  be a proper subset of $\Omega$, which we call the discriminant.

A correct test sequence (CTS) of length $L$ for $\Omega$ with discriminant $\Sigma$ is a finite set of $L$ elements ${\bf Q}:=\{x_1,\ldots, x_L\} \subseteq X$ such that the following formula holds:
\begin{equation}\label{CTS-formula:eqn}
\forall f\in \Omega, \;\; f(x_1)=\cdots = f(x_L)=0\in K^m \Longrightarrow f\in \Sigma.
\end{equation}
In the case $\Sigma=\{0\}\lestricto \Omega$, we say that ${\bf Q}$ is a correct test sequence for $\Omega$.
\end{definition}

Given $\Omega\subseteq \scrF(X)^m$ we denote by $\Omega-\Omega$ the subset  of $(\scrF(X)^m, +)$ given by the differences between elements in $\Omega$. Namely,
$$\Omega-\Omega:=\{f-g\; : \; f,g\in \Omega\}\subseteq \scrF(X)^m.$$
Observe that if $\Omega$ is a semi-group in $(\scrF(X)^m, +)$, then $\Omega-\Omega$ is simply the abelian subgroup of $(\scrF(X)^m, +)$ generated by $\Omega$.

For a list of points ${\bf Q}:=\{x_1,\ldots, x_L\}$ we define the evaluation map at the points of ${\bf Q}$ as the mapping:
$$\begin{matrix} {\rm ev}_{\bf Q}: & \scrF(X)^m & \longrightarrow & K^{mL}\\
& (f_1,\ldots, f_m) & \longmapsto & (f_1(x_1), \ldots, f_m(x_1), \ldots,f_1(x_L), \ldots, f_m(x_L)),\end{matrix} $$
which is a $K-$linear mapping when $\scrF(X)^m$ is a vector subspace of $(K^m)^X$.

\begin{proposition}[{\bf CTS and Function Identity Tests}]
With the notations of the previous definition, let  $\Omega$  be a subset  of $\scrF(X)^m$.  Then, the following properties are equivalent for every finite subset
${\bf Q}:=\{x_1,\ldots, x_L\}\subseteq X$:
\begin{enumerate}
\item ${\bf Q}$ is a correct test sequence for $\Omega-\Omega$ (with respect to $\Sigma:=\{0\}$).
\item The following formula holds:
$$\forall f,g \in \Omega, \;\; f(x_1)=g(x_1), \ldots, f(x_L)=g(x_L) \Longrightarrow f=g.$$
\end{enumerate}

 If, additionally, $\scrF(X)^m$ is a vector subspace of $(K^m)^X$, both properties are equivalent to the following ones:\\
\begin{enumerate}[resume]
\item The evaluation map ${\rm ev}_{\bf Q}$ at the points in ${\bf Q}$ is a $K-$linear mapping  such that its restriction to $\Omega$   is injective.
\item If $L:=\sharp\left({\bf Q}\right)$, there is a linear map $\Lambda:K^{mL}\longrightarrow \scrF(X)^m$ such that the following two properties hold:
    \begin{enumerate}
    \item For all $f\in \scrF(X)^m$, $f-\Lambda({\rm ev}_{\bf Q}(f))\in \ker\left({\rm ev}_{\bf Q}\right)$ and,
    \item ${\rm ev}_{\bf Q}(\Omega) \cap \ker(\Lambda)=\{0\}$.
    \end{enumerate}
\end{enumerate}
\end{proposition}
\begin{proof} The equivalences between $i)$, $ii)$ and, in the case $\scrF(X)$ is a vector space, with $iii)$ are immediate. As for the equivalence between $iii)$ and $iv)$ simply assume that $\Lambda$ is some extension to $K^{mL}$ of the right inverse of the epimorphism between $\scrF(X)^m$ and ${\rm ev}_{\bf Q}(\scrF(X)^m)\subseteq K^{mL}$.
\end{proof}

Correct test sequences viewed as identity tests is the most common interpretation in the mathematical literature. When $\scrF(X)$ is a group of polynomials defined on some algebraic variety $X$, correct test sequences are called \emph{Polynomial Identity Tests} (as in \cite{Saxena} and references therein). These Polynomial Identity Tests have been used in Elimination Theory for years to compute Noether's normalizations, Kronecker's birational descriptions of equi-dimensional varieties and for many other purposes (see, for instance, \cite{lower-diophantine:jpaa}, \cite{Kronecker-CRAS}, \cite{Pardo-survey} and references therein). We return to this polynomial case in the next subsection.

With the same notations as above, assume that $\scrF(X)$ is a $K-$algebra. Then, for every $x\in X$, the following is a maximal ideal in $\scrF(X)$:
$${\frak m}_ x:=\{ f\in \scrF(X) \; : \; f(x)=0\}\in {\rm MaxSpec}(\scrF(X)).$$
In fact, ${\frak m}_x$ is the kernel of the onto map ${\rm ev}_x:\scrF(X) \longrightarrow K$ given by ${\rm ev}_x(f):=f(x)$,  for all $f\in \scrF(X)$. For every ${\bf Q}\subseteq X$ let us denote by $I({\bf Q})$ the radical ideal given by the following identity:
$$I({\bf Q}):=\bigcap_{x\in{\bf Q}} {\frak m}_x.$$

\begin{proposition} With the same notations as above, assume $\scrF(X)$ is a $K-$algebra  and $\Omega\subseteq \scrF(X)$ a subset. The following properties are equivalent:
\begin{enumerate}
\item A finite subset ${\bf Q}\subseteq X$ is a correct test sequence for $\Omega-\Omega$.
\item The canonical projection $\pi:\scrF(X)\longrightarrow \scrF({\bf Q}):=\scrF(X)/I({\bf Q})$ satisfies that its restriction $\restr{\Omega}{{\bf Q}}$ is injective.
\item There is a $K-$vector subspace $W({\bf Q})$ of $\scrF(X)$ of dimension $\sharp({\bf Q})$ such that the following properties hold:
    \begin{itemize}
    \item For every $f\in \Omega$ there is one and only one $g\in W({\bf Q})$ such that $f-g\in I({\bf Q})$.
    \item For every $g\in W({\bf Q})$ there is at most one $f\in \Omega$ such that $f-g\in I({\bf Q})$.
    \end{itemize}
     \end{enumerate}
\end{proposition}
\begin{proof}
By the Chinese Remainder Theorem, we have the following ring isomorphism:
$$\varphi: \scrF(X)/I({\bf Q}) \longrightarrow \prod_{i=1}^L \scrF(X)/{\frak m}_{x_i}\isomorf K^L,$$
where ${\bf Q}:=\{ x_1,\ldots, x_L\}$. Noting that:
$$\varphi(f + I({\bf Q})):=(f+ {\frak m}_{x_1}, \ldots, f + {\frak m}_{x_L}),$$
we obviously conclude the equivalence between $i)$ and $ii)$. Moreover, the Chinese Remainder Theorem means that there is interpolation in $\scrF(X)$. Namely, there are functions $\chi_1,\ldots, \chi_L\in \scrF(X)$ such that $\chi_i(x_j)=\delta_{i,j}$, for every $i,j$, $1\leq i,j \leq L$, where $\delta_{i,j}$ is Kronecker's delta. Then, taking $W({\bf Q}):={\rm Span}(\{\chi_1,\ldots, \chi_L\})$, the $K-$vector space spanned by $\{\chi_1,\ldots, \chi_L\}$, we obviously have that for every $f\in\Omega$, the function
$$g:=\sum_{i=1}^L f(x_i) \chi_i\in W({\bf Q}),$$
satisfies the requirements of the statement. Moreover, given any $g\in W({\bf Q})$ there is at most one element $f\in \Omega$ such that $f-g\in I({\bf Q})$.\end{proof}

\begin{example}[{\bf Finite Norming sets}] With the same notations as above, assume that $(K, \vert \cdot \vert)$ is a field with some absolute value. And for any positive integer $m\in \N$, let $\vert\vert \cdot \vert\vert$ be any norm in $K^m$ induced by $\vert \cdot \vert$. Assume that $\scrF(X)$ is a $K-$vector space and $\Omega \subseteq \scrF(X)$ is a subset, a \emph{finite norming set} for $\Omega$ is a finite subset ${\bf Q}:=\{x_1,\ldots, x_L\}\subseteq X$ such that  the following is a norm in $\Omega$:
$$\vert\vert f \vert\vert_\infty^{({\bf Q})}:=\max \{  \vert f(x_i)\vert \; :\; 1\leq i \leq L\}.$$
As $\Omega$ is not always a vector subspace, we use the term ``norm'' in a wide sense here. Just for completeness, a norm for $\Omega$ is a function $\vert\vert \cdot \vert \vert: \Omega \longrightarrow \R_+$ such that the following properties hold:
\begin{itemize}
\item $\vert\vert f \vert\vert \geq 0$, $\forall f\in \Omega$.
\item For all $f\in \Omega$, $\vert\vert f \vert\vert=0$ if and only if $f=0$.
\item Given $f,g\in \Omega$, if $f+g\in \Omega$, then we have
$$\vert\vert f + g \vert \vert \leq \vert\vert f \vert\vert+ \vert\vert g \vert\vert.$$
\item Given $f\in \Omega$ and $\lambda \in K$, if $\lambda f \in \Omega$, then $\vert\vert\lambda f \vert\vert \leq \vert \lambda \vert \vert\vert f \vert\vert$.
\end{itemize}
Obviously, if $\Omega$ is a vector subspace, this is the usual notion of norm. See \cite{Albiac} and references therein for the term \emph{norming set}.
\end{example}

\begin{proposition} \label{cuestores-y-metrica:prop}
 With the same notations and assumptions as in the former example, given ${\bf Q}:=\{x_1,\ldots, x_L\}\subseteq X$ a finite set and $\Omega \subseteq \scrF(X)$ a subgroup of the additive group $(\scrF(X), +)$, the following are equivalent:

\begin{enumerate}
\item The set ${\bf Q}$ is a correct test sequence for $\Omega$.
\item The set ${\bf Q}$ is a finite norming set for $\Omega$.
\item The restriction to the evaluation map ${\rm ev}_{\bf Q}$ to $\Omega$ is a group monomorphism.
\item There exists $p$, $1\leq p \leq \infty$, such that the following is a norm on $\Omega$:
$$\vert\vert f \vert\vert_p^{({\bf Q})}:= \left( \sum_{i=1}^L \vert f(x_{i}) \vert^p\right)^{1/p}.$$
\item For every $p$, $1\leq p \leq \infty$, the following is a norm on $\Omega$:
$$\vert\vert f \vert\vert_p^{({\bf Q})}:= \left( \sum_{i=1}^L \vert f (x_{i}) \vert^p\right)^{1/p}.$$
\item $\Omega$ is a metric space with the distance function:
$$d_\infty^{({\bf Q})}(f,g):=\vert\vert f-g \vert\vert_\infty^{({\bf Q})}.$$
\item For every $p$, $1\leq p \leq \infty$,  $\Omega$ is a metric space with the distance function:
$$d_p^{({\bf Q})}(f,g):=\vert\vert f-g \vert\vert_p^{({\bf Q})}.$$

\end{enumerate}
In particular, $\Omega$ has a finite length correct test sequence if and only if it admits a finite norming set. If, additionally, $\Omega$ were a linear subspace of $\scrF(X)$, then $\Omega$ admits a correct test sequence of finite length if and only if $\Omega$ is a normed linear space of finite dimension.
\end{proposition}
\begin{proof} It is obvious from the definitions. \end{proof}
The interesting cases are then restricted to subsets (not necessarily linear) of some finite dimension vector spaces.

\begin{example}[{\bf CTS's and the ring of continuous functions}]
Assume $X$ is a compact topological space, $\kappa=\R$ and $\scrF(X)\subseteq\scrC(X)$ is a vector subspace of the $\R-$algebra of continuous functions defined on $X$ with values in $\R$. Assume that $\scrC(X)$ is endowed with the topology defined by the maximum norm, which is a Banach $\R-$algebra. Because of the Theorem of
Banach-Stone-\v{C}ech-Gel'fand-Kolmogorov (cf. \cite{Gillman-Heinriksen}),
 we know that there exists a bijection between the maximal ideals in the spectrum of $\scrC(X)$ (which we denote by ${\rm MaxSpec}(\scrC(X))$) and the set of points in $X$. This bijection is given by the following identification:
$$ x\longmapsto {\frak m}_x:=\{f\in \scrC(X)\; :\; f(x)=0\}.$$
With these notations and assumptions, a finite subset ${\bf Q}:=\{x_1,\ldots, x_L\}\subseteq X$ of cardinality $L$ is a correct test sequence for $\Omega\subseteq \scrF(X)$ with respect to a discriminant $\Sigma\lestricto \Omega$ if and only if the following holds:
$$\Omega \cap \left( \bigcap_{i=1}^L {\frak m}_{x_i}\right) \subseteq \Sigma,$$
or, equivalently, if and only if the following is true:
$$\Omega\setminus \Sigma\subseteq  \left( \bigcup_{i=1}^L {\frak m}_{x_i}^c\right),$$
where ${\frak m}_{x_i}^c= \scrC(X)\setminus {\frak m}_{x_i}$.
\end{example}

\begin{proposition}\label{CTS-funciones-continuas:prop} With the same notations as above, assume that $X$ is a compact Hausdorff space, $\scrF(X)=\scrC(X)$, $\Omega \subseteq \scrC(X)$ and $\Sigma\lestricto \Omega$. Let $(\Sigma)$ be the ideal in $\scrC(X)$ generated by $\Sigma$ and assume that:
\begin{equation}\label{Hewit:eqn}
\sqrt[J]{(\Sigma)}\cap \Omega = \Sigma,
\end{equation}
where $\sqrt[J]{\cdot}$ means Jacobson radical.
Then, if  $\Omega\setminus \Sigma$ is quasi-compact in $\scrC(X)$ (i.e. every open covering has
a finite sub-covering), then there is a finite length correct test sequence for $\Omega$ with respect to $\Sigma$.\\
In particular, if $\Sigma=\{0\}$ and $\Omega\setminus \{0\}$ is quasi-compact, then there is a correct test sequence of finite length for $\Omega$ or, equivalently, $\Omega$ has a finite norming set.
\end{proposition}
\begin{proof} From Banach-Stone-\v{C}ech-Gel'fand-Kolmogorov Theorem, as $X$ is compact and Hausdorff, we know that there is an identification:
$$X\isomorf {\rm MaxSpec}(\scrC(X)).$$
Let us also consider the following closed subset of $X$:
$$V(\Sigma)=\{x\in X\; :\; f(x)=0, \; \forall f\in \Sigma\}=
\{x\in X \; :\; f(x)=0, \; \forall f \in (\Sigma)\}.$$
Then, the following equalities hold:
$$\sqrt[J]{(\Sigma)}= \bigcap_{\begin{matrix} {\frak m}\in {\rm MaxSpec}(\scrC(X))\\{\frak m}\supseteq \Sigma\end{matrix}} {\frak m}= \bigcap_{x\in V(\Sigma)} {\frak m}_x.$$
The first equality is simply the definition of Jacobson radical of an ideal. As for the second, we obviously have:
$$\bigcap_{\begin{matrix} {\frak m}\in {\rm MaxSpec}(\scrC(X))\\{\frak m}\supseteq \Sigma\end{matrix}} {\frak m}\subseteq \bigcap_{x\in V(\Sigma)} {\frak m}_x.$$
Conversely, let ${\frak m}$ be a maximal ideal in $\scrC(X)$ that contains $\Sigma$.  Then, there is some $x\in X$ such that ${\frak m}={\frak m}_x$. As ${\frak m}_x\supseteq \Sigma$, then $f(x)=0$ for all $f\in \Sigma$ and, hence, $x\in V(\Sigma)$. This implies the converse inclusion.
Thus, our hypothesis implies that:
$$\sqrt[J]{(\Sigma)}\cap \Omega = \left( \bigcap_{x\in V(\Sigma)} {\frak m}_x\right)\cap \Omega= \Sigma.$$
Thus,
$$\Omega\setminus \Sigma \subseteq \bigcup_{x\in V(\Sigma)} {\frak m}_x^c.$$
As ${\frak m}_x^c$ is open in $\scrC(X)$, we have an open covering of $\Omega\setminus \Sigma$, which is quasi-compact by hypothesis. Then, there is a finite sub-covering:
$$\Omega\setminus \Sigma \subseteq \bigcup_{i=1}^L {\frak m}_{x_i}^c,$$
for some finite subset ${\bf Q}:=\{x_1,\ldots, x_L\}$. But this last inclusion just means that ${\bf Q}$ is a correct test sequence of length $L$ for $\Omega$ with respect to $\Sigma$, and the first claim of our statement holds.\\
The particular case holds simply because $V(\{0\})=X$ and $\sqrt[J]{(0)}= \{0\}$.
\end{proof}

Similarly, we also have:
\begin{corollary}
Let $r\in \N\cup\{\infty\}$ be a positive integer or $\infty$.  Let $X$ be a compact $\scrC^r-$ manifold.
Then, for every $\Omega \subseteq \scrC^r(X)$ such that $0\in \Omega$, if $\Omega \setminus \{0\}$ is quasi-compact, then there is a correct test sequence for $\Omega$ of finite length or, equivalently, $\Omega$ has a finite norming set.

\end{corollary}
\begin{proof} Essentially the same proof after observing that:
$${\rm MaxSpec}(\scrC^r(X)) \isomorf X.$$
\end{proof}

Note that the hypothesis described in Identity (\ref{Hewit:eqn}) may be rewritten as:
 $$\overline{(\Sigma)}\cap \Omega \subseteq \Sigma,$$
 where $\overline{(\Sigma)}=\sqrt[J]{(\Sigma)}$ is the closure of the ideal $(\Sigma)$ in $\scrC(X)$ with respect to the Hewitt $m-$topology (cf. \cite{Gillman-Heinriksen} and \cite{Hewit}).

A typical example where the hypothesis \emph{`` $\Omega\setminus \Sigma$ is quasi-compact''} holds is the case where $\Sigma$ is closed in $\Omega$ for the topology induced  by that of $\scrF(X)$ and  $\Omega$ is a Noetherian topological space.

Inspired by the case of continuous functions, we introduce the following notion:

\begin{definition}[{\bf CTS covering numbers}]\label{covering-number-CTS:def}
With the same notations as above, the covering number for $\Omega\subseteq \scrF(X)$ with respect to $\Sigma$ is the minimum length of a finite correct test sequence for $\Omega$ with respect to $\Sigma$.
We denote this quantity by:
$$\scrN_{\rm cts}(\Omega, \Sigma):= \min\{ L \in \N \; : \;\exists\; {\hbox {\rm a correct test sequence ${\bf Q}$ of length $L$  for $\Omega$ with respect to $\Sigma$}}\}.$$
We denote by $\scrN_{\rm cts}(\Omega):=\scrN_{\rm cts}(\Omega, \{0\})$.
\end{definition}

\begin{example}[{\bf Correct Test Sequences in Reproducing Kernel Hilbert spaces}]
Let $K: X\times X \longrightarrow \C$ be a kernel and let $(\scrH_K, \langle \cdot, \cdot \rangle_K)$ be its associated Hilbert space. Recall that $\scrH_K\subseteq \C^X$ and, in the case that $X$ is a metric space and that $K$ is a Mercer kernel,
$\scrH_K\subseteq \scrC(X,\C)\isomorf\scrC(X)[i]:=\scrC(X)[T]/(T^2+1)$ is made by continuous functions with complex values. For every $x\in X$, let $K_x: X \longrightarrow \C$ be the element in $\scrH_K$ given by:
$$K_x(y):= K(x,y),\; \forall y \in X.$$ Then, it is well-known that for every $f\in \scrH_K$ and for every $x\in X$, the following equality holds:
$$f(x)= \langle f, K_x \rangle.$$
\end{example}

\begin{proposition}\label{CTS-RKHS:prop} With these last notations, let $\Omega\subseteq \scrH_K$ be a subset and let $\Sigma\lestricto \Omega$ be a discriminant. A finite subset of points  ${\bf Q}=\{x_1,\ldots, x_L\}\subseteq X$ is a correct test sequence of length $L$ for $\Omega$ with respect to $\Sigma$ if and only if:
$$\Omega \cap \left( {\rm Span}(\{K_{x_1}, \ldots, K_{x_L}\})\right)^\perp \subseteq \Sigma,$$
where ${\rm Span}(\{K_{x_1}, \ldots, K_{x_L}\})$ is the vector subspace of $\scrH_K$ generated by $\{K_{x_1}, \ldots, K_{x_L}\}$ and ${}^\perp$ means orthogonal complement in $\scrH_K$.
\end{proposition}
\begin{proof}
Note that this is simply another way to express Identity (\ref{CTS-formula:eqn}).
\end{proof}

Observe also that the condition of being a correct test sequence is hereditary. Namely, we obviously have the following statement:

\begin{proposition}\label{hereditary-cts:prop}
Let $\Omega\subseteq \widetilde{\Omega}\subseteq \scrF(X)^{m}$ be two subsets and let $\Sigma \lestricto \widetilde{\Omega}$ be a discriminant. Let ${\bf Q} \subseteq X$ be a correct test sequence for $\widetilde{\Omega}$ with respect to $\Sigma$. Then, ${\bf Q}$ is a correct test sequence for $\Omega$ with respect to $\Sigma\cap \Omega$.
\end{proposition}

\subsection{The case of polynomials and lists of polynomials}\label{CTS-polynomials:subsec}

For every positive number $d\in \N$ we denote by $P_d^{K}(X_1,\ldots, X_n)$ the class of all polynomials of degree at most $d$ with coefficients in $K$ in the set of variables $\{X_1,\ldots, X_n\}$ (when both the field $K$ and the set of variables $\{X_1,\ldots, X_n\}$ are clear from the context we simply write $P_d$). For a list of degrees $(d):=(d_1,\ldots, d_m)$, with $m\leq n$, we introduce the class $\scrP_{(d)}^K(X_1,\ldots, X_n)$ of all lists of polynomials $f:=(f_1,\ldots, f_m)$ such that each $f_i \in P_{d_i}^K$. Namely, $\scrP_{(d)}^K$ is the Cartesian product:
$$\scrP_{(d)}^{K}(X_1,\ldots, X_n):=\prod_{i=1}^mP_{d_i}^K(X_1,\ldots, X_n).$$
If both $K$ and the set of variables are clear from the context, we write $\scrP_{(d)}$.

Obviously, $\scrP_{(d)}$ is a finite dimension affine space over $K$. We denote by $N_{(d)}$ this dimension, which  is given by the following identity:
$$N_{(d)}:=\sum_{i=1}^m {{d_i + n}\choose {n}}.$$
In \cite{HeintzSchnorr}, the authors considered  correct test sequences for a constructible subset $\Omega \subseteq P_{d}$ with respect to $\{0\}$ as discriminant. Here, we also consider constructible sets $\Omega\subseteq \scrP_{(d)}$.

For every constructible set $\Omega\subseteq \scrP_{(d)}$ and for every point $x\in\A^n$ we denote by $\Omega_x\subseteq \A^n$ the following constructible subset:
$$\Omega_x:=\Omega\cap\{ f\in \scrP_{(d)}\; :\; f(x)=0\}.$$
Thus, we may also rewrite the notion of correct test sequence as a list  ${\bf Q}:=(x_1,\ldots, x_L)\in V^L$, where $V \subset \A^{n}$,  of points such that:
$$\Omega_{x_1}\cap \cdots \cap \Omega_{x_L}\subseteq \Sigma.$$

\begin{proposition}
Let ${\bf Q}=\{x_1,\ldots, x_L\}$ be a correct test sequence of  $\Omega$ with respect to  $\Sigma$. Then, we have:
$$\deg_{\rm lci}\left(\Omega_{x_1}\cap \cdots \cap \Omega_{x_L}\right)\leq \deg_{\rm lci}(\Omega).$$
Moreover, if $\Sigma$ is zero-dimensional, the intersection $\Omega_{x_1}\cap \cdots \cap  \Omega_{x_L}$ is a finite set and we have:
$$\sharp\left(\Omega_{x_1}\cap \cdots \cap \Omega_{x_L}\right)\leq \deg_{\rm lci}(\Omega).$$
In particular, if $\deg_{\rm lci}(\Sigma)> \deg_{\rm lci}(\Omega)$,
we have an strict inclusion:
$$\Omega_{x_1}\cap \cdots \cap \Omega_{x_L}\subsetneq \Sigma.$$
\end{proposition}
\begin{proof}
Observe that for every point $x\in \A^n$, the set $\Omega_x$ is the intersection of $\Omega$ with a linear affine variety. Then, all the claims are immediate consequences of Proposition \ref{varias-propiedades-basicas-grado-construtibles:prop}.
\end{proof}

The following Proposition shows the limits of length of correct test sequences in terms of the Krull dimension of the given constructible set:

\begin{proposition}\label{codimension-longitud-cuestores:prop}
With the same notations as above, let $\Omega\subseteq \scrP_{(d)}$ be a constructible subset, $\Sigma\subseteq \Omega$ be a discriminant in $\Omega$ and ${\bf Q} \in V^L$  be a correct test sequence of length $L$ for  $\Omega$ with respect to $\Sigma$. Then, we have:
$$\sharp({\bf Q})=L\geq \codim_\Omega(\Sigma)=\dim(\Omega)-\dim(\Sigma).$$
In particular, the CTS covering number satisfies:
$$\scrN_{\rm cts}(\Omega, \Sigma)\geq \dim(\Omega)-\dim(\Sigma).$$
\end{proposition}
\begin{proof} The intersection $\Omega_{x_1}\cap \cdots \cap \Omega_{x_L}$ is the intersection of $\Omega$ with a sequence  of $L$ linear equations whose variables are the coefficients of the elements in $\scrP_{(d)}$. Then, by Krull's Hauptidealsatz, the dimension of this intersection satisfies:
$$\dim \left(\Omega \cap \{ f\in \scrP_{(d)}\; :\; f(x_1)=0\} \cap\cdots \cap \{ f\in \scrP_{(d)}\; :\; f(x_L)=0\}\right)\geq \dim(\Omega)- L.$$
However, if ${\bf Q} :=(x_1,\ldots, x_L)$ is a correct test sequence for  $\Omega$ with respect to $\Sigma$, we also have:
$$\Omega \cap \{ f\in \scrP_{(d)}\; :\; f(x_1)=0\} \cap\cdots \cap \{ f\in \scrP_{(d)}\; :\; f(x_L)=0\}\subseteq \Sigma.$$
Hence, we conclude the Proposition since we have:
$$\dim(\Omega)- L\leq \dim(\Sigma).$$
\end{proof}

From this statement, we will say that \emph{a correct test sequence for $\Omega$ is of optimal length} if its length is in $O(\dim(\Omega))$.

\section[Probability and CTS in the polynomial case]{High probability of existence of short Correct Test Sequences in equi-dimensional locally closed sets for constructible sets of lists of polynomials}\label{probability-CTS-zero-dimensional:sec}

With the same notations as in previous sections, let  $\Omega \subseteq \scrP_{(d)}^K$ be a constructible set of polynomials with $(d)=(d_1,\ldots, d_m)$, $m\leq n$. We generalize the main probability outcome of \cite{HeintzSchnorr} by showing that the ``grid'' requirement is not a must and that we may find (with high probability) correct test sequences in many locally closed sets of accurate dimension and degree. Additionally, we see that correct test sequences are well-suited not only for zero tests but also to decide membership to proper subvarieties of $\Omega$.

As in Subsection \ref{Bezout-Inequality-locally-closed:subsec}, given $n,r\in \N$  two positive integers such that $r\leq n$, we consider the ``Grassmannian'' $\G(n,r)$  of all linear affine varieties in $\A^n$ of co-dimension $r$. We assume that $\G(n,r)$ is endowed with its Zariski topology as introduced in Subsection \ref{Bezout-Inequality-locally-closed:subsec}.

Let $\Sigma\subseteq \Omega$ be a constructible subset of co-dimension at least 1. Let $C\subseteq \A^n$ be a constructible subset and $L\geq 0$ a positive integer. We denote by $R(\Omega, \Sigma, C, L)$ the constructible set of all sequences ${\bf Q}\in C^L$ of length $L$ which are correct test sequences for $\Omega$ with respect to $\Sigma$.

We have added an hypothesis (Hypothesis (\ref{dimension-hipotesis-cts:eqn})) on the dimension of the varieties defined by lists of polynomials in $\Omega\setminus\Sigma$. Observe that this simply generalizes the case of $\Sigma=\{0\}$ and explains that the main subject in correct test sequences is \emph{dimension} and not only zero tests. For locally closed sets $C\subseteq\A^n$ we just denote by $\deg(C)$ the degree of $C$, as discussed in previous sections.

\begin{theorem}\label{teorema-principal-densidad-questores:teor} Let $m,n \in\N$ be two positive integers, with $m\leq n$, and let $(d):=(d_1,\ldots, d_m)$ be a list of degrees and $d:=\max\{d_1,\ldots, d_m\}$. Let $\Sigma\subseteq\Omega$ be two constructible subsets of $\scrP_{(d)}^K$ such that $\Sigma$ has co-dimension at least 1 in $\Omega$. Assume that $\Omega\setminus \Sigma$ satisfies  the following property:
  \begin{equation}\label{dimension-hipotesis-cts:eqn}
  \forall f:=(f_1,\ldots, f_m)\in \Omega \setminus \Sigma,\; \dim(V_\A(f_1,\ldots, f_m))= n-m.
\end{equation}
Let $C\subseteq \A^n$ be an equi-dimensional locally closed set of co-dimension $r\geq (n-m)+ m/2 + 1/2$. Assume that there are locally closed subsets $C_1,\ldots, C_s\subseteq \A^n$ such that $C:= C_1\cap\ldots\cap C_s$.
Let $D:=\max\{ \deg(C_1), \ldots, \deg(C_s)\}\geq 1$  be the maximum of their degrees. Let $L\in \N$ be a positive integer and assume that the following properties hold:
\begin{enumerate}
\item $L\geq 6\dim(\Omega)$,
\item $\log(\deg\left(C\right)) \geq r\max\{2(1+\log(d+1)), {{2\log(\deg_{\rm lci}(\Omega))}\over{\dim(\Omega)}}\}$,
\item $D\leq(1+ {{1}\over{n-m}})\deg(C)^{{{1}\over{\codim(C)}}}$,
\end{enumerate}
where $\log$ stands for the natural logarithm.
Then, there is a non-empty Zariski open subset $\G(C)\subseteq \G(n, n-r)$ such that for every $A\in \G(C)$, the probability that a randomly chosen list ${\bf Q}\in (C\cap A)^L$ is in $R:=R(\Omega, \Sigma, C, L)$ satisfies:
$${\rm Prob}_{(C\cap A)^L}[R]\geq 1- {{1}\over{\deg_{\rm lci}(\Omega)e^{\dim(\Omega) + (m-1)L}}},
$$
where $(A\cap C)^L$ is endowed with its uniform probability distribution.
\end{theorem}

\begin{proof}
First of all, let us denote by $\delta$ the following quantity:
$$\delta:=\deg(C)^{{{1}\over{\codim(C))}}}.$$
We obviously have $\deg(C):=\delta^r$ and our hypothesis yield the following inequalities:
\begin{itemize}
\item $\log(\delta) \geq \max\{2(1+\log(d+1)), {{2\log(\deg_{\rm lci}(\Omega))}\over{\dim(\Omega)}}\}$,
\item $D\leq (1+ {{1}\over{n-m}})\delta$.
\end{itemize}
Then, we introduce the following incidence constructible set:
$$V(\Omega, L):=\{(f, x_1, \ldots, x_L)\in \Omega\times (\A^n)^L \; : \; f=(f_1,\ldots, f_m)\in \scrP_{(d)}^K, \; f_i(x_j)=0, 1\leq i \leq m, 1\leq j \leq L\}.$$
We also consider the following two canonical projections:
\begin{itemize}
\item The projection onto the set of lists of consistent equations: $\pi_1: V(\Omega,L) \longrightarrow \Omega$.
\item The projection onto the possible zeros: $\pi_2: V(\Omega) \longrightarrow (\A^n)^L$.
\end{itemize}
By the B\'ezout's Inequality for constructible sets (Theorem \ref{Desigualdad-Bezout-constructibles:teor}), we have:
$$\deg_{\rm lci}(V(\Omega,L))\leq \deg_{\rm lci}(\Omega)\left(\prod_{i=1}^m (d_i+1)\right)^L.$$
Let  $\scrD$ be the class formed by locally closed irreducible sets  $W_1,\ldots, W_M$ such that:
 $$ V(\Omega,L)=W_1\cup \cdots \cup W_M,$$
 and they are minimal with respect to the degree of $V(\Omega)$:
 $$\deg_{\rm lci}(V(\Omega,L))=\sum_{i=1}^M \deg(W_i).$$
Let $\scrC\subseteq \scrD$ be the class of those locally closed irreducible components of $V(\Omega, L)$ such that its projection by $\pi_1$ is not completely included in $\Sigma$:
 $$\scrC:=\{W\in \scrD\; :\; \pi_1(W)\setminus \Sigma\not=\emptyset\}.$$
As $\pi_1$ is the restriction of a projection, for all $W\in \scrC$,
we may consider the following mapping:
$$\restr{\pi_1}{\overline{W}^z}:=\overline{W}^z\longrightarrow \overline{\pi_1(W)}^z.$$
It is a dominant morphism between two irreducible algebraic varieties.
Then, by the Theorem on the Dimension of the Fibers (cf. \cite{Shafarevich}, Theorem 7, p. 60), for all $f\in \pi_1(W)$ we have:
\begin{equation}\label{cota-dimension-malas-componentes:eqn}
\dim\left(\restr{\pi_1}{\overline{W}^z}^{-1}\left(\{ f\}\right)\right)\geq \dim (\overline{W}^z) - \dim (\overline{\pi_1(W)}^z)\geq
\dim(W)- \dim (\Omega).
\end{equation}
Observe that for every $f\in \pi_1(W)$ the following equality holds:
$$\restr{\pi_1}{\overline{W}^z}^{-1}\left(\{ f\}\right)=\{f\} \times \left(V_\A(f)\right)^L.$$
Next, for every $W\in\scrC$, there is some system of equations $f\in \pi_1(W)\setminus \Sigma$
and, hence, we have:
$$\dim\left(\restr{\pi_1}{\overline{W}^z}^{-1}\left(\{ f\}\right)\right)= (n-m)L,$$
And Inequality (\ref{cota-dimension-malas-componentes:eqn}) becomes for every $W\in \scrC$:
\begin{equation}\label{cota-dimension-buenas-componentes:eqn}
\dim(W)\leq  \dim (\Omega)+ (n-m)L.
\end{equation}
Let us now define the constructible set:
$$B(\Omega, L):= \bigcup_{W\in \scrC} W \subseteq V(\Omega, L).$$
We have:
\begin{itemize}
\item $\dim(B(\Omega, L))= \max\{ \dim(W)\; :\; W\in \scrC\}\leq (n-m)L+\dim(\Omega)$ and,
\item we also have:
$$\deg_{\rm lci}(B(\Omega, L)) \leq \deg_{\rm lci}(\Omega) \left( \prod_{i=1}^m (d_i+1)\right)^L\leq \deg_{\rm lci}(\Omega) \left( d+1\right)^{mL},$$
where $d:=\max\{d_1,\ldots, d_m\}$.
\end{itemize}
Now, as in Corollaries \ref{generaliza-grado-localmente-cerrados:prop} and  \ref{intereseccion-lineales-no-equi-dimensional:corol}, there is a non-empty Zariski open subset $\G(C)\subseteq \G(n,n-r)$ such that for all $A\in\G(C)$ we have $\sharp(C\cap A)= \deg(C)$.
Let us now consider $A\in \G(C)$ and the following constructible set:
$$\overline{\pi_2(B(\Omega, L))}^z\cap (C\cap A)^L.$$
From Proposition \ref{constructibles-grado-imagenes-por-lineal:prop}   we conclude:
$$\deg\left(\overline{\pi_2(B(\Omega, L))}^z\right) \leq \deg_{\rm lci}(B(\Omega, L))\leq \deg_{\rm lci}(\Omega) (d+1)^{mL}.$$
Next, as $C:=C_1\cap \cdots \cap C_s$ we may define the locally closed sets:
$$C_{i,j}:=\A^{n(j-1)}\times C_i \times \A^{n(L-j)},$$
whose degree equals $\deg(C_i)$. We also define the class of linear affine varieties $A_j$, $1\leq j \leq L$, of degree $1$, given by:
$$A_{j}:=\A^{n(j-1)}\times A \times \A^{n(L-j)},$$
Then, we have:
$$\overline{\pi_2(B(\Omega, L))}^z\cap (C \cap A)^L= \overline{\pi_2(B(\Omega, L))}^z\cap \left( \bigcap_{i,j} C_{i,j}\right)\cap \left( \bigcap_{j=1}^L A_j\right).$$
From Proposition \ref{grado-constructible-cerrados-multiple:prop} we also conclude:
$$\deg\left(\overline{\pi_2(B(\Omega, L))}^z\cap (C\cap A)^L\right)\leq
\deg_{\rm lci}(B(\Omega, L)) D^{\dim(\pi_2(B(\Omega, L)))},$$
where, as $\deg(A_j)=1$, we have:
$$1\leq \max\{\deg(C_{i,j})\; :\; 1\leq i \leq s, \; 1 \leq j \leq L\}=D=\max\{ \deg(C_i)\; : \; 1\leq i \leq s\}.$$
Putting all together, as $\dim(\pi_2(B(\Omega, L)))\leq \dim(B(\Omega, L))$ we get:
$$\deg\left(\overline{\pi_2(B(\Omega, L))}^z\cap (C\cap A)^L\right)\leq
\deg_{\rm lci}(\Omega) (d+1)^{mL} D^{\dim(\Omega) + (n-m)L}.$$
This yields:
$$\deg\left(\overline{\pi_2(B(\Omega, L))}^z\cap (C\cap A)^L\right)\leq
\deg_{\rm lci}(\Omega) D^{\dim(\Omega)}D^{ (n-m)L}(d+1)^{mL}.$$
As $\delta\geq e^2(d+1)^2$, we have:
$$\deg\left(\overline{\pi_2(B(\Omega, L))}^z\cap (C\cap A)^L\right)\leq
\deg_{\rm lci}(\Omega) D^{\dim(\Omega)}  D^{(n-m)L}\left({{\delta}\over{e^2}}\right)^{{{m}\over{2}} L}.$$
Then, we obtain:
$${{\deg\left(\overline{\pi_2(B(\Omega, L))}^z\cap (C\cap A)^L\right)}\over{\deg(C^L)}}\leq
{{\deg_{\rm lci}(\Omega) D^{\dim(\Omega)}  D^{(n-m)L}\delta^{{{m}\over{2}} L}}\over{e^{mL}\delta^{\codim(C)L}}}.$$
As $\codim(C)\geq (n-m) + m/2 + 1/2$, we also have:
$${{\deg\left(\overline{\pi_2(B(\Omega, L))}^z\cap (C\cap A)^L\right)}\over{\deg(C^L)}}\leq
{{\deg_{\rm lci}(\Omega) D^{\dim(\Omega)} }\over{e^{mL}\delta^{1/2L}}} {{ D^{(n-m)L}}\over{\delta^{(n-m)L}}}.$$
As $D\leq (1+ {1\over{n-m}})\delta$ we get:
$${{\deg\left(\overline{\pi_2(B(\Omega, L))}^z\cap (C\cap A)^L\right)}\over{\deg(C^L)}}\leq
{{\deg_{\rm lci}(\Omega) D^{\dim(\Omega)} }\over{e^{mL}\delta^{1/2L}}} \left(1+ { 1 \over{n-m}}\right)^{(n-m)L}.$$
Hence,
$${{\deg\left(\overline{\pi_2(B(\Omega, L))}^z\cap (C\cap A)^L\right)}\over{\deg(C^L)}}\leq
{{\deg_{\rm lci}(\Omega) D^{\dim(\Omega)} }\over{e^{mL}\delta^{1/2L}}} \left(\left(1+ { 1 \over{n-m}}\right)^{(n-m)}\right)^L.$$
And thus,
$${{\deg\left(\overline{\pi_2(B(\Omega, L))}^z\cap (C\cap A)^L\right)}\over{\deg(C^L)}}\leq
{{\deg_{\rm lci}(\Omega) D^{\dim(\Omega)} }\over{\delta^{1/2L}}} \left({{1}\over{e}}\right)^{(m-1)L}.$$
Again, as  $D\leq (1+ {1\over{n-m}})\delta$ and $L\geq 6 \dim(\Omega)$ we also have:
$${{\deg\left(\overline{\pi_2(B(\Omega, L))}^z\cap (C\cap A)^L\right)}\over{\deg(C^L)}}\leq
{{\deg_{\rm lci}(\Omega) }\over{\delta^{\dim(\Omega)}}}{{1}\over{\delta^{\dim(\Omega)}}}\left({{D}\over{\delta}}\right)^{\dim(\Omega)}
\left({{1}\over{e}}\right)^{(m-1)L}.$$
Namely,
$${{\deg\left(\overline{\pi_2(B(\Omega, L))}^z\cap (C\cap A)^L\right)}\over{\deg(C^L)}}\leq
{{\deg_{\rm lci}(\Omega) }\over{\delta^{\dim(\Omega)}}}{{1}\over{\delta^{\dim(\Omega)}}}e^{{{\dim(\Omega)}\over{n-m}}}
\left({{1}\over{e}}\right)^{(m-1)L}.$$
As $\delta \geq e^2 (d+1)^2$, this yields:
$${{\deg\left(\overline{\pi_2(B(\Omega, L))}^z\cap (C\cap A)^L\right)}\over{\deg(C^L)}}\leq
{{\deg_{\rm lci}(\Omega) }\over{\delta^{\dim(\Omega)}}}\left({{1}\over{e}}\right)^{\dim(\Omega)}
\left({{1}\over{e}}\right)^{(m-1)L}.$$
Finally, as $\log(\delta) \dim(\Omega) \geq 2 \log(\deg_{\rm lci}(\Omega))$ we also get:
\begin{equation}\label{desigualdad-final-CTS:eqn}
{{\deg\left(\overline{\pi_2(B(\Omega, L))}^z\cap (C\cap A)^L\right)}\over{\deg(C^L)}}\leq
{{1}\over{\deg_{\rm lci}(\Omega)}}\left({{1}\over{e}}\right)^{\dim(\Omega)}
\left({{1}\over{e}}\right)^{(m-1)L}.
\end{equation}
Now, we consider the following constructible subset:
$$R^c(\Omega,\Sigma, C, L)=\{ {\bf Q}:=(x_1,\ldots, x_L)\in C^L\; : \; {\hbox {\rm ${\bf Q}$ is not a CTS for $\Omega$ w.r.t. $\Sigma$}}\}.$$
Note that $R^c(\Omega, \Sigma, C, L)= (C)^L\setminus R(\Omega,\Sigma, C, L)$.
\begin{claim}With these notations, if $A\in \G(C)$, we have:
$$\sharp\left(R^c(\Omega,\Sigma, C, L)\cap A^L \right) \leq \deg\left(\overline{\pi_2(B(\Omega, L))}^z\cap (C\cap A)^L\right).$$
\end{claim}
\begin{proof-claim} Observe that
$$R^c(\overline{\Omega}^z,\Sigma, C, L)\subseteq \pi_2(B(\Omega, L))\cap C^L.$$
For if ${\bf Q}:=(x_1,\ldots, x_L)\in R^c(\Omega, \Sigma, C, L)$ there is some list of polynomial equations $f\in \Omega\setminus \Sigma$, such that:
$$f(x_1)= \cdots = f(x_L)=0.$$
Hence, $(f, x_1,\ldots, x_L)\in V(\Omega, L)$. Moreover, there must be some locally irreducible component $W\in \scrD$ of $V(\Omega, L)$ that contains $(f, x_1,\ldots, x_L)$. But $f:=\pi_1(f, x_1,\ldots, x_L)\in \Omega\setminus\Sigma$ and, hence, $W\in \scrC$. Then, $(f, x_1,\ldots, x_L)\in B(\Omega, L)$ and
$${\bf Q}:=(x_1,\ldots, x_L)\in \pi_2(B(\Omega, L)).$$
Then, we conclude:
$$R^c(\overline{\Omega}^z,\Sigma, C, L)\subseteq \pi_2(B(\Omega, L))\cap C^L\subseteq \overline{\pi_2(B(\Omega, L))}^z\cap C^L,$$
and
$$R^c(\overline{\Omega}^z,\Sigma, C, L)\cap A^L\subseteq \overline{\pi_2(B(\Omega, L))}^z\cap (C\cap A)^L$$
As $A\in \G(C)$, then $(A\cap C)$ is a zero-dimensional algebraic variety (i.e. a finite set) and, thus,
$R^c(\overline{\Omega}^z,\Sigma, C, L)\cap A^L$ is also a finite set. Finally, in the zero-dimensional case degree equals cardinality and, hence, the claim follows:
$$\sharp\left(R^c(\Omega,\Sigma, C, L)\cap A^L \right) \leq \deg\left(\overline{\pi_2(B(\Omega, L))}^z\cap (C\cap A)^L\right).$$
\end{proof-claim}\\
Just to conclude our statement, as for every $A\in \G(C)$, $\sharp(A\cap C)= \deg(C)$, then we obtain:
$${\rm Prob}_{(C\cap A)^L}[R]:=1-{\rm Prob}_{(C\cap A)^L}[R^c],$$
where $R^c:=R^c(\Omega, \Sigma, C,L)$. Namely, we have:
 $${\rm Prob}_{(C\cap A)^L}[R] = 1- {{1}\over{\sharp\left(\left(A\cap C\right)^L \right)}}\sum_{\zeta\in (A\cap C)^L}\chi_{R^c}(\zeta)= 1- {{\sharp(R^c\cap (A\cap C)^L)}\over{\sharp\left(\left(A\cap C\right)^L \right) }}.$$
Then, from the previous Claim and Inequality (\ref{desigualdad-final-CTS:eqn}) we conclude:
 $${\rm Prob}_{(C\cap A)^L}[R] \geq 1- {{1}\over{\deg_{\rm lci}(\Omega)}}\left({{1}\over{e^{\dim(\Omega)+ (m-1)L}}}\right),$$
and the Theorem follows.
\end{proof}

In the complex case ($K=\C$) the complex Grassmannian may be equipped with a probability distribution $\mu$ such that for every Borel subset $B\subseteq \G(n, n-r)$ of the Zariski topology the following holds:
\begin{equation}\label{medida-Zariski-Borel:eqn}
\mu[B]:=\left\{\begin{array}{lr}1, & {\hbox {\rm $B$ has non-empty interior in $\G(n,r)$}}\\
0 & {\hbox {\rm otherwise}} \end{array}\right.
\end{equation}
Then, with this measure the previous statement becomes the following:

\begin{corollary} Assuming $K=\C$, a probability distribution $\mu$ on the Borel sets $\scrB(\G(n,n-r))$ that satisfies Identity (\ref{medida-Zariski-Borel:eqn}) above and with the same notations and assumptions as in Theorem \ref{teorema-principal-densidad-questores:teor}, the following inequality holds:
$${{1}\over{\deg_{\rm lci}(C)^L}}E_{\G(n,n-r)}[\sharp_{R}^{(L)}]\geq 1- {{1}\over{\deg_{\rm lci}(\Omega)e^{\dim(\Omega) + (m-1)L}}},
$$
where $E_{\G(n, n-r)}$ means expectation and
$$\begin{matrix}
\sharp_{R}^{(L)}: & \G(n, n-r) & \longrightarrow & \R_+\cup\{\infty\}\\
& A & \longmapsto & \sharp \left( R(\Omega, \Sigma, C, L) \cap A^L\right).\end{matrix}$$
\end{corollary}
\begin{proof} Just note that with the hypothesis of Identity (\ref{medida-Zariski-Borel:eqn}),
$$E_{\G(n,n-r)}[\sharp_{R}^{(L)}]= E_{\G(C)}[\sharp_{R}^{(L)}],$$
where $\G(C)$ is the Zariski open subset of $\G(n,n-r)$ discussed in the proof of the previous Theorem.
And, finally, just apply the previous Theorem on the quotient:
$${{\sharp_{R}^{(L)}(A)}\over{\deg_{\rm lci}(C)^L}}= {\rm Prob}_{(A\cap C)^L}[R].$$
\end{proof}

With a slight modification of the hypothesis of Theorem \ref{teorema-principal-densidad-questores:teor}, we obtain the following result for the case in which $C$  is a globally equi-dimensional constructible set:

\begin{corollary}\label{corolario-constructible-equidimensional-densidad-questores:corol} Let $m,n \in\N$ be two positive integers, with $m\leq n$, and let $(d):=(d_1,\ldots, d_m)$ be a list of degrees and $d:=\max\{d_1,\ldots, d_m\}$. Let $\Sigma\subseteq\Omega$ be two constructible subsets of $\scrP_{(d)}^K$ such that $\Sigma$ has co-dimension at least 1 in $\Omega$. Assume that $\Omega\setminus \Sigma$ satisfies  the following property:
  \begin{equation}\label{dimension-hipotesis-cts-constructible-equidimensional-corol:eqn}
  \forall f:=(f_1,\ldots, f_m)\in \Omega \setminus \Sigma,\; \dim(V_\A(f_1,\ldots, f_m))= n-m.
\end{equation}
Let $C\subseteq \A^n$ be a globally equi-dimensional constructible set of co-dimension $r\geq (n-m)+ m/2 + 1/2$. Assume that there are constructible subsets $C_1,\ldots, C_s \subseteq \A^n$ such that $C:= C_1\cap\ldots\cap C_s$.
Let $D:=\max\{ \deg_{\rm lci}(C_1), \ldots, \deg_{\rm lci}(C_s)\}\geq 1$  be the maximum of their degrees. Let $L\in \N$ be a positive integer and assume that the following properties hold:
\begin{enumerate}
\item $L\geq 6\dim(\Omega)$,
\item $\log(\deg_z\left(C\right)) \geq r\max\{2(1+\log(d+1)), {{2\log(\deg_{\rm lci}(\Omega))+sL}\over{\dim(\Omega)}}\}$,
\item $D\leq \frac{1}{e} (1+ {{1}\over{n-m}})\deg_z(C)^{{{1}\over{\codim(C)}}}$,
\end{enumerate}
Let $R:=R(\Omega, \Sigma, C, L)$ be the constructible set of all sequences ${\bf Q}\in C^L$ of length $L$ which are correct test sequences for $\Omega$ with respect to $\Sigma$. Then, there is a non-empty Zariski open subset $\G(C)\subseteq \G(n, n-r)$ such that for every $A\in \G(C)$, the probability that a randomly chosen list ${\bf Q}\in (C\cap A)^L$ is in $R$ satisfies:
$${\rm Prob}_{(C\cap A)^L}[R]\geq 1- {{1}\over{\deg_{\rm lci}(\Omega)e^{\dim(\Omega) + (m-1)L}}},
$$
where $(A\cap C)^L$ is endowed with its uniform probability distribution.
\end{corollary}
\begin{proof}
With the same notations as in the proof of Theorem \ref{teorema-principal-densidad-questores:teor}, we have that:
$$\overline{\pi_2(B(\Omega, L))}^z\cap (C \cap A)^L= \overline{\pi_2(B(\Omega, L))}^z\cap \left( \bigcap_{i,j} C_{i,j}\right)\cap \left( \bigcap_{j=1}^L A_j\right),$$
where the $C_{i,j}$'s are constructible sets.  From Theorem \ref{cotas-grado-varias-intersecciones:teor} we get:
\begin{equation*}
\begin{split}
 \deg\left(\overline{\pi_2(B(\Omega, L))}^z\cap (C\cap A)^L\right) & \leq {(sL+1) +\dim(\pi_2(B(\Omega, L)))-1 \choose \dim(\pi_2(B(\Omega, L)))}\\
& \deg_{\rm lci}(B(\Omega, L)) D^{\dim(\pi_2(B(\Omega, L)))}.
\end{split}
\end{equation*}
where, as $\deg(A_j)=1$, we have:
$$1\leq \max\{\deg_{\rm lci}(C_{i,j})\; :\; 1\leq i \leq s, \; 1 \leq j \leq L\}=D=\max\{ \deg_{\rm lci}(C_i)\; : \; 1\leq i \leq s\}.$$
Note that if $a,b \in \Z^+$, we have:
$$ {a+b \choose b} \leq e^{a+b}  $$
Thus, putting all together, as $\dim(\pi_2(B(\Omega, L)))\leq \dim(B(\Omega, L))$, we get:
$$\deg\left(\overline{\pi_2(B(\Omega, L))}^z\cap (C\cap A)^L\right)\leq
\deg_{\rm lci}(\Omega) (d+1)^{mL} D^{\dim(\Omega) + (n-m)L} e^{sL+\dim(\Omega)+(n-m)L}.$$
This yields:
$$\deg\left(\overline{\pi_2(B(\Omega, L))}^z\cap (C\cap A)^L\right)\leq
\deg_{\rm lci}(\Omega) (De)^{\dim(\Omega)} (De)^{ (n-m)L}  (d+1)^{mL} e^{sL}.$$
As $\delta\geq e^2(d+1)^2$, we have:
$$\deg\left(\overline{\pi_2(B(\Omega, L))}^z\cap (C\cap A)^L\right)\leq
\deg_{\rm lci}(\Omega) (De)^{\dim(\Omega)} (De)^{ (n-m)L}    \left({{\delta}\over{e^2}}\right)^{{{m}\over{2}} L} e^{sL}.$$
Then, we obtain:
$${{\deg\left(\overline{\pi_2(B(\Omega, L))}^z\cap (C\cap A)^L\right)}\over{\deg_z(C^L)}}\leq
{{\deg_{\rm lci}(\Omega) (De)^{\dim(\Omega)} (De)^{ (n-m)L}   \delta^{{{m}\over{2}} L} e^{sL}   } \over{e^{mL}\delta^{\codim(C)L}}}.$$
As $\codim(C)\geq (n-m) + m/2 + 1/2$, we also have:
$${{\deg\left(\overline{\pi_2(B(\Omega, L))}^z\cap (C\cap A)^L\right)}\over{\deg_z(C^L)}}\leq
{{\deg_{\rm lci}(\Omega) (De)^{\dim(\Omega)}  }\over{e^{mL}\delta^{1/2L}}} {{ (De)^{(n-m)L}  }\over{\delta^{(n-m)L}}} e^{sL}.$$
As $ D\leq \frac{1}{e} (1+ {1\over{n-m}})\delta$ we get:
$${{\deg\left(\overline{\pi_2(B(\Omega, L))}^z\cap (C\cap A)^L\right)}\over{\deg_z(C^L)}}\leq
{{\deg_{\rm lci}(\Omega) (De)^{\dim(\Omega)} }\over{e^{mL}\delta^{1/2L}}} \left(1+ { 1 \over{n-m}}\right)^{(n-m)L} e^{sL}.$$
Hence,
$${{\deg\left(\overline{\pi_2(B(\Omega, L))}^z\cap (C\cap A)^L\right)}\over{\deg_z(C^L)}}\leq
{{\deg_{\rm lci}(\Omega) (De)^{\dim(\Omega)} }\over{e^{mL}\delta^{1/2L}}} \left(\left(1+ { 1 \over{n-m}}\right)^{(n-m)}\right)^L e^{sL}.$$
And thus,
$${{\deg\left(\overline{\pi_2(B(\Omega, L))}^z\cap (C\cap A)^L\right)}\over{\deg_z(C^L)}}\leq
{{\deg_{\rm lci}(\Omega) (De)^{\dim(\Omega)} }\over{\delta^{1/2L}}} \left({{1}\over{e}}\right)^{(m-1)L} e^{sL}.$$
Again, as  $D\leq \frac{1}{e}(1+ {1\over{n-m}})\delta$ and $L\geq 6 \dim(\Omega)$ we also have:
$${{\deg\left(\overline{\pi_2(B(\Omega, L))}^z\cap (C\cap A)^L\right)}\over{\deg_z(C^L)}}\leq
{{\deg_{\rm lci}(\Omega) }\over{\delta^{\dim(\Omega)}}}{{1}\over{\delta^{\dim(\Omega)}}}\left({{De}\over{\delta}}\right)^{\dim(\Omega)}
\left({{1}\over{e}}\right)^{(m-1)L} e^{sL}.$$
Namely,
$${{\deg\left(\overline{\pi_2(B(\Omega, L))}^z\cap (C\cap A)^L\right)}\over{\deg_z(C^L)}}\leq
{{\deg_{\rm lci}(\Omega) }\over{\delta^{\dim(\Omega)}}}{{1}\over{\delta^{\dim(\Omega)}}}e^{{{\dim(\Omega)}\over{n-m}}}
\left({{1}\over{e}}\right)^{(m-1)L} e^{sL}.$$
As $\delta \geq e^2 (d+1)^2$, this yields:
$${{\deg\left(\overline{\pi_2(B(\Omega, L))}^z\cap (C\cap A)^L\right)}\over{\deg_z(C^L)}}\leq
{{\deg_{\rm lci}(\Omega) }\over{\delta^{\dim(\Omega)}}}\left({{1}\over{e}}\right)^{\dim(\Omega)}
\left({{1}\over{e}}\right)^{(m-1)L} e^{sL}.$$
Finally, as $\log(\delta) \dim(\Omega) \geq 2 \log(\deg_{\rm lci}(\Omega))+sL$ we also get:
$${{\deg\left(\overline{\pi_2(B(\Omega, L))}^z\cap (C\cap A)^L\right)}\over{\deg_z(C^L)}}\leq
{{1}\over{\deg_{\rm lci}(\Omega)}}\left({{1}\over{e}}\right)^{\dim(\Omega)}
\left({{1}\over{e}}\right)^{(m-1)L}.$$
The final part of the proof is identical to that of Theorem \ref{teorema-principal-densidad-questores:teor}.
\end{proof}

Using Corollary \ref{intereseccion-lineales-no-equi-dimensional:corol}, we can generalize the previous Corollary for any constructible set as follows:

\begin{corollary}\label{corolario-constructible-densidad-questores:corol} Let $m,n \in\N$ be two positive integers, with $m\leq n$, and let $(d):=(d_1,\ldots, d_m)$ be a list of degrees and $d:=\max\{d_1,\ldots, d_m\}$. Let $\Sigma\subseteq\Omega$ be two constructible subsets of $\scrP_{(d)}^K$ such that $\Sigma$ has co-dimension at least 1 in $\Omega$. Assume that $\Omega\setminus \Sigma$ satisfies  the following property:
  \begin{equation}\label{dimension-hipotesis-cts-constructible-corol:eqn}
  \forall f:=(f_1,\ldots, f_m)\in \Omega \setminus \Sigma,\; \dim(V_\A(f_1,\ldots, f_m))= n-m.
\end{equation}
Let $C\subseteq \A^n$ be a constructible set of co-dimension $r\geq (n-m)+ m/2 + 1/2$ and let $V_1, \ldots, V_k$ be the irreducible components of higher dimension of the Zariski closure of $C$ (as in Proposition \ref{componentes-alta-dimension:def}). Assume that there are constructible subsets $C_1,\ldots, C_s \subseteq \A^n$ such that $C:= C_1\cap\ldots\cap C_s$.
Let $D:=\max\{ \deg_{\rm lci}(C_1), \ldots, \deg_{\rm lci}(C_s)\}\geq 1$  be the maximum of their degrees. Let $L\in \N$ be a positive integer and assume that the following properties hold:
\begin{enumerate}
\item $L\geq 6\dim(\Omega)$,
\item $\log(\sum_{i=1}^k \deg(V_{i})) \geq r\max\{2(1+\log(d+1)), {{2\log(\deg_{\rm lci}(\Omega))+sL}\over{\dim(\Omega)}}\}$,
\item $D\leq \frac{1}{e} (1+ {{1}\over{n-m}})(\sum_{i=1}^k \deg(V_{i}))^{{{1}\over{\codim(C)}}}$,
\end{enumerate}
 Let $R:=R(\Omega, \Sigma, C, L)$ be the constructible set of all sequences ${\bf Q}\in C^L$ of length $L$ which are correct test sequences for $\Omega$ with respect to $\Sigma$.  Then, there is a non-empty Zariski open subset $\G(C)\subseteq \G(n, n-r)$ such that for every $A\in \G(C)$, the probability that a randomly chosen list ${\bf Q}\in (C\cap A)^L$ is in $R$ satisfies:
$${\rm Prob}_{(C\cap A)^L}[R]\geq 1- {{1}\over{\deg_{\rm lci}(\Omega)e^{\dim(\Omega) + (m-1)L}}},
$$
where $(A\cap C)^L$ is endowed with its uniform probability distribution.
\end{corollary}
\begin{proof}
The proof is similar to that of Corollary \ref{corolario-constructible-equidimensional-densidad-questores:corol}. In this case, we take:
$$\delta := (\sum_{i=1}^k \deg(V_{i}))^{{{1}\over{\codim(C)}}},$$
and, from Corollary \ref{intereseccion-lineales-no-equi-dimensional:corol}, we have that there is a non-empty Zariski open subset $\G(C) \subseteq \G(n, n-r)$ such that for all $A \in \G(C)$ we have $\sharp (C \cap A) = \sum_{i=1}^k \deg(V_{i})$.
\end{proof}

If we assume that $C$ is a complete intersection algebraic variety, we obtain the following statement:

\begin{corollary}\label{densidad-questores-variedades-ic:corol} Let $m,n \in\N$ be two positive integers, with $m\leq n$, and let $(d):=(d_1,\ldots, d_m)$ be a list of degrees and $d:=\max\{d_1,\ldots, d_m\}$.  Let $\Sigma\subseteq\Omega$ be two constructible subsets of $\scrP_{(d)}(X_1,\ldots, X_n)$ such that $\Sigma$ has co-dimension at least 1 in $\Omega$.  Assume that $\Omega\setminus \Sigma$ satisfies the following property:
  \begin{equation}\label{dimension-hipotesis-cts-corol:eqn}
  \forall f:=(f_1,\ldots, f_m)\in \Omega \setminus \Sigma, \; \dim(V_\A(f_1,\ldots, f_m))= n-m,
\end{equation}
Let $V:=V_\A(h_1,\ldots, h_r)\subseteq \A^n$ be a complete intersection algebraic variety of co-dimension $r\geq (n-m)+ m/2 + 1/2$ such that $\deg(V)\geq \delta^r$, where $\delta:=\min\{\deg(h_1),\ldots, \deg(h_r)\}$.
Let $L\in \N$ be a positive integer and assume that the following properties hold:
\begin{enumerate}
\item $L\geq 6\dim(\Omega)$,
\item $\log(\delta) \geq \max\{2(1+\log(d+1)), {{2\log(\deg_{\rm lci}(\Omega))}\over{\dim(\Omega)}}\}$,
\item $\max\{ \deg(h_1), \ldots, \deg(h_r)\}\leq(1+ {{1}\over{n-m}})\delta$.
\end{enumerate}
Let $R:=R(\Omega, \Sigma, V, L)$ be the constructible set of all sequences ${\bf Q}\in V^L$ of length $L$ which are correct test sequences for $\Omega$ with respect to $\Sigma$. Then, there is a non-empty Zariski open subset $\G(V)\subseteq \G(n, n-r)$ such that for every $A\in \G(V)$ the probability that a randomly chosen list ${\bf Q}\in (V\cap A)^L$ is in $R$ satisfies:
$${\rm Prob}_{(V\cap A)^L}[R]\geq 1- {{1}\over{\deg_{\rm lci}(\Omega)e^{\dim(\Omega) + (m-1)L}}},
$$
where $(A\cap V)^L$ is endowed with its uniform probability distribution.
\end{corollary}
\begin{proof} This is a straight-forward application of Theorem \ref{teorema-principal-densidad-questores:teor}, just taking $C=V$, $s=r$ and $C_i:=V_\A(h_i)$ for every $i$, $1\leq i \leq s$, we obtain the statement.\end{proof}

We may state the previous Corollary in the case we search for correct test sequences in zero-dimensional varieties:

\begin{corollary}\label{densidad-questores-zero-dimensional:corol}  Let $m,n \in\N$ be two positive integers, with $m\leq n$, and let $(d):=(d_1,\ldots, d_m)$ be a list of degrees and $d:=\max\{d_1,\ldots, d_m\}$.  Let $\Sigma\subseteq\Omega$ be two constructible subsets of $\scrP_{(d)}(X_1,\ldots, X_n)$ such that $\Sigma$ has co-dimension at least 1 in $\Omega$.  Assume that $\Omega\setminus \Sigma$ satisfies the following property:
$$  \forall f:=(f_1,\ldots, f_m)\in \Omega \setminus \Sigma, \; \dim(V_\A(f_1,\ldots, f_m))= n-m,
$$
Let $V:=V_\A(h_1,\ldots, h_n)\subseteq \A^n$ be a zero-dimensional algebraic variety given by polynomial equations of the same degree $\delta:=\deg(h_i)$, $1\leq i \leq n$. Assume that $\deg(V)= \delta^n$.
Let $L\in \N$ be a positive integer and assume that the following properties hold:
\begin{enumerate}
\item $L\geq 6\dim(\Omega)$,
\item $\log(\delta) \geq \max\{2(1+\log(d+1)), {{2\log(\deg_{\rm lci}(\Omega))}\over{\dim(\Omega)}}\}$,
\end{enumerate}
Let $R:=R(\Omega, \Sigma, V, L)$ be the constructible set of all sequences ${\bf Q}\in V^L$ of length $L$ which are correct test sequences for $\Omega$ with respect to $\Sigma$. Assume that $V$ is endowed with its uniform probability distribution. Then, we have:
$${\rm Prob}_{V^L}[R]\geq 1- {{1}\over{\deg_{\rm lci}(\Omega)e^{\dim(\Omega) + (m-1)L}}}.
$$
\end{corollary}

Corollary \ref{densidad-questores-variedades-ic:corol} gives sufficient conditions to have complete intersection varieties which are evasive to any system of polynomial equations in a given constructible set (as in \cite{Dvir-Kollar}). Given a degree list $(d):=(d_1,\ldots, d_m)$ and a constructible set $\widetilde{\Omega}\subseteq \scrP_{(d)}$, we say that a variety $V\subseteq \A^n$ is evasive with respect to $\widetilde{\Omega}$ if the following holds:
$$\forall f=(f_1,\ldots, f_m)\in \widetilde{\Omega}, \;\; \dim(V\setminus V_\A(f_1,\ldots, f_m))=\dim(V).$$
In the case $V$ is equi-dimensional, the previous property simply means that for all system $f\in\widetilde{\Omega}$ there is some irreducible component of $V$ which is not completely embedded in $V_\A(f_1,\ldots, f_m)$ or, equivalently, that there is always an irreducible component that evades any variety defined by equations in $\widetilde{\Omega}$.

\begin{corollary}[{\bf Evasive varieties for lists of equations in a constructible set}]\label{evesivas-variedades-ic:corol} Let $m,n \in\N$ be two positive integers, with $m\leq n$, and let $(d):=(d_1,\ldots, d_m)$ be a list of degrees and $d:=\max\{d_1,\ldots, d_m\}$.  Let $\Sigma\subseteq\Omega$ be two constructible subsets of $\scrP_{(d)}(X_1,\ldots, X_n)$ such that $\Sigma$ has co-dimension at least 1 in $\Omega$.  Assume that $\Omega\setminus \Sigma$ satisfies the property described in Equation (\ref{dimension-hipotesis-cts-corol:eqn}).\\
Let $V\subseteq \A^n$ be  any complete intersection variety such that $V:=V_\A(h_1,\ldots, h_r)\subseteq \A^n$  of co-dimension $r\geq (n-m)+ m/2 + 1/2$ and $\deg(V)\geq \delta^r$, where $\delta:=\min\{\deg(h_1),\ldots, \deg(h_r)\}$. Assume that the following properties hold:
\begin{enumerate}
\item $\log(\delta) \geq \max\{2(1+\log(d+1)), {{2\log(\deg(\Omega))}\over{\dim(\Omega)}}\}$ and,
\item $\max\{ \deg(h_1), \ldots, \deg(h_r)\}\leq(1+ {{1}\over{n-m}})\delta$.
\end{enumerate}
Then, $V$ evades $\Omega\setminus\Sigma$.
\end{corollary}
\begin{proof} Suppose that $V$ does not evade $\Omega\setminus\Sigma$. Then, there is some $f:=(f_1,\ldots, f_m)\in \Omega\setminus \Sigma$ such that:
\begin{equation}\label{evasive-dimension-prueba:eqn}
\dim(V\setminus V_\A(f_1,\ldots, f_m))<\dim(V).
\end{equation}
As $V$ is complete intersection, all irreducible components of $V$ have the same  dimension $n-r$. Hence, Inequality (\ref{evasive-dimension-prueba:eqn}) implies that all irreducible components of $V$ are included in $V_\A(f_1,\ldots, f_m)$.\\
By the same reason, there is a non-empty Zariski open subset $\G_1(V) \subseteq \G(n, n-r)$ such that
for all $A\in \G_1(V)$ and for every irreducible component $W$ of $V$, $\sharp(W\cap A) = \deg(W)$.\\
Let $\G(V)\subseteq \G(n,n-r)$ be the non-empty Zariski open subset of Corollary \ref{densidad-questores-variedades-ic:corol} above, then $\G_1(V)\cap \G(V)$ is also a non-empty Zariski open subset of $\G(n,n-r)$. Finally, for any positive integer number $L\in \N$, such that $L\geq 6\dim(\Omega)$, and any $A\in \G_1(V)\cap \G(V)$ we would have that the following properties hold:
\begin{itemize}
\item For every ${\bf Q}\in (A\cap V)^L$, ${\bf Q}$ is not a correct test sequence of $\Omega$ with respect to $\Sigma$. This holds because $f$ vanishes identically in ${\bf Q}$. In particular, with the same notations as above, we would have:
    $$\sharp\left(R(\Omega,\Sigma, V, L)\cap A^L\right)=0.$$
\item From Corollary \ref{densidad-questores-variedades-ic:corol}, as all its hypothesis hold, we also have:
    $${{\sharp\left(R(\Omega,\Sigma, V, L)\cap A^L\right)}\over{\deg(V)^L}}\geq 1- {{1}\over{\deg_{\rm lci}(\Omega)e^{\dim(\Omega) + (m-1)L}}}>0.$$
\end{itemize}
We have then arrived to a contradiction and no such $f\in\Omega\setminus\Sigma$ exist.
\end{proof}
A customary usage of correct test sequences is the case of constructible sets given as polynomial images of some affine space (i.e. the parameter space). We may also discuss our statements in these cases.

\begin{corollary}[{\bf Unirational families of lists of polynomials}]\label{questores-constructibles-uniracionales:corol}
With the same notations as above, let $(d):=(d_1,\ldots, d_m)$ be a degree list, $d:=\max\{ d_1,\ldots, d_m\}$ and $\Omega\subseteq \scrP_{(d)}$ an unirational constructible set, given as the image of a polynomial mapping:
$$\varphi:=(\varphi_1,\ldots, \varphi_N): W\subseteq \A^{M}\longrightarrow \Omega\subseteq \scrP_{(d)},$$
where  $N=N_{(d)}=\dim(\scrP_{(d)})$, $\Omega=\varphi(W)$ and $W$ is a constructible set of dimension $s$. Let $t:=s/\dim(\Omega)\geq 1$ be the quotient between the dimensions of $W$ and $\Omega$ and
 let  $D:=\max\{ \deg(\varphi_i) \; : \; 1\leq i \leq N\}$ be the maximum of the degrees of the coordinates of $\varphi$.  Let $\Sigma\subseteq \Omega$ be a constructible subset of co-dimension at least  $1$ in $\Omega$ and assume that the following property holds:
 $$
 \forall f:=(f_1,\ldots, f_m)\in \Omega \setminus \Sigma, \; \dim(V_\A(f_1,\ldots, f_m))= n-m,
$$

Let $V:=V_\A(h_1,\ldots, h_r)\subseteq \A^n$ be a complete intersection algebraic variety of co-dimension $r\geq (n-m)+ m/2 + 1/2$ such that $\deg(V)\geq \delta^r$, where $\delta:=\min\{\deg(h_1),\ldots, \deg(h_r)\}$. Let $L\in \N$ be a positive integer and assume that the following properties hold:
\begin{enumerate}
\item $L\geq 6s$,
\item $\log(\delta) \geq \max\{2(1+\log(d+1)), 2t\left({{\log(\deg_{\rm lci}(W))}\over{s}} +  \log(D)\right)\}$,
\item $\max\{ \deg(h_1), \ldots, \deg(h_r)\}\leq(1+ {{1}\over{n-m}})\delta$.
\end{enumerate}
Let $R:=R(\Omega, \Sigma, V, L)$ be the constructible set of all sequences ${\bf Q}\in V^L$ of length $L$ which are correct test sequences for $\Omega$ with respect to $\Sigma$. Then, there is a non-empty Zariski open subset $\G(V)\subseteq \G(n, n-r)$ such that for every $A\in \G(V)$ the probability that a randomly chosen list ${\bf Q}\in (V\cap A)^L$ is in $R$ satisfies:
$${\rm Prob}_{(V\cap A)^L}[R]\geq 1- {{1}\over{\deg_{\rm lci}(W)e^{s(\log(D)+1) + (m-1)L}}},
$$
where $(A\cap V)^L$ is endowed with its uniform probability distribution.
\end{corollary}
\begin{proof} Obvious since correct test sequences are hereditary (cf. Proposition \ref{hereditary-cts:prop})  and since the following inequality holds because of Corollary \ref{grado-imagen-constructible:corol}:
$$\deg_z(\Omega)= \deg\left(\overline{\varphi(W)}^z\right)\leq \deg_{\rm lci}(W) D^s.$$
\end{proof}

\begin{remark}
Note that if $W$ is a linear affine subvariety and the dimension does not decrease too much through $\varphi$ (for instance, $t:={{s}\over{\dim(\Omega)}}\leq 2$), the hypothesis of the previous Corollary become:
\begin{enumerate}
\item $L\geq 6s$,
\item $\log(\delta) \geq \max\{2(1+\log(d+1)), 4\left(\log(D)\right)\}$,
\item $\max\{ \deg(h_1), \ldots, \deg(h_r)\}\leq(1+ {{1}\over{n-m}})\delta$,
\end{enumerate}
And the thesis thus become the same:
$${\rm Prob}_{(V\cap A)^L}[R]\geq 1- {{1}\over{e^{s(\log(D)+1) + (m-1)L}}},
$$
\end{remark}

This is simply devoted to emphasize how the previous probability results apply to prove that correct test sequences of optimal length are highly dense in Demillo-Lipton-Schwartz-Zippel sampling set.
We return to Corollary \ref{DeMillo-Lipton-Schwartz-Zippel:corol}.  As main application of that Corollary, we consider a finite subset $Q\subseteq \kappa$,  a non-zero polynomial $f\in P_d^K$ and the hyper-surface $C:=V_\A(f)\subseteq \A^n(K)$, where $K$ is the algebraic closure of $\kappa$.
As in  \cite{DeMilloLipton}, \cite{Zippel} and  \cite{Schwartz}, we proved that the probability that $f$ does not vanishes at a random point $x\in Q^n$ is greater than:
$$1- {{d}\over{\sharp(Q)}}.$$
Thus, taking $Q\subseteq \kappa$ such that $\sharp(Q)\geq (2(d+1))^2$, we can produce a probabilistic algorithm for testing equality to zero for polynomials in the \emph{dense input case}: Inputs moving freely in $P_d^K$.

 Now, we shall use Corollary \ref{densidad-questores-zero-dimensional:corol} to prove that correct test sequences are highly distributed inside the sampling set $Q^n$ with small changes in the cardinality of $Q$.

\begin{corollary}\label{Zippel-Schwartz-repleto-CTS:corol}
Let $n, d \in\N$ be two positive integers.  Let $\Omega$ be a constructible subset of $P_d^K:=P_{d}^K(X_1,\ldots, X_n)$. Assume that $\Omega\setminus \{0\}\not=\emptyset$.

Let $Q\subseteq \kappa$ be a finite set.
Let $L\in \N$ be a positive integer such that the following properties hold:
\begin{enumerate}
\item $L\geq 6\dim(\Omega)$,
\item $\sharp(Q) \geq \max\{(2(d+1))^2, \deg_{\rm lci}(\Omega)^{{2}\over{\dim(\Omega)}}\}$.
\end{enumerate}
Then, the probability that a random sequence ${\bf Q}\in V^L=\left(Q^n\right)^L$ is a correct test sequence for $\Omega$ is greater than:
$$1- {{1}\over{\deg_{\rm lci}(\Omega)e^{\dim(\Omega)}}}.
$$
In the dense input case (i.e. when $\Omega=P_d^K$), for any $L\geq 6\dim(P_d^K)$ and $Q\subseteq \kappa$ such that:
$$\sharp(Q)\geq (2(d+1))^2,$$
the probability that a random point ${\bf Q}\in \left(Q^n\right)^L$ is a correct test sequence for $P_d^K$ is greater than:
$$1- {{1}\over{e^{{{d+n}\choose{n}}}}}.
$$
\end{corollary}

\section{A ${\bf BPP}_K$ algorithm for detecting generic dimension: ``Suite S\'ecante''}\label{suite-secante:sec}

In this section we exhibit an algorithmic application of the usage of correct test sequences that \emph{differs form the usual zero tests: ``Suite S\'ecante'' Problem.} The problem may be stated as follows:

Let $K$ be an algebraically closed field, $m,n \in \N$ two positive integers with $m\leq n$. Let $(d):=(d_1,\ldots, d_m)$ be a list of degrees and $\scrP_{(d)}:=\scrP_{(d)}^K(X_1,\ldots, X_n)$ the $K-$vector space of lists $f:=(f_1,\ldots,f_m)$ of polynomials $f_i\in K[X_1,\ldots, X_n]$ such that $\deg(f_i)\leq d_i$ for each $i$, $1\leq i \leq m$.  We denote by $N_{(d)}$ the dimension of $\scrP_{(d)}$.

\begin{problem}[{\bf ``Suite S\'ecante'' Problem}]
 Design an algorithm that performs the following task:\\
Given $f\in \scrP_{(d)}$ decide whether $f$ is a ``Suite S\'ecante''. Namely, decide whether the following property holds:\\
The algebraic variety $V_\A(f_1,\ldots, f_m)\subseteq \A^n$ is a non-empty variety of dimension $n-m$.
\end{problem}

Some authors would prefer to say that $V_\A(f_1,\ldots, f_m)$ is a set-theoretical complete intersection variety with defining equations $\{f_1,\ldots, f_m\}$.

We study not only the algorithm but also its complexity. The Model of Computation will be that of \emph{Turing Machines over $K$} (in the sense of \cite{BlCuShSm:book} and references therein). We distinguish between the \emph{generic case} and the \emph{restricted input case}, where the input lists of polynomials will belong to some constructible subset $\Omega\subseteq \scrP_{(d)}$. We take into account the number of arithmetic operations required to evaluate the polynomials in the list $f:=(f_1,\ldots, f_m)$. This is represented by \emph{non-scalar straight-line program} encoding of the input list $f$. Note that this encoding is equivalent to \emph{neural networks with polynomial activation functions}. See Sections 2 and 3 of \cite{Krick-Pardo96} for details about non-scalar straight-line programs.

Many works in \emph{Computational Algebraic Geometry} have deal with this problem. It will be too long to cite all of them. \emph{The main novelty of our contribution is that we do not perform any operation (neither algebraic nor numeric) with the input polynomials in the list $f$: We just evaluate them at some given points of a correct test sequence} (see Algorithm \ref{suite-secante-algoritmo:algo} below).

We begin our discussion by stating some facts about ``Suites S\'ecantes''. We include proofs in order to be self-contained.

We first consider the following incidence variety $V_{(d)}\subseteq \scrP_{(d)}\times \A^n(K)$:
\begin{equation}\label{variedad-incidencia:eqn}
V_{(d)}:=\{ (f,x)\in \scrP_{(d)}\times \A^n(K) \; : \; f(x)=0\in \A^m(K)\}.
\end{equation}

We may see this incidence variety as a fiber of the natural evaluation morphism. Namely, we consider the evaluation morphism defined as follows:
\begin{equation}\label{morfismo-evaluacion:eqn}
\begin{matrix}
{\rm ev}_{(d)}: & \scrP_{(d)}\times \A^n(K) & \longrightarrow & \A^m(K)\\
& (f,x) & \longmapsto & f(x).
\end{matrix}
\end{equation}
Obviously, $V_{(d)}:={\rm ev}_{(d)}^{-1}(\{0\})$. We additionally have two canonical projections defined as follows:
\begin{itemize}
\item The projection $\pi_1: V_{(d)}\longrightarrow \scrP_{(d)}$ given by $\pi_1(f,x):=f$, $\forall (f,x)\in V_{(d)}$.
\item The projection $\pi_2: V_{(d)}\longrightarrow \A^n(K)$ given by $\pi_2(f,x):=x$, $\forall (f,x)\in V_{(d)}$.
\end{itemize}
Observe that for every $f\in \scrP_{(d)}$, the fiber $\pi_1^{-1}(\{f\})$ may be identified with the algebraic variety $V_\A(f)$ of all zeros in $\A^n(K)$.

The following statement contains the main properties of this incidence variety.

\begin{theorem}\label{variedad-incidencia-propiedades:teor}
With these notations and assumptions, we have:
\begin{enumerate}
\item The incidence variety $V_{(d)}$ is a smooth complete intersection variety of dimension $N_{(d)}+ (n-m)$ and it is unmixed (i.e. all its irreducible components have the same dimension).
\item The canonical projection $\pi_2$ is onto and the fiber $\pi_2^{-1}(\{x\})$ is a linear affine variety of dimension $N_{(d)}-m$ contained in $V_{(d)}$ and isomorphic to a linear affine variety of $\scrP_{(d)}$ of the same dimension.
\item There is a Zariski open subset $U_\emptyset\subseteq \scrP_{(d)}$ contained in $\pi_1(V_{(d)})$ or, equivalently, there is a closed proper algebraic subvariety $W_\emptyset\subseteq \scrP_{(d)}$, such that for all $f\in \scrP_{(d)}\setminus W_\emptyset = U_\emptyset$, $\pi_1^{-1}(\{f\})= V_\A(f)\not=\emptyset$.
\item The canonical projection $\pi_1$ is dominant on each irreducible component $C$ of $V_{(d)}$ such that $\pi_1(C)$ contains a ``Suite S\'ecante''.  In fact, for every irreducible component $C$ of $V_{(d)}$ we have just two possible cases:
    \begin{itemize}
    \item Either there is $f\in \pi_1(C)$ such that $\dim(\pi_1^{-1}(\{f\})\cap C)=n-m$, in which case $\pi_1(C)$ is dense in $\scrP_{(d)}$ for the Zariski topology.
    \item Or for all $f\in \pi_1(C)$, $\dim(\pi_1^{-1}(\{ f\})\cap C) \geq n-m+1$ and the Zariski closure of $\pi_1(C)$ is strictly included in $\scrP_{(d)}$. In this case, $f_1,\ldots, f_m$ is not a ``Suite S\'ecant''.
    \end{itemize}
\end{enumerate}
In conclusion, there is a Zariski open subset $U\subseteq \scrP_{(d)}$ such that for every $f=(f_1,\ldots, f_m)\in U$, $V_\A(f)\not=\emptyset$ and it is an algebraic variety of dimension $n-m$. Equivalently, for
all $f=(f_1,\ldots, f_m)\in U$, the sequence $f_1,\ldots, f_m$ is a ``Suite S\'ecante''.

\end{theorem}
\begin{proof} We first consider the evaluation map ${\rm ev}_{(d)}$. It is easy to observe that ${\rm ev}_{(d)}$ is onto since given $v\in \A^m(K)$ and given $f\in \pi_2^{-1}(\{0\})$, then $f-v\in \scrP_{(d)}$
and ${\rm ev}_{(d)}(f-v, 0)=v$.

By the Theorem on the Dimension of the Fibers (cf. \cite{Shafarevich}, for instance), as $\scrP_{(d)}\times \A^n(K)$ is irreducible, there is a Zariski open subset $\scrU\subseteq \A^m(K)$ such that for all $v\in \scrU$ the non-empty fiber ${\rm ev}_{(d)}^{-1}(\{ v\})$ satisfies:
$$\dim\left({\rm ev}_{(d)}^{-1}(\{ v\})\right)= \dim\left(\scrP_{(d)}\times \A^n(K)\right)-m = N_{(d)}+ (n-m).$$
Moreover, all fibers are identifiable by the obvious translation:
$$\begin{matrix}
{\rm ev}_{(d)}^{-1}(\{ v\}) & \longrightarrow & {\rm ev}_{(d)}^{-1}(\{ 0\})\\
(f, x) & \longmapsto & (f-v, x).
\end{matrix}$$
Thus, we conclude that $V_{(d)}={\rm ev}_{(d)}^{-1}(\{ 0\})$ is an algebraic variety of dimension $N_{(d)}+ (n-m)$.
We consider two sets of variables to represent the elements in $\scrP_{(d)}\times \A^n(K)$. On one side, we consider the set of variables $\{X_1,\ldots, X_n\}$ that represent the coordinates of the points in $\A^n(K)$. On the other side, we consider the family of generic coefficients of polynomials in $\scrP_{(d)}$:
$$\{ U_{\mu}^{(i)} \; : \; 1\leq i \leq m, \; \vert \mu \vert \leq d_i\}.$$
Thus, generic polynomials in $P_{d_i}$ have the following descriptions:
$$F_i:=\sum_{\vert \mu \vert \leq d_i} U_{\mu}^{(i)} X_1^{\mu_1} \cdots X_n^{\mu_n}.$$
Then, the equations defining $V_{(d)}$ may be understood as the following $m$ equations in these two sets of variables:
$$F_i\left( (U_{\mu}^{(i)}\; : \; 1\leq i \leq m, \vert\mu \vert\leq d_i), X_1,\ldots, X_n\right)=0, \; 1\leq i \leq m,$$
where $F_i\in K[(U_{\mu}^{(i)}\; : \; 1\leq i \leq m, \vert\mu \vert\leq d_i), X_1,\ldots, X_n]$ are polynomials of degree $d_i+1$ in these two sets of variables. We may also consider the Jacobian matrix $D(F):=D(F_1,\ldots, F_m)$
 of these equations with respect to all variables involved. This Jacobian matrix contains an $m\times m$ identity matrix:
 $$\left(\begin{matrix} {{\partial F_1}\over{\partial U_{(0)}^{(1)}}} & \cdots & {{\partial F_1}\over{\partial U_{(0)}^{(m)}}}\\
 \vdots & \ddots & \vdots\\
 {{\partial F_m}\over{\partial U_{(0)}^{(1)}}} & \cdots & {{\partial F_m}\over{\partial U_{(0)}^{(m)}}}
  \end{matrix}\right)=\left(\begin{matrix} 1 & \cdots & 0\\
 \vdots & \ddots & \vdots\\
 0 & \cdots & 1 \end{matrix}\right)=Id_m,$$
 where $(0)=(0, \ldots, 0)\in \N^m$ is the exponent of the independent term. Hence, $\rank(D(F))=m$ and it is independent of the point $(f,x)\in V_{(d)}$ under consideration. By the Jacobian Criterion Theorem, we conclude that the variety $V_{(d)}$ is smooth and the tangent space $T_{(f,x)}V_{(d)}$ is of dimension $N_{(d)}+ (n-m)$. This proves Claim $i)$.

 As for Claim $ii)$ we have simply to observe that for every $x\in \A^n(K)$ the following is a vector subspace of co-dimension $1$ of $P_{d_i}$:
 $$\{ f \in P_{d_i}\; : \; f(x)=0\},$$
and the rest of the Claim immediately follows.

For Claim $iii)$ we recall \emph{Generalised Pham systems} as in  \cite{Corujo} and references therein.
Let $f:=(f_1,\ldots, f_m)\in \scrP_{(d)}$ be a sequence of polynomials.  We just work on the Zariski open subset $\scrW_{(d)}\subseteq \scrP_{(d)}$  of all sequences of polynomials $f=(f_1,\ldots, f_m)\in \scrP_{(d)}$ such that $\deg(f_i)=d_i$ for each $i$, $1 \leq i \leq m$. Given $f\in \scrW_{(d)}$, for every $i$, $1\leq i \leq m$, let us consider the highest degree homogeneous component of $f_i$. Namely, if:
$$f_i:=\sum_{\vert \mu \vert \leq d_i} a_{\mu}^{(i)} X_1^{\mu_1}\cdots X_n^{\mu_n},$$
then its homogeneous part of degree $d_i$ is given by:
$$(f_i)_{d_i}:=\sum_{\vert \mu \vert =d_i}a_{\mu}^{(i)} X_1^{\mu_1}\cdots X_n^{\mu_n}.$$

Let us now consider a matrix $M\in \scrM_{(n-m)\times n}(K)$
in  the Zariski open set $G((n-m)\times n, K)$ of all matrices in $\scrM_{(n-m)\times n}(K)$ of rank $n-m$ and a vector $b\in \A^{n-m}(K)$. They define a sequence of degree $1$ polynomials $\ell_1,\ldots, \ell_m: \A^n(K) \longrightarrow \A^{n-m}(K)$ by the following identity:
$$\left(\begin{matrix} \ell_1(x)\\
\vdots\\
\ell_{n-m}(x)\end{matrix}\right):= M \left( \begin{matrix} x_1\\ \vdots \\ x_n\end{matrix}\right) + b.$$

Then, we have an enlarged system $F:=(f_1,\ldots, f_m, \ell_1, \ldots, \ell_{n-m})$ of polynomial equations. We may then consider the highest degree homogeneous components of all of them:
$$F_{\rm init}:=( (f_1)_{d_1}, \ldots, (f_m)_{d_m}, (\ell_1)_{1}, \ldots, (\ell_{n-m})_{1}).$$
This is a family of homogeneous polynomials of respective degrees determined by the list $(\delta):=(d_1,\ldots, d_m, 1, \ldots, 1)\in \N^n$, in the set of variables $\{X_1,\ldots, X_n\}$.
As in  \cite{Corujo}, we say that $F$ is a \emph{Generalized Pham system} if the set of common zeros in the projective space $\mathbb{P}_{n-1}(K)$ of the polynomials in the list $F_{\rm init}$ is empty. Hence,
the classical multivariate resultant theory (cf.  \cite{Jouanolou}, \cite{Cattani}, \cite{Fernandez-Pardo}, \cite{Buse-J} and references therein) implies that there exists a non-zero polynomial in the coefficients $U_{\mu}^{(i)}$ of the polynomials $(f_i)_{d_i}$ and the coordinates $T_{r,s}$ of the matrices $M\in \scrM_{(n-m)\times n}(K)$,
$$\Res_{(\delta)}\in \Z[ (U_{\mu}^{(i)} : 1\leq i \leq m, \; \vert \mu\vert =d_i), (T_{r,s}\; :\; 1\leq r\leq n-m, 1\leq s \leq n)],$$
such that
$$\Res_{(\delta)}((f_1)_{d_1}, \ldots, (f_m)_{d_m}, (\ell_1)_{1}, \ldots, (\ell_{n-m})_{1})=0 \Longleftrightarrow V_{\PP_{n-1}(K)}(F_{\rm init})\not=\emptyset,$$
where $\Res_{(\delta)}((f_1)_{d_1}, \ldots, (f_m)_{d_m}, (\ell_1)_{1}, \ldots, (\ell_{n-m})_{1})$ represents the result of evaluating $\Res_{(\delta)}$ at the coefficients of the polynomials in the list $F_{\rm init}$ and $V_{\PP_{n-1}(K)}(F_{\rm init})$ is the set of projective zeros in $\PP_{n-1}(K)$ of the list of homogenerous polynomial equations in $F_{\rm init}$. This implies that the following set (formed by Generalized Pham systems) is a Zariski open subset in $\scrP_{(d)}\times \scrM_{(n-m)\times n}(K) \times \A^{n-m}(K)$:
\begin{equation}\label{generalized-Pham-sistemas:eqn}
\begin{matrix}\scrG_{(d)}:=\{ F:=(f_1,\ldots, f_m, \ell_1, \ldots, \ell_{n-m})\in\scrP_{(d)}\times \scrM_{(n-m)\times n}(K) \times \A^{n-m}(K) \; :\; \\
 \;\;\;\;\;\;\;\;\;\;\;\;\;\;\;\;\;\;\;\;\;\;\;\;\;\;\;\;\;\;\;\;\;\;\;\;(f_1,\ldots, f_m)\in \scrW_{(d)}, {\hbox {\rm and $F$ is a Generalized Pham system}}\}.\end{matrix}
 \end{equation}

Now we consider the homogenisation with respect to a new variable $X_0$ of a list of polynomials equations $F:=(f_1,\ldots, f_m, \ell_1,\ldots, \ell_{n-m})\in \scrG_{(d)}$. We denote by ${}^hF$ such new list of homogeneous polynomials equations in the set of variables $\{X_0,\ldots, X_n\}$.  As the number of homogeneous equations $n$ and the dimension of $\PP_n(K)$ are equal, the variety of the zeros of ${}^hF$ in $\PP_n(K)$ is non-empty and as $F\in \scrG_{(d)}$, no zero of ${}^hF$ lies at the infinity hyperplane: $$H_\infty:=\{ (x_0: x_1: \cdots : x_n)\in \PP_n(K)\; : \; x_0=0\}.$$
In other words, there exists a Zariski open subset $\scrG_{(d)}\subseteq \scrP_{(d)}\times \scrM_{(n-m)\times n}(K)\times \A^{n-m}(K)$ such that for all $F\in \scrG_{(d)}$ the set of affine zeros in $\A^n(K)$ of $F$ agree with the set of projective zeros of $^hF$ in $\PP_n(K)$, which is non-empty:
$$V_\A(f_1,\ldots, f_m, \ell_1, \ldots, \ell_{n-m})\isomorf V_{\PP_n(K)}(^hF)\not=\emptyset.$$
In particular, there is a non-zero polynomial in the list of coefficients $\CU:=(U_\mu^{(i)}\; :\; 1\leq i \leq m, \vert \mu \vert \leq d_i)$ of the equations in $\scrP_{(d)}$, the coordinates $\CM:=(T_{r,s}\; :\; 1\leq r \leq n-m, 1\leq s \leq n)$ of the matrices in $\scrM_{(n-m)\times}(K)$ and the coordinates $\CZ:=(Z_1,\ldots, Z_{n-m})$ of the points in $\A^{n-m}(K)$:
$$P(\CU, \CM, \CZ) \in K[\CU, \CM, \CZ]\setminus\{0\},$$
such that if $P(f_1,\ldots, f_m, \ell_1,\ldots, \ell_{n-m})\not =0$, then the variety:
$$V_\A(f_1, \ldots, f_m, \ell_1, \ldots \ell_{n-m})\not=\emptyset.$$
Now, as $K$ is an infinite field, there are some matrix $M\in \scrM_{(n-m)\times n}(K)$ and some point $\zeta \in \A^{n-m}(K)$ such that the polynomial:
$$Q(\CU):=P(\CU, M, \zeta)\in K[\CU]\setminus \{0\}.$$

Hence, we have proved that:
$$\forall f:=(f_1,\ldots, f_m)\in \scrP_{(d)}, \; Q(f_1,\ldots, f_m)\not=0 \Longrightarrow V_\A(f_1,\ldots, f_m)\not=\emptyset.$$
Thus, the following Zariski open subset is contained in $\pi_1(V_{(d)})$ and Claim $iii)$ immediately follows:
$$\pi_1(V_{(d)})\supseteq \{ f=(f_1,\ldots, f_m)\in\scrP_{(d)}\; :\; Q(f_1,\ldots, f_m)\not=0\}.$$
Let $U_\emptyset$ be a non-empty Zariski open subset of $\scrP_{(d)}$ contained in $\pi_1(V_{(d)})$. Then, for every $f\in U_\emptyset, \; V_\A(f)\not=\emptyset$.

Finally, for Claim $iv)$ we consider an irreducible component $C$ of $V_{(d)}$ and we take  its projection $\pi_1(C)\subseteq \scrP_{(d)}$. Then, we consider the restriction mapping:
$$\restr{\pi_1}{C}: C \longrightarrow \overline{\pi_1(C)}^z.$$
This is a dominant morphism and, by the Theorem on the Dimension of the Fibers, we first conclude that for any $f\in \pi_1(C)$ we have:
$$\dim(\pi_1^{-1}(\{f\})\cap C)=\dim (\restr{\pi_1}{C}^{-1}(\{f\}))\geq \dim(C) -  \dim(\overline{\pi_2(C)}^z).$$
If $\dim(\pi_1^{-1}(\{f\})\cap C)=n-m$, then we have that:
$$\dim(\overline{\pi_1(C)}^z)\geq \dim(C) - (n-m)= N_{(d)}+ (n-m) - (n-m)= N_{(d)},$$
and, hence $\pi_1(C)$ is Zariski dense in $\scrP_{(d)}$.

Otherwise, there is a Zariski open subset $U_C\subseteq \pi_1(C)\subseteq \overline{\pi_1(C)}^z$ such that for all $g\in U_C$ we will have:
$$\dim(\pi_1^{-1}(\{g\})\cap C)=\dim (\restr{\pi_1}{C}^{-1}(\{g\}))= \dim(C) -  \dim(\overline{\pi_1(C)}^z).$$
If $\dim(\pi_1^{-1}(\{g\})\cap C)\geq n-m+1$, for all $g\in \pi_2(C)$, then we conclude
$$\dim(\overline{\pi_1(C)}^z)= \dim(C) - \dim(\pi_1^{-1}(\{g\})\cap C)\leq N_{(d)}+ (n-m) - (n-m+1)= N_{(d)}-1,$$
and, hence, $\pi_1(C)$ will not be Zariski dense in $\scrP_{(d)}$.

As for the last claim of the Theorem, we just have to prove that the class:
$$\scrS:=\{f\in \scrP_{(d)}\;:\; {\hbox {\rm $f$ is a ``Suite S\'ecante''}}\},$$
contains a non-empty Zariski open subset. Let $\scrC$ be the class of all irreducible components of $V_{(d)}$ such that there is some $f\in \pi_1(C)$ satisfying $\dim(\pi_1^{-1}(f)\cap C)=n-m$. Let $\scrW$ be the class of all irreducible components of $V_{(d)}$ such that $\pi_1(W)$ is not Zariski dense in $\scrP_{(d)}$.

Then, we consider the non-emtpy Zariski open subset  of $\scrP_{(d)}$ given by the following identity:
$$A:=\scrP_{(d)} \setminus\left(\bigcup_{C\in \scrW} \overline{\pi_1(C)}^z\right).$$

For every $C\in \scrC$, we have that $\pi_1(C)$ is Zariski dense in $\scrP_{(d)}$. Then, because of the Theorem on the Dimension of the Fibers (cf. \cite{Shafarevich}) for every $C\in \scrC$ there is a Zariski open $U_C\subseteq \scrP_{(d)}$ such that for all $f\in U_C$, we have:
$$\dim\left( \pi_1^{-1}(\{f\})\cap C\right) = \dim\left(\restr{\pi_1}{C}^{-1}(\{ f\})\right)=
\dim(C)-\dim (\scrP_{(d)})= N_{(d)}+ (n-m) - N_{(d)}= n-m.$$
As $\scrP_{(d)}$ is irreducible, the following open subset $U\subseteq \scrP_{(d)}$ is non-empty:
$$U:=A\cap U_\emptyset \cap \left(\bigcap_{C\in \scrC} U_C\right).$$
Moreover, for every $f\in U$ we have:
\begin{itemize}
\item As $f\in U_\emptyset$, $V_\A(f)\not=\emptyset$.
\item As $f\in A$, then $f\not\in \pi_1(D)$ for any $D\in \scrW$ or, equivalently, we have  $\pi_1^{-1}(\{f\}) \cap D= \emptyset$ for every $D\in \scrW$. In particular, we have that:
$$V_\A(f)=\pi_1^{-1}(\{f\})= \left(\bigcup_{D\in \scrW} \left(\pi_1^{-1}(\{f\})\cap D\right)\right) \bigcup \left(\bigcup_{C\in \scrC} \left(\pi_1^{-1}(\{f\}) \cap C\right)\right)= $$
$$ = \bigcup_{C\in \scrC} \left(\pi_1^{-1}(\{f\}) \cap C\right).$$
\item Moreover, as $f\in U_C$ for every $C\in \scrC$ ,
$\dim(\pi_1^{-1}(\{f\})\cap C) = n-m$. Hence, for every $f\in \scrP_{(d)}$ we have that:
$$\dim\left(V_\A(f)\right)=\dim\left(\pi_1^{-1}(\{f\})\right)=\max\{ \dim\left( \pi_1^{-1}(\{f\})\cap C\right)) \; :\; C\in \scrC\}=n-m.$$
\end{itemize}

Namely, for every $f\in U$, $V_\A(f)$ is a complete intersection variety of dimension $n-m$ and, hence, $f$ is a ``Suite S\'ecante'' and the last claim follows.
\end{proof}

We now consider a class of fields well-suited for guessing correct test sequences.

\begin{definition}[{\bf Fields which are well-suited for CTS's guessing}]
We say that a field $K$ is well-suited for CTS's guessing in zero-dimensional varieties if it satisfies the following property: For every $R\in \N$ and every positive integer $n\in \N$, there is a zero-dimensional variety $V_{R}\subseteq \A^n(K)$ of degree $R^n$, given by polynomial equations of degree at most $R$ and such that the following task may be performed with at most $O(n \log_2(R))$ arithmetic operations:\\
$${\hbox {\bf guess at random}}\;\; x\in V_R.$$
\end{definition}

Fields of characteristic zero are obviously well-suited for CTS's guessing. It is enough to consider the subset $\{1,\ldots, R\}\subset \Z\subseteq K$ and the variety $V_R:=\{1,\ldots, R\}^n$.
Many other fields satisfy a similar property.

\begin{theorem}\label{suite-secante-BPP:teor}
If $K$ is a field well-suited for CTS's guessing, the problem ``Suite S\'ecante'' with restricted inputs is in ${\bf BPP}_K$. Namely, let $(d)$ be a degree list and let $\Omega\subseteq \scrP_{(d)}$ be a constructible subset of  lists of polynomials. Let $U\subseteq \scrP_{(d)}$ be the Zariski open subset described in Theorem \ref{variedad-incidencia-propiedades:teor} of lists that are ``Suites S\'ecantes''. Assume that $\Omega \setminus \left(\Omega\cap U\right)$ has dimension at most $\dim(\Omega)-1$. Then, there is an algorithm in ${\bf BPP}_K$ that solves ``Suite S\'ecante'' Problem with inputs in $\Omega$, i.e.:\\
\emph{Given as input $f\in \Omega$ and the data $\dim(\Omega)$ and $\deg_{\rm lci}(\Omega)$, the algorithm decides whether $f$ is a ``Suite S\'ecante'' or not.}\\
The running time of the algorithm in terms of arithmetic operations is at most:
$$O\left( \dim(\Omega)n\left((T +\log(\dim(\Omega))+ \log(d) \right) + Tn\log(\deg_{\rm lci}(\Omega)) \right).$$
where $T$ is the  maximum number of arithmetic operations required to evaluate at one point any list $f:=(f_1,\ldots, f_m)\in\Omega$ and $d:=\max\{ d_1,\ldots, d_m\}$. The error probability is bounded by:
$${{1}\over{\deg_{\rm lci}(\Omega)e^{6m\dim(\Omega)}}}.$$
\end{theorem}
\begin{proof}
The algorithm is what we expected from the notion of correct test sequence and the results described in Section \ref{probability-CTS-zero-dimensional:sec} above and especially Corollary \ref{densidad-questores-zero-dimensional:corol}.

\begin{algorithm}[{\bf ``Suite S\'ecante'' by evaluation of input polynomials at sampling points}]\label{suite-secante-algoritmo:algo}
{\sc Input:} A list of  polynomials $f:=(f_1,\ldots, f_m)\in \Omega\subseteq \scrP_{(d)}^K$, and the values $\dim(\Omega)$ and $\deg_{\rm lci}(\Omega)$.

{\bf eval:}
\begin{itemize}
\item $L:=6\dim(\Omega)$
\item $R:=\max\{ \lfloor(6(d+1))^2\rfloor, \lfloor \deg_{\rm lci}(\Omega)^{ {{2}\over{\dim(\Omega)}}}\rfloor\}$.
\end{itemize}

As $K$ is well-suited for CTS's, let $V_R\subseteq \A^n(K)$ be the corresponding zero-dimensional variety
of degree $R^n$.

{\bf guess at random} a list ${\bf Q}:=(x_1,\ldots, x_L)\in V_R^L$

{\bf if} $f(x_1)= f(x_2)= \cdots = f(x_L)=0\in K^m$, {\bf then} {\sc Ouput:} {\bf No}.

\hskip 1cm {\bf else}, {\sc Output:} {\bf Yes}.

{\bf fi}

{\bf end}

\end{algorithm}
\bigskip

The running time of this algorithm (in terms of arithmetic operations)  is of  the following order:
\begin{itemize}
\item Time $O(Ln\left(\log(\dim(\Omega))+  \log(d) + {{\log(\deg_{\rm lci}(\Omega))}\over{\dim(\Omega)}})\right))$ to generate $L$, $R$ and to perform the guessing of the list ${\bf Q}:=(x_1, \ldots, x_L)\in V_R^L$, since $K$ is well-suited for CTS's.
\item Time $O(LT)$ in order to evaluate the $m$ polynomials involved in an input $f:=(f_1,\ldots, f_m)$ at the $L$ points of the list ${\bf Q}$.

\end{itemize}

As $L=6\dim(\Omega)$ the total number of arithmetic operations is bounded by:
$$O\left( \dim(\Omega)n\left((T +\log(\dim(\Omega))+ \log(d) \right) + Tn\log(\deg_{\rm lci}(\Omega)) \right).$$

The probability of error is bounded by the probability that the list $(x_1,\ldots, x_L)$ obtained in the guessing step were a correct test sequence. Hence, according to Corollary \ref{densidad-questores-zero-dimensional:corol}, this error probability is bounded by:
\begin{equation}\label{cuestores-final-fino-prueba:eqn}
{{1}\over{\deg_{\rm lci}(\Omega)e^{\dim(\Omega) + (m-1)L}}}\leq {{1}\over{\deg_{\rm lci}(\Omega)e^{6m\dim(\Omega)}}}.
 \end{equation}
Thus, the algorithm is in ${\bf BPP}_K$ as wanted.\end{proof}

\begin{corollary}[{\bf ``Suite S\'ecante'' with dense input polynomials}]\label{suite-secante-dense:corol}
If $K$ is a field well-suited for CTS's guessing, the problem ``Suite S\'ecante'' is in ${\bf BPP}_K$ in the case of dense encoding of input polynomials. Namely, let $(d):=(d_1,\ldots, d_m)$ be a degree list and let $\Omega=\scrP_{(d)}^K$ be the class of all sequences of polynomials of respective degrees at most $(d)$.
Then, there is an algorithm in ${\bf BPP}_K$ that solves the following problem:\\
\emph{Given as input $f\in \scrP_{(d)}^K$, the algorithm decides whether $f$ is a ``Suite S\'ecante'' or not.}\\
The running time of the algorithm in terms of arithmetic operations is at most:
$$O\left( nN_{(d)}^2 \right),$$
where $N_{(d)}$ is the dimension of $\scrP_{(d)}^K$ as $K-$vector space
and the error probability is bounded by
$${{1}\over{e^{6mN_{(d)}}}}.$$
\end{corollary}
\begin{proof} The statement follows from the previous one since $N_{(d)}=\dim(\Omega)$ and
$\deg_{\rm lci}(\Omega)=1$.
\end{proof}

\begin{remark}
In the case $m=1$, ``Suite S\'ecante'' Problem becomes the usual \emph{Zero Test Problem} and all our algorithms become in the class ${\bf RP}_K$. All the bounds remain in both corollaries for this case.
\end{remark}

\begin{remark} In the case $\Omega$ is the class of all sequences of polynomials evaluable by some straight-line program (or neural network with polynomial activation functions), the results also hold, just replacing $\dim(\Omega)$ and $\deg_{\rm lci}(\Omega)$ by the bounds described in Sections 2 and 3 of \cite{Krick-Pardo96}, for instance.
\end{remark}

\appendix

\section{A remark on bounds for projections and images of constructible sets}\label{degree-projection-constructible:sec}

As we observed in Propositions \ref{constructibles-grado-imagenes-por-lineal:prop} and \ref{pi-degree-propiedades:prop}, one of the main differences is that the projection degree $\deg_\pi$ and the degree as a decomposition $\deg_{\rm lci}$ have a different behaviour with respect to linear images of constructible sets and algebraic varieties. Given a linear map $\ell:\A^n \longrightarrow \A^m$ and any constructible subset $C\subseteq \A^n$, the following inequalities hold:
$$\deg_z(\ell(C)) \leq \deg_\pi(\ell(C))\leq \deg_\pi(C) \leq \deg_{\rm lci}(C).$$
Additionally, we have seen examples such that $\deg_{\rm lci}(\ell(C))> \deg_{\rm lci}(C)$ in Proposition \ref{constructibles-grado-imagenes-por-lineal:prop}.  For the sake of completeness, we exhibit in this Appendix some upper bounds either for the ``defect'' $\deg_{\rm lci}(\ell(C))-\deg_z(\ell(C))$ and, in terms of extrinsic bounds, upper bounds for $\deg_{\rm lci}(\ell(C))$ based of the Effective Nullstellensatz of \cite{Jelonek}.

\subsection{A first bound on the defect}\label{deffective-projection:subsec}
First of all, we see that the defect may be controlled by the degree of the constructible set of those points whose fibers do not have the generic fiber dimension. We use the same notations as in previous sections of this manuscript.

\begin{lemma}\label{proyeccion-dimension-fibra:lema} Let $V\subseteq \A^n$ be an irreducible algebraic variety and $\ell: \A^n \longrightarrow \A^m$ be a linear mapping. Let $W:=\overline{\ell(V)}^z$ be the Zariski closure of the image. Let $U\subseteq \A^m(K)$ be any Zariski open subset such that $U\cap W\not=\emptyset$ and:
    \begin{equation}\label{fiber-dimension-hypothesis:eqn}
    \forall y \in W\cap U, \; \dim(\ell^{-1}(\{y\}))=\dim(V)-\dim(W).
    \end{equation}
Thus, $U\cap W\subseteq \ell(V)$. Additionally, let $U^c:=\A^m\setminus U$ be the algebraic variety  of those points whose fibers are not granted to have the appropriate dimension. Then, we have:
        $$\deg_{\rm lci}(\ell(V)) -\deg_z(\ell(V))\leq \deg_{\rm lci}(\ell(V)\cap U^c),$$
where $\dim(\ell(V)\cap U^c) \leq \dim(\ell(V))-1$. We also have:
        $$\deg_{\rm lci}(\ell(V)) -\deg_z(\ell(V))
\leq \deg_{\rm lci}(\ell(V\cap \ell^{-1}(U^c))).$$
\end{lemma}

\begin{proof}
Let us first observe that for each $y\in U\cap W$, $\ell^{-1}(\{y\})\not=\emptyset$. This holds because of Identity (\ref{fiber-dimension-hypothesis:eqn}) since $\dim(V)-\dim(W)\geq 0$. Thus, $\ell(V)\cap U= W\cap U$ is a locally closed subset. Then, we have:
$$\ell(V)= \left(\ell(V)\cap U\right) \cup \left( \ell(V) \cap U^c\right)=
\left(W\cap U\right) \cup \left( \ell(V) \cap U^c\right).$$
From the sub-additivity of  $\deg_{\rm lci}$, we obtain:
$$\deg_{\rm lci}(\ell(V)) \leq \deg_{\rm lci}(W\cap U) + \deg_{\rm lci}(\ell(V)\cap U^c).$$
As $W\cap U$ is locally closed $\deg_{\rm lci}(W\cap U)= \deg(W) = \deg_z(\ell(C))$ and, hence, we have:
$$\deg_{\rm lci}(\ell(V))-\deg_z(\ell(V))\leq \deg_{\rm lci}(\ell(V)\cap U^c).$$
Since $W \cap U$ is Zariski dense in $W$, we deduce that $\overline{(\ell(V) \cap U^c)}^z$ is a closed proper subvariety of the irreducible variety $W$. Therefore,
$$\dim(\ell(V)\cap U^c)\leq \dim(\ell(V))-1.$$
Finally, let us observe that the following equality holds:
$$\ell(V\cap \ell^{-1}(U^c))= \ell(V) \cap U^c.$$
For if $x\in \ell(V\cap \ell^{-1}(U^c))$ we obviously have that $x\in \ell(V)\cap U^c$. On the other hand, if $x\in \ell(V)\cap U^c$, and $y\in V$ is such that $\ell(y)=x$, we obviously have $y\in \ell^{-1}(U^c)$ and, hence, $x\in \ell(V\cap\ell^{-1}(U^c))$.
\bigskip
\end{proof}

\begin{lemma}\label{cota-defecto-lema:lema}
Let $V\subseteq \A^n(K)$ be an irreducible algebraic variety of dimension $r$ and degree $D$. Let $\ell:\A^n(K)\longrightarrow \A^m(K)$ be a linear map. Assume that the restriction $\restr{\ell}{V}: V\longrightarrow \A^m(K)$ is dominating (and, hence, $m\leq r$). Let $\{X_1,\ldots, X_n\}$ be the variables associated to $\A^n$ and, accordingly, $\{ Y_1,\ldots, Y_m\}$ the variables associated to $\A^m$.
Then, there is a non-zero polynomial $q\in K[Y_1,\ldots, Y_m]$ such that the following properties hold:
\begin{enumerate}
\item The polynomial $p:=q\circ\ell\in K[X_1,\ldots, X_n]$ is also a non-zero polynomial and their respective degrees satisfy:
  $$\deg(p)\leq \deg(q) \leq (n-r) (\deg(V)-1).$$
\item For every $y\in \A^m $ such that $q(y)\not=0$, the dimension of the fiber satisfies:
$$\dim\restr{\ell}{V}^{-1}\left( \{ y\}\right) = \dim \left( V \cap \ell^{-1}\left(\{ y\}\right) \right)= r-m.$$
\item Either $V\cap V_\A(q)=\emptyset$ or the dimension satisfies $\dim(V\cap V_\A(q))= \dim(V)-1$.
\item The defect of $\ell(V)$ satisfies:
$$\deg_{\rm lci}(\ell(V))-\deg_z(\ell(V))\leq \deg_{\rm lci}(\ell(V\cap V_\A(q))),$$
where $\deg_z(\ell(V))=1$ since $\restr{\ell}{V}$ is dominating onto $\A^m$.
\end{enumerate}
\end{lemma}

\begin{proof}First of all, as $\restr{\ell}{V}$ is dominating, then $\ell: \A^n \longrightarrow \A^m$ is an onto linear mapping. Let $\ell_1, \ldots, \ell_m$ be the coordinate functions of $\ell$. In particular, the ring $K[\ell_1,\ldots, \ell_m]$ is a ring of polynomials in $m$ variables with coefficients in $K$. Up to a generic choice of a linear change of coordinates in $\A^n$ we may assume that $\ell_1=X_1,\ldots, \ell_m=X_m$. Then, we proceed in our proof as if $\ell$ were the canonical projection $\ell:=\pi_m:\A^n(K)\longrightarrow \A^m(K)$ onto the first $m$ coordinates of a point in $\A^n(K)$.
Thus, let $I(V)\in \Spec(K[X_1,\ldots, X_n])$ be the prime ideal associated to $V$. As $\restr{\ell}{V}$ is dominating, we would have the following ring extension:
$$\ell^*:K[X_1,\ldots, X_m]\longhookrightarrow K[V]:=K[X_1,\ldots, X_n]/I(V).$$
Let us consider the multiplicative system $S:=K[X_1,\ldots, X_m]\setminus \{0\}$  and let $F:=K(X_1,\ldots,X_m)$ be the quotient field of $K[X_1, \ldots, X_m]$.  As $I(V)\cap K[X_1,\ldots, X_m]=(0)$ we also have the following ring extension:
$$S^{-1}\ell^*: F \longhookrightarrow F[V]:=S^{-1}K[V]=F[X_{m+1}, \ldots, X_n]/{\frak p},$$
where ${\frak p}= S^{-1} I(V)$.

As $\restr{\ell}{V}$ is dominating, the Theorem on the Dimension of the Fibers yields that the generic dimension of a fiber $(\restr{\ell}{V})^{-1}(\{x\})$ has dimension $r-m$. Hence, the Krull dimension satisfies:
$$\dim\left(F[V] \right)=\dim\left(F[X_{m+1}, \ldots, X_n]/{\frak p} \right)= r-m.$$
As $K$ is an infinite field, so does $F$ and  there are generically many matrices $A\in \scrM_{n-m}(K)$ such that the following linear change of coordinates:
$$\left(\begin{matrix} Y_{m+1}\\ \vdots \\ Y_n\end{matrix}\right) = A \left(\begin{matrix} X_{m+1}\\ \vdots \\ X_n\end{matrix}\right),$$
 puts the variables $\{ Y_{m+1}, \ldots, Y_n\}$ in Noether position with respect to ${\frak p}$, i.e. the following is an integral ring extension:
$$F[Y_{m+1}, \ldots, Y_r] \longhookrightarrow F[V]:=F[Y_{m+1}, \ldots, Y_n]/{\frak p},$$
where we also call ${\frak p}$ the transformation of ${\frak p}$ under such linear change of coordinates.
Then, for every $j$, $r+1\leq j \leq n$, there is a monic polynomial $h_j\in F[Y_{m+1}, \ldots, Y_r][T]$ of positive degree with respect to the variable $T$ (i.e. $\delta_j=\deg_T(h_j)\geq 1$), such that:
$$h_j(Y_j)\in {\frak p}, \; r+1\leq j \leq n.$$
As ${\frak p}$ is a prime ideal, we may assume that $h_j$ is irreducible.
For every $j$, $r+1\leq j \leq n$,  there is a primitive polynomial $H_j\in K[X_1,\ldots, X_m, Y_{m+1}, \ldots, Y_r][T]$ such that $h_j$ and $H_j$ are associated in the ring $F[Y_{m+1}, \ldots, Y_r][T]$. Namely, for every $j$, $r+1\leq j \leq n$, there is a polynomial of the following form:
$$H_j:= e_{\delta_j}^{(j)}(X_1,\ldots, X_m) T^{\delta_j}+ \sum_{k=0}^{\delta_j-1} a_k^{(j)} (X_1,\ldots,X_m,Y_{m+1},\ldots, Y_r) T^k,$$
where $e_{\delta_j}^{(j)}\in K[X_1,\ldots, X_m]\setminus\{0\}$, $a_k^{(j)}\in K[X_1,\ldots, X_m, Y_{m+1}, \ldots, Y_r]$ and the following properties hold:
\begin{itemize}
\item Every $H_j$ is a primitive polynomial,
\item For every $j$, $r+1\leq j\leq n$, the polynomial $H_j$ is irreducible in $K[X_1,\ldots, X_m, Y_{m+1}, \ldots, Y_r][T]$ and the following equality holds:
$$h_j:= \left( e_{\delta_j}^{(j)}\right)^{-1} H_j,$$
\item For every $j$, $r+1\leq j\leq n$, $H_j(Y_j)\in {\frak p}$, $e_{\delta_j}^{(j)}\not\in {\frak p}$.
\end{itemize}
The last property is a consequence of the fact that $e_{\delta_j}^{(j)}\not=0$ and ${\frak p}$ is a prime ideal. Let us then define the following non-zero polynomial:
$$q:=\prod_{j=r+1}^n e_{\delta_j}^{(j)}\in K[X_1,\ldots, X_m]\setminus\{0\}.$$
Observe that, as $I(V)$ was a prime ideal, then $q\not\in I(V)$. We then define the following rings and locally closed sets:
\begin{enumerate}
\item The distinguished open subset of $\A^m$ given by:
$$D(q):=\{ x\in \A^m \; :\; q(x)\not=0\}.$$
\item The non-empty Zariski open subset $V_q$ of $V$ given by the following identity:
$$V_q:=\{\zeta:=(x,y,u)\in \A^m\times \A^{r-m}\times \A^{n-r}\; : \; \zeta \in V, \; q(x)\not=0\}.$$
\item The ring of regular functions defined on $D(q)$, given as the localization on the multiplicative system defined by $q$:
$$K[D(q)]:=K[X_1,\ldots,X_m]_q.$$
\item The ring of regular functions define on $V_q$, given as:
$$K[V_q]:= K[X_1,\ldots,X_m]_q[Y_{m+1}, \ldots, Y_n]/{\frak q},$$
where ${\frak q}$ is, at the same time, the extension of $I(V)$ to $K[X_1,\ldots,X_m]_q[Y_{m+1}, \ldots, Y_n]$ and the contraction of ${\frak p}$ to the same ring.
\end{enumerate}
As, for every $j$, $r+1\leq j\leq n$, $h_j\in K[X_1,\ldots,X_m]_q[Y_{m+1}, \ldots, Y_r][T]$ and $h_j(Y_j)\in {\frak p}$, the following is an integral ring extension:
\begin{equation}\label{extension-entera-en-q:eqn}
K[D(q)][Y_{m+1}, \ldots, Y_r]\longhookrightarrow K[V_q]:= K[X_1,\ldots,X_m]_q[Y_{m+1}, \ldots, Y_n]/{\frak q}.
\end{equation}
From this integral equation we immediately conclude the following first Claim:
\begin{equation}\label{fibra-no-vacia:eqn}
\forall x \in D(q), \; \restr{\ell}{V}^{-1}\left(\{ x\}\right)\not=\emptyset.
\end{equation}
In fact, recalling that $\ell=\pi_m$ is a projection, the integral ring extension of (\ref{extension-entera-en-q:eqn}) implies that the following map is onto:
$$\begin{matrix}
\pi: & V_q& \longrightarrow & D(q)\times \A^{r-m}\\
& \zeta=(x,y,u) & \longmapsto & (x,y).
\end{matrix}
$$
Let us now consider the following projection maps:
$$\begin{matrix}
\Pi_j: & V_q & \longrightarrow & \A^{r+1}\\
&\zeta:=(x,y,u) & \longmapsto &(x,y, u_j),
\end{matrix}$$
where $x\in D(q)$, $y\in \A^{r-m}$ and $u:=(u_{r+1}, \ldots, u_n)\in \A^{n-r}$. We now prove the following Claim:
\begin{claim}
With these notations and assumptions, $\Pi_j(V_q)$ is a hyper-surface in $D(q)\times \A^{r-m+1}$ of degree
at most $\deg(V)$ and its minimal equation is the polynomial $H_j$ introduced above, whose total degree satisfies $\deg(H_j)\leq \deg(V)$.
\end{claim}
\begin{proof-claim} It is clear that $\Pi_j(V_q)\subseteq D(q)\times \A^{r-m+1}$. Moreover, if $\lambda:\A^{r+1}\longrightarrow \A^r$ is the projection that forgets the last coordinate, we have that:
$$\lambda \circ \Pi_j= \pi,$$
where $\pi$ is the onto map defined above. Hence, we conclude that:
$$\dim(\Pi_j(V_q)) \geq \dim (\lambda(\Pi_j(V_q)))\geq \dim(\pi(V_q))= \dim(D(q)\times \A^{r-m})=r.$$
On the other hand, the polynomial $H_j\in K[X_1,\ldots, X_m, Y_{m+1}, \ldots, Y_r][T]$ is a non-zero polynomial and we have:
$$\Pi_j(V_q)\subseteq \{ (x,y, u_j)\in \A^{r+1}\; :\; q(x)\not=0, H_j(x,y,u_j)=0\}.$$
By Krull's Hauptidealsatz, the co-dimension of $\Pi_j(V_q)$ is at least $1$ and we have proved the following inequality:
$$\dim(\Pi_j(V_q))= r.$$
Thus, $\overline{\Pi_j(V_q)}^z$ is a hyper-surface in $\A^{r+1}$. It is irreducible since $V$ was irreducible  and, hence, there is an irreducible polynomial $g_j\in K[X_1,\ldots, X_m, Y_{m+1}, \ldots, Y_r, Y_j]$ such that $I(\overline{\Pi_j(V_q)}^z)=(g_j)$. As $H_j$ vanishes on $\Pi_j(V_q)$ we then have $g_j\mid H_j$ and, as $H_j$ was primitive and irreducible, we conclude that $g_j=H_j$ up to some constant in $K\setminus\{0\}$.\\
Finally, the degrees satisfy the following equalities and inequalities:
$$\deg(V)=\deg(V_q)\geq \deg_z(\Pi_j(V_q))= \deg(\overline{\Pi_j(V_q)}^z) = \deg(g_j)= \deg(H_j).$$

\end{proof-claim}

From this claim we conclude that for every $j$, $r+1\leq j \leq n$, we have:
$$\deg(e_j^{(\delta_j)}) + \delta_j \leq \deg(H_j) \leq \deg(V).$$
Thus, as $\delta_j\geq 1$ for every $j$, we have:
\begin{equation}\label{grado-del-q:eqn}
\deg(q)= \sum_{j=r+1}^n \deg(e_j^{(\delta_j)})\leq \sum_{j=r+1}^n (\deg(V)-1) \leq (n-r)(\deg(V)-1).
\end{equation}
Combining this last inequality with the statement in Equation (\ref{fibra-no-vacia:eqn}) and the previous Lemma we conclude that $q\in K[X_1,\ldots, X_m] \setminus \{ 0 \}$ is the polynomial claimed at the statement.\end{proof}

\begin{remark}\label{optimalidad-defecto-ejemplo:rk}
This last Lemma explains the existence of the defect in the Example exhibited to prove Claim $iv)$  of Proposition \ref{constructibles-grado-imagenes-por-lineal:prop}. There was a quadric hyper-surface in $\A^3(\C)$ of degree $2$ given by the following identity:
$$W':=\{(x,y,z)\in \C^3\; : \; xz+ (y^2-1)=0\}.$$
We consider the canonical projection $\ell=\pi_2:\A^3\longrightarrow \A^2$. The polynomial $q$ cited in the previous Lemma is the leading coefficient of the quadratic polynomial defining $W'$, i.e. $q:=Z$. Its total degree satisfies:
$$\deg(q) = 1 \leq (\codim(W'))(\deg(W')-1) = (3-2) (2-1)=1.$$
The defect is then bounded by:
$$\deg_{\rm lci}(\pi_2(W'))- \deg_z(\pi_2(W'))\leq \deg_{\rm lci}(\pi_2(W'\cap V_\A(q))).$$
However, $W'\cap V_\A(q)$ is a pair of lines given by:
$$W'\cap V_\A(q)=\{ (x,y,z)\in W' \; : \; z=0\}= \{(x,y,0)\in \A^3 \; : \; y^2-1=0\}.$$
Hence, $\pi_2(W'\cap V_\A(q))=\{(x, y)\; : \; y\in \{\pm 1\}\}$ is an algebraic variety and, hence,
$$\deg_{\rm lci}(\pi_2(W'\cap V_\A(q)))= \deg(\pi_2(W'\cap V_\A(q)))=2.$$
We have already seen that $\deg_{\rm lci}(\pi_2(W'))=3$ and $\deg_z(\pi_2(W'))=1$. In particular, the bound in the previous Lemma is optimal in this example since:
$$\deg_{\rm lci}(\pi_2(W'))- \deg_z(\pi_2(W'))=3-1= 2 = \deg_{\rm lci}(\pi_2(W'\cap V_\A(q))).$$
This does not mean that the upper bound in Claim $iv)$ of the previous Lemma is always optimal, but at least it is optimal in some examples.
\end{remark}

\subsection{An extrinsic bound for the degree of a projection, based on the Effective Nullstellensatz}
\label{extrinsic-bound-projection:subsec}
We revise Proposition \ref{constructibles-grado-imagenes-por-lineal:prop}. Here we show a coarse and syntactical upper bound for the $LCI$-degree
of the image $\pi(V)$ of some irreducible algebraic variety $V$ with respect to the canonical projection. The bound is syntactical since we just take into account a syntactical description of $V$. The method we choose is based on the Effective Nullstellensatz. This simply means that $LCI-$ degree is ``controlled'' under linear images in terms of syntactical descriptions of the domain. Although we could have used either \cite{Kollar} or \cite{Sombra}, we preferred to choose the bounds of \cite{Jelonek}.
We also use \cite{Mulmuley} for the degree bounds, although other may also be used.

As above, $K$ is an algebraically closed field.

\begin{theorem}\label{extrinisic-bound-Nullstellensatz:teor}
Let $V\subseteq \A^n(K)$ be an irreducible algebraic variety of dimension $r$.  Let $m$ be an integer such that $m\leq n$. Assume that there are polynomials $g_1, \ldots, g_s\in K[X_1,\ldots, X_n]$ of degrees $d_i:=\deg(g_i)$, $1\leq i \leq r$, such that $V=V_\A(g_1,\ldots, g_s)$ and the following inequalities hold:
$$d_1\geq d_2 \geq \ldots \geq d_s.$$
Let us consider the following quantity:
$$N:=N(d_1, \ldots, d_s, n, r):=\left\{\begin{matrix} \prod_{i=1}^s d_i & {\hbox {\rm if $s \leq n-m$}}\\
2d_s\left(\prod_{i=1}^{n-m-1}d_i\right)-1 & {\hbox {\rm if $s > n-m$}} \end{matrix}\right.$$

Let us define the following quantities:
$$\widetilde{N}:={{N+(n-m)}\choose{n-m}},$$
$$M:=\sum_{i=1}^s {{N-d_i+(n-m)}\choose{(n-m)}},$$
and
$$\scrN':=\min\{N, M+1\}.$$

Let $\pi: \A^n(K) \longrightarrow \A^m (K)$ be the canonical projection that forgets the last $n-m$ variables and let $W:=\overline{\pi(V)}^z$ be the Zariski closure of $\pi(V)$ in $\A^m(K)$. Then, we have:

$$\deg_{\rm lci}(\pi(V)) \leq \deg(V)\left(2d_1\right)^{\dim (W)} \left(\scrN'\right)^{\dim(W)+1}.$$

\end{theorem}
\begin{proof}
First of all, for each monomial exponent $\nu=(\nu_1,\ldots, \nu_k)\in \N^k$, we denote its degree by $\vert \nu\vert:= \nu_1+ \cdots + \nu_k\in \N$. Next, for each $i$, $1\leq i \leq s$,
we are going to introduce a new set of variables:
$$\underline{Z}^{(i)}:=\{ Z_\mu^{(i)}\; :\; \mu\in \N^{n-m}, \vert \mu \vert \leq N-d_i\},$$
ordered by the ``degree+lexicographic'' monomial ordering,
and the polynomial with generic coefficients:
$$F_i:=\sum_{\mu\in \N^{n-m}, \vert \mu \vert \leq N-d_i} Z_\mu^{(i)}X_{m+1}^{\mu_{m+1}} \cdots X_n^{\mu_n}.$$
Note that the number of variables involved in $\underline{Z}^{(i)}$ (and also the number of generic coefficients of $F_i$) is given by:
$$M_i:={{N-d_i+(n-m)}\choose{(n-m)}}.$$
Also, for each $i$, $1\leq i \leq s$, we rewrite the polynomials $g_1, \ldots, g_s$ as:
$$g_i:=\sum_{\nu\in \N^{n-m}, \; \vert \nu \vert \leq d_i} h_{\nu}^{(i)}(X_1,\ldots, X_m)X_{m+1}^{\nu_{m+1}}
\cdots X_n^{\nu_n}.$$
Finally, we consider the column vector:
$${\bf 1}:=\left(  \begin{matrix} 1 \\0 \\ \vdots \\ 0 \end{matrix}\right)\in K^{\widetilde{N}},$$
that encodes the coefficients of $1\in K[X_{m+1},\ldots, X_n]$ using a polynomial of degree at most $\widetilde{N}$, with respect to ``degree+lexicographic'' monomial order in $K[X_{m+1},\ldots, X_n]$. The number of coordinates in ${\bf 1}$ is given by the following quantity:
$$\widetilde{N}:={{N+(n-m)}\choose{n-m}}.$$
Now, we consider the following sum:
$$\sum_{i=1}^s F_ig_i \in K[\underline{Z}^{(1)}, \ldots, \underline{Z}^{(s)}, X_1,\ldots, X_m][X_{m+1}, \ldots, X_n].$$
Taking $R:=K[\underline{Z}^{(1)}, \ldots, \underline{Z}^{(s)}, X_1,\ldots, X_m]$, the coefficients of
this sum as polynomials in $R[X_{m+1}, \ldots, X_n]$ can be represented as a matrix product:
$$A(X_1,\ldots, X_m) \left( \begin{matrix} \underline{Z}^{(1)}\\ \vdots \\ \underline{Z}^{(s)}\end{matrix}\right).$$
The coefficients of $A(X_1,\ldots, X_m)$, being those of the list:
\begin{equation}\label{coeficientes:eqn}
\scrH:=\{ h_{\nu}^{(i)} \; : \; \nu\in \N^{n-m}, \vert \nu \vert \leq d_i, \; 1\leq i \leq s\},
\end{equation}
have all of them degree at most $D := d_1$. Moreover, the matrix $A(X_1,\ldots, X_m)$ is in
$$ \scrM_{\widetilde{N}\times M}(K[X_1,\ldots, X_m]),$$
where the number of rows and columns are determined by the following rules:
\begin{itemize}
\item The number of rows is the quantity:
$$\widetilde{N}:={{N+(n-m)}\choose{n-m}}.$$
\item The number of columns is given by the quantity:
$$M:=\sum_{i=1}^s {{N-d_i+ (n-m)} \choose{n-m}}.$$
\end{itemize}

We finally consider the system of equations:
\begin{equation}\label{Nullstellensatz-lineal1:eqn}\scrS(X_1,\ldots, X_m)\equiv \left\{ A(X_1,\ldots, X_m)\left(\begin{matrix}\underline{Z}^{(1)}\\
\underline{Z}^{(2)}\\\vdots \\
\underline{Z}^{(s)} \end{matrix} \right)= {\bf 1}
\right \rbrace. \end{equation}
We observe that the following two claims are equivalent for every $x=(x_1,\ldots, x_m)\in \A^m(K)$:
\begin{enumerate}
\item The fiber $\pi^{-1}(\{ x \})\cap V \not= \emptyset$,
\item The system of linear equations $\scrS(x_1,\ldots, x_m)$, obtained by specializing the variables $X_i$ into the coordinates $x_i$, is \emph{inconsistent}.
\end{enumerate}

Now, we follow the main idea of \cite{Mulmuley} which works in any field of any characteristic. Mulmuley's trick goes as follows. First,
let $W:=\overline{\pi(V)}^z$ be the Zariski closure of $\pi(V)$ in $\A^m(K)$ and let us consider the following two matrices:
$$A(X_1,\ldots, X_m)\in \scrM_{\widetilde{N}\times M}(K[W]), B(X_1,\ldots, X_m):= \left( X_1,\ldots, X_m \mid {\bf 1}\right)
\in \scrM_{\widetilde{N}\times (M +1)}(K[W]).$$

Then, we consider the following two square and symmetric matrices:
$$A^*(X_1,\ldots, X_m):= \left( \begin{matrix} A(X_1,\ldots, X_m) & 0\\
0 & A(X_1,\ldots, X_m)^T\end{matrix}\right)\in \scrM_{\widetilde{N}+M}(K[W]),$$
$$B^*(X_1,\ldots, X_m):=\left( \begin{matrix} B(X_1,\ldots, X_m) & 0\\
0 & B(X_1,\ldots, X_m)^T\end{matrix}\right)\in \scrM_{\widetilde{N}+M+1}(K[W]).$$
Moreover, the ranks satisfy the following equalities for all $x\in \A^m(K)$:
$$\rank\left(A^*(x)\right)= 2 \rank \left(A(x)\right), \rank\left(B^*(x)\right)= 2 \rank \left(B(x)\right).$$
For every $x\in W$, the system of linear equations $\scrS(x)$, described in (\ref{Nullstellensatz-lineal1:eqn}), is inconsistent if and only if $\rank(A^*(x))\not = \rank (B^*(x))$.

As $A^*(\underline{X})$ and $B^*(\underline{X})$ are symmetric, their ranks depend on the order of $0$ as root of their respective characteristic polynomials. We, thus, consider their characteristic polynomials modulo $I(W)$:
\begin{itemize}
\item The characteristic polynomial of $A^*(X_1,\ldots, X_m)$:
$$\chi_A(T):= T^{\widetilde{N}+M} + a_{\widetilde{N}+M-1}T^{\widetilde{N}+M-1} + \cdots + a_1 T + a_0 \in K[W][T],$$
where the coefficients $a_i$ are of the form $A_i+ I(W)$, where $A_i\in K[X_1, \ldots, X_m]$.
\item Similarly, we have the characteristic polynomial of $B^*(X_1, \ldots, X_m)$:
$$\chi_B(T):= T^{\widetilde{N}+M+1} + b_{\widetilde{N}+M}T^{\widetilde{N}+M} + \cdots + b_1 T + b_0 \in K[W][T].$$
We also conclude that
 the coefficients $b_i\in K[W]$ are of the form $B_i+ I(W)$, where $B_i\in K[X_1, \ldots, X_m]$.
\end{itemize}
Let $\scrN:=\min\{\widetilde{N}, M\}$ and  $\scrN':=\min\{\widetilde{N}, M+1\}$ be the minimum between the number of rows and columns of $A(X_1,\ldots, X_m)$ and $B(X_1,\ldots, X_m)$. For every $x\in W$ we have:
$$\rank(A^*(x))\leq 2\scrN, \; \rank(B^*(x))\leq 2 \scrN'.$$
Thus, the order of $0$ in $\chi_A(T)$ and $\chi_B(T)$ satisfy in $K[W]$:
$$ord_0(\chi_A) \geq \widetilde{N}+ M - 2\scrN, \; ord_0(\chi_B) \geq \widetilde{N}+ M+ 1 - 2\scrN'.$$
In particular, the following equalities hold:

$$\chi_A(T):= T^{\widetilde{N}+M} + a_{\widetilde{N}+M-1}T^{\widetilde{N}+M-1} + \cdots + a_{\widetilde{N}+M- 2\scrN} T^{\widetilde{N}+M- 2\scrN} \in K[W][T],$$
and
$$\chi_B(T):= T^{\widetilde{N}+M+1} + b_{\widetilde{N}+M}T^{\widetilde{N}+M} + \cdots + b_{\widetilde{N}+M+1- 2\scrN'} T^{\widetilde{N}+M+1- 2\scrN'}  \in K[W][T].$$

As the coordinates of the matrices $A^*(X_1,\ldots, X_m)$ and $B^*(X_1,\ldots, X_m)$ are in the class $\scrH\cup\{0,1\}$, where $\scrH$ is the list of polynomials introduce in Identity (\ref{coeficientes:eqn}), and the polynomials in $\scrH$ have degree at most $D:=d_1$, we conclude:
\begin{itemize}
\item For each $i$, $1\leq i \leq 2\scrN$, there is a polynomial $A_{\widetilde{N}+ M- i}(X_1,\ldots, X_m)\in K[X_1,\ldots, X_m]$ of degree at most $d_1i$ such that $A_{\widetilde{N}+ M- i}+ I(W)=
    a_{\widetilde{N}+ M- i}$.
\item For each $j$, $1\leq j \leq 2\scrN'$, there is a polynomial $B_{\widetilde{N}+ M+1- j}(X_1,\ldots, X_m)\in K[X_1,\ldots, X_m]$ of degree at most $d_1j$ such that $B_{\widetilde{N}+ M+1- j}+ I(W)=
    b_{\widetilde{N}+ M+1- j}$.

\end{itemize}

For each $x\in W$, as  $\rank(B^*(x))=2\rank(B(x))\geq 2\rank(A(x))=\rank(A^*(x))$,  the system
of linear equations $\scrS(x)$ of (\ref{Nullstellensatz-lineal1:eqn}) is inconsistent if and only if there exists $k\in \{0, \ldots , \scrN\}$ such that:
$$\rank(A^*(x))= 2k, \; \rank (B^*(x))=2(k+1).$$

Transforming this rank equalities into order of $0$ as zero of $\chi_A$ and $\chi_B$, for each $x\in W$, system $\scrS(x)$ of (\ref{Nullstellensatz-lineal1:eqn}) is inconsistent if and only if there is $k$, $0\leq k \leq \scrN$, such that $x\in R_k$, where
$R_k$ is the class given by the following intersection:
$$R_k:=\{ x\in W \; : \; ord_0(\chi_{A(x)}(T))= \widetilde{N}+ M-2k\}\cap \{ x\in W \; : \; ord_0(\chi_{B(x)}(T))= \widetilde{N}+ M+1-2(k+1)\}.$$
In terms of polynomial equations and equalities, this may be written as follows:
$$R_k:=\CA_k\cap \CB_k,$$
where
$$\CA_k:=\{x\in W \; : \; A_i(x)=0, \widetilde{N}+M- 2\scrN \leq i \leq \widetilde{N}+M-2k-1, A_{\widetilde{N}+M-2k}(x)\not=0\},$$
and
$$\CB_k:=\{ x\in W \; : \; B_j(x)=0, \widetilde{N}+M+1- 2\scrN'\leq j \leq \widetilde{N}+M+1-2(k+1)-1, B_{\widetilde{N}+M+1-2(k+1)}(x)\not=0\}.$$
In particular, the following equality holds:
$$\pi(V):= \bigcup_{k=0} ^{\scrN-1} R_k,$$

Moreover, each $R_k$ is a locally closed set given as the intersection of $W$ with open sets and hyper-surfaces given by polynomial equations of degree at most:
$$\max\left(\{\deg(A_i)\; :\; \widetilde{N}+M- 2\scrN \leq i \} \cup \{\deg(B_j)\; :\; \widetilde{N}+M+1- 2\scrN' \leq j\}\right).$$
In particular, $R_k$ is a locally closed set given by a Zariski open subset of $W$ with hyper-surfaces of degree at most $2\scrN'$.
From Proposition \ref{grado-constructible-cerrados-multiple:prop}, we conclude:
 $$\deg(R_k) \leq \deg(W) \left( 2\scrN'd_1 \right)^{\dim(W)}.$$
As the degree of constructible sets is sub-additive, we, then, conclude:
$$\deg_{\rm lci}(\pi(V)) \leq \deg(W)\scrN\left( 2\scrN'd_1 \right)^{\dim(W)}.$$
Applying Proposition \ref{grado-preservado-por-transformaciones-lineales:prop} to the previous inequality we obtain the inequality of the statement.
\end{proof}

\section{Other Exponential length Correct Test Sequences in the mathematical literature} \label{Kakeyaetal-CTS:sec}
\subsection{Correct Test Sequences and determinantal varieties} \label{variedades-determinanciales:ej}
Let $m,n\in \N$ be two positive integers with $m\geq n$. We consider the \emph{determinantal varieties} in the space $\scrM_{m\times n}(K)$ of matrices $m\times n$ with entries in an algebraically closed field $K$:
$$\Sigma_{r}:= \{ M\in \scrM_{m\times n}(K)\; :\;  corank(M)\geq r\}.$$
It is well-known that $\Sigma_r$ is an algebraic variety, given by homogeneous polynomial equations (and, hence, also a projective variety), and its co-dimension in $\Omega:=\scrM_{m\times n}(K)$ is given by the following identity:
$$\codim_\Omega(\Sigma_r)=\dim\left(\scrM_{m\times n}(K)\right)- \dim\left(\Sigma_r\right) = (m-r)(n-r).$$
Then,  for every correct test sequence ${\bf Q}:=(x_1,\ldots, x_L)\in \A^{nL}$ for $\Omega:=\scrM_{m\times n}(K)$ with respect to $\Sigma_r$ we have (see Proposition \ref{codimension-longitud-cuestores:prop}):
$$L:=\sharp({\bf Q}) \geq (m-r)(n-r).$$
Indeed, we proved in \ref{teorema-principal-densidad-questores:teor} that there are correct test sequences for $\Omega$ with respect to $\Sigma_r$ of length $O((m-r) (n-r))$. \\
Similarly,  for $r<s$, we would have:
$$\codim_{\Sigma_r}\left(\Sigma_s\right)=(r-s)(m+n- (r+s)).$$
Thus, for every correct test sequence ${\bf Q}\subseteq \A^{nL}$ for $\Sigma_r$  with respect to  $\Sigma_s$ we also have:
$$\sharp({\bf Q}) \geq (r-s)(m+n- (r+s)).$$

\subsection{Correct Test Sequences and Kakeya Sets}\label{Kakeya:subsec} The notion of Kakeya sets in the case of finite fields comes from \cite{Wolff}.  An exponential lower bound for the cardinal of Kakeya sets in the case of finite fields was established in \cite{Dvir}. See also \cite{Tao} for other developments. In \cite{Tao}, Tao claims: \emph{``There is no known proof of the finite field Kakeya conjecture that does not go through the polynomial method''}. Here, we explain this phenomenon by observing that \emph{Kakeya sets are correct test sequences for certain constructible sets made by polynomials}. In fact, most of the methods used in what is known as  \emph{``the Polynomial Method''}  are  simply the use of a correct test sequence adapted to some class (typically constructible sets) of polynomials. Nevertheless, we also prove that the converse is false: In Corollary \ref{CTS-not-Kakeya:corol} we prove  that for  small and positive $\varepsilon$ and degrees smaller than $q^{1-\varepsilon}-1$, most correct test sequences are not Kakeya sets. Additionally,  we do not know of any correct test sequence for polynomials of degree $q-1$ which are not Kakeya sets.

 Here, we just prove a small generalization  of  the main outcome of \cite{Dvir} by using the correct test sequence terminology. Notations are the same as in previous Sections. For a projective point $v\in \PP_{n-1}(K)$ we denote by $K\langle v \rangle\subseteq K^n$ the vector subspace of dimension $1$ of $K^n$ generated by any non-zero representant $\widetilde{v}\in K^{n}\setminus \{0\}$ of the projective point $v\in \PP_{n-1}(K)$.

\begin{definition}[{\bf $q-$Kakeya set with directions in a zero-dimensional projective variety}]
With these notations, let $V\subseteq \PP_{n-1}(K)$ be a zero-dimensional projective variety and $q\in \N$ a positive integer. A $q-$Kakeya set with directions in $V$ is a finite subset $E\subseteq \A^n(K)$ such that the following property holds:\\
For every direction $v\in V$, there is a point $x\in E$ such that the affine line $\rho(x,v):=x+ K\langle v \rangle$ satisfies
$$\sharp\left( E \cap \rho(x,v)\right) \geq q.$$
\end{definition}

When $\kappa=\F_q$  is a finite field with $q$ elements, we denote by $\F_q\langle v \rangle$ the $\F_q-$vector space generated by $v$. Usual Kakeya sets are then defined as follows:

\begin{definition}[{\bf Usual Kakeya sets over finite fields}] Let $\F_q$ be a finite field of cardinal  $q$. A finite subset  $E\subseteq\A^n(\F_q)$ is called a Kakeya set if the following property holds:\\
For every projective point  $v\in \PP_{n-1}(\F_q)$, there exists $x \in E$ such that the $\F_q-$rational points of the affine line
 $\rho (x,v)\subseteq \A^n(\F_q)$  determined by $x$ and $v$ is totally included in $E$. Namely, such that the following holds:
$$\rho(x,v):=x + \F_q\langle v \rangle :=\{ x + t v\; :\; t\in \F_q\}\subseteq E.$$
\end{definition}

Recall that the Hilbert function of a projective variety $V\subseteq \PP_{n-1}(K)$ is a function
$\chi_V:\N\longrightarrow \N$ given by the following identity for all $m\in\N$:
$$\chi_V(m):=\dim_K\left( H_m(X_1,\ldots, X_n)/I_m(V)\right)=\dim_K\left( H_m(X_1,\ldots, X_n)\right) - \dim_K\left(I_m(V)\right),$$
where:
\begin{itemize}
\item $\dim_K$ means the dimension as $K-$vector spaces,
\item $H_m(X_1,\ldots, X_n)$ is the vector space spanned by  all homogeneous polynomials of degree $m$ in the set of variables $\{ X_1,\ldots, X_n\}$  with coefficients in $K$ and,
\item $I_m(V)$ is the vector subspace of polynomials in $H_m(X_1,\ldots, X_n)$ that vanish on the projective points of $V$.
\end{itemize}
Hilbert function is known to be a polynomial function of degree equal to the Krull dimension of $V$ as projective variety. Namely, there is a unique polynomial $q \in \Z[T]$ of degree equal to $dim(V)$ and an integer $R$ such that:
$$\chi_{V} (m) = q(m), \; \forall m \geq R$$
The polynomial $q$ is called Hilbert's polynomial of $V$ and the minimum of those $R$ such that $\chi_{V}$ and $q$ agree is called the \emph{regularity of the Hilbert function of $V$}.

The next statement proves that $q-$Kakeya sets are correct test sequences and, hence, we have a lower bound for its cardinal based on Proposition \ref{codimension-longitud-cuestores:prop}.
\begin{proposition}\label{Dvir-generalizado:prop} With the previous notations and assumptions, given a zero-dimensional projective variety $V\subseteq \PP_{n-1}(K)$,  for every $q-$Kakeya set $E$ with directions in $V$ we have that:
$$\sharp\left(E\right)\geq \chi_V(d), \; \forall d\in \N, \; d\leq q-1.$$
Moreover, if $R$ is the regularity of the Hilbert function of $V$ and if $q-1\geq R$ we also have:
$$\sharp(E)\geq \sharp(V)=\deg(V).$$
On the other hand,  $q-$Kakeya sets with directions in $V\subseteq \PP_{n-1}(K)$ with cardinal $(q-1)\deg(V)+1$ do exist.
\end{proposition}
\begin{proof}
Let $\A^N(K)$ be the affine space over $K$ given as the $K-$vector space $P_d:=P_d(X_1,\ldots, X_n)$ of all polynomials of degree at most $d$ with coefficients in $K$, $\A^N(K):=P_d$, where $N:=N_d$.  For every polynomial $f\in P_d(X_1,\ldots, X_n)$, we consider its decomposition in homogeneous components:
$$f=f_m + \cdots + f_0,$$
where $f_i\in H_i(X_1,\ldots, X_n)$ and $f_m\not=0$.

Let us denote by $LHC(f)\in P_d$ (leading homogeneous component) the homogeneous component of higher degree of $f$, i.e. $LHC(f)=f_m$.
For every $m$, $0\leq m\leq d$, let us define the following constructible subsets of $P_d:=P_d(X_1,\ldots, X_n)$:
$$\Omega_m:=\{ f\in P_d\; : \; \deg(f)=m\},$$
$$\Sigma_m:=\{ f\in \Omega_m\; :\; LHC(f)\in I_m(V)\},$$
where $I_m(V)$ is the vector space spanned by all homogeneous polynomials of degree $m$ that vanish on $V$. We finally define:
$$\Omega:=\bigcup_{m=0}^d \Omega_m=P_d,$$
and
$$\Sigma:= \bigcup_{m=0}^d \Sigma_m\subseteq \Omega.$$
\begin{claim}
With these notations, the following properties hold:
\begin{enumerate}
\item If $E$ is a $q-$Kakeya set with directions in $V$, then $E$ is a correct test sequence for $\Omega$ with respect to $\Sigma$.
\item Let us define $d_0:=\max\{ m \in \N\; :\; m\leq d, \; I_m(V)\not=\{0\}\}$.
Then, either $d_0=d$ or $\Sigma_m=\{0\}$ for all $m\leq d$.
\item The dimension of the constructible sets $\Omega$ and $\Sigma$ are given by the following identities:
    $$\dim(\Omega):=\sum_{i=0}^d\dim_K(H_i(X_1,\ldots, X_n))={{d+n}\choose{n}},$$
    $$\dim(\Sigma):=\left\{\begin{matrix}
    \dim_K\left(I_{d}(V)\right) + \sum_{i=0}^{d-1}\dim_K(H_i(X_1,\ldots, X_n)), & {\hbox {\rm if $I_{d}(V)\not=\{0\}$}}\\
    0, & {\hbox {\rm otherwise}}\end{matrix} \right.$$
\end{enumerate}
\end{claim}
\begin{proof-claim} We prove each part of the claim separately:
\begin{enumerate}
\item Let $f\in \Omega$ be a non-zero polynomial of degree $m$ such that $f|_E= 0$. Then, for every $v\in V$ there is some $x\in E$ such that $f$ vanishes on the intersection $E\cap\rho(x,v)$. Taking the homogeneous components of $f$:
    $$f=f_m + f_{m-1} + \cdots + f_0,$$
    where $f_m=LHC(f)\not=0$, there are polynomials in $2n$ variables $h_i$ such for all $t\in K$ the following identity holds:
    $$f(x+tv) = f_m(v)t^m + \sum_{i=1}^{m-1}h_i(x,v)t^{i} + f_0(x).$$
    We consider the univariate polynomial:
    $$F(T):=f(x+ Tv)=f_m(v)T^m + \sum_{i=1}^{m-1}h_i(x,v)T^{i} + f_0(x)\in K[T].$$
As $E$ is a $q-$Kakeya set with directions in $V$, there are at least $q$ different points $t_1,\ldots, t_q\in K$ such that $x+t_iv\in E$, $1\leq i \leq q$. Therefore, we deduce
that $F(t_i)=0$, $1\leq i \leq q$, and, by elementary interpolation arguments, we obtain that $F$ is the zero polynomial in $K[T]$. Thus, we conclude that $LHC(f)(v)=f_m(v)=0$ and, hence, $LHC(f)\in I_m(V)$ and we have proved that:
$$\forall f\in \Omega, \; f|_E= 0 \Longrightarrow f\in \Sigma,$$
as wanted.
\item Note that if $d_0>0$, then for every non-zero homogeneous polynomial $f\in I_{d_0}(V)$, the polynomial $X_1^{d-d_0}f\in I_d(V)$. Thus, either $d_0=d$ or $\Sigma_m=\{0\}$ for all $m\leq d$.
\item Let us first observe that for every $m\leq d$ the dimension of the constructible sets $\Omega_m$ and $\Sigma_m$ satisfy the following identities:
    $$\dim(\Omega_m):=\sum_{i=0}^m \dim_K(H_i(X_1,\ldots, X_n)),$$
    whereas
    $$\dim(\Sigma_m):=\left\{\begin{matrix}
    \dim_K(I_m(V))+\sum_{i=0}^{m-1} \dim_K(H_i(X_1,\ldots, X_n)),& {\hbox {\rm if $I_m(V)\not=\{0\}$}}\\
    0, & {\hbox {\rm otherwise.}}\end{matrix}\right.$$
    Both identities are immediate from the definitions of $\Omega_m$ and $\Sigma_m$ combined with the definitions of Krull dimension of constructible sets.\\
    Additionally, we also conclude that for $m\leq d$ the following inequalities hold:
$$\dim(\Omega_m)\leq \dim(\Omega_{d}), \; \dim(\Sigma_m)\leq\dim(\Sigma_{d}).$$
The first inequality being obvious, for the second one observe that if $m<d$ and $I_m(V)\not=\{0\}$, then $I_{d}(V)\not=\{0\}$. In this case,
$$\dim_K(I_m(V))+\sum_{i=0}^{m-1} \dim_K(H_i(X_1,\ldots, X_n))\leq \sum_{i=0}^m \dim_K(H_i(X_1,\ldots, X_n)),$$
and, hence, we have:
$$\dim(\Sigma_m) \leq \dim_K(I_d(V)) + \sum_{i=0}^{d-1}\dim_K(H_i(X_1,\ldots, X_n)).$$
We finally conclude:
$$\dim\left(\Omega\right)=\max\{ \dim(\Omega_m)\; :\; 0\leq m \leq d\} =\dim\left(\Omega_d\right)={{d+n}\choose{n}}.$$
On the other hand, if $I_d(V)\not=\{0\}$, we have:
$$\dim\left(\Sigma\right)=\max\{ \dim(\Sigma_m)\; :\; 0\leq m \leq d\} \leq
\dim_K(I_d(V))+\sum_{i=0}^{d-1} \dim_K(H_i(X_1,\ldots, X_n)).$$
Obviously, if $I_d(V)=\{0\}$, then $I_m(V)=\{0\}$ for all $m\leq d$ and, hence, $\Sigma=\{0\}$.
\end{enumerate}
\end{proof-claim}

Thus, we just recall Proposition \ref{codimension-longitud-cuestores:prop} to conclude:
$$\sharp(E)\geq \dim(\Omega)-\dim(\Sigma)\geq \dim_K(H_d(X_1,\ldots, X_n))-\dim_K(I_d(V))=
 \chi_V(d).$$
The second Claim of the Proposition follows since $V$ is zero-dimensional. Thus, for every $d\geq R$ we would have
that $\chi_V(d)=\deg(V)=\sharp(V)$.

Finally, the last Claim is obvious and not optimal. Just fix a set $\Lambda$ of $q$ elements in $K$ including $0$ and any mapping $\varphi: V\longrightarrow K^n$. Then, the following is a $q-$Kakeya set with directions in $V$:
$$E:=\bigcup_{v\in V} \{ \varphi(v) + \lambda v\; :\; \lambda \in \Lambda\},$$
whose cardinal is at most $q\deg(V)=q\sharp(V)$. Taking $\varphi$ any constant mapping, we would have a cardinal bounded by $(q-1)\deg(V)+1$.
\end{proof}

We shall see now that the main outcome of \cite{Dvir} is merely a consequence of our Proposition
\ref{Dvir-generalizado:prop} above:

\begin{corollary}[\cite{Dvir}]\label{Dvir-CTS-curse:corol} Let $\kappa:=\F_q$ be a finite field of cardinal $q$, and $K:=\overline{\F_q}$ its algebraic closure. Let $P_d^K=P_d^K(X_1,\ldots,X_n)$ be the constructible set of all polynomials in the set of variables $\{ X_1,\ldots, X_n\}$ of degree at most $d$ with coefficients in  $K=\overline{\F_q}$.\\
Then, if $d\leq q-1$ and if  $E\subseteq \A^n(\F_q)$ is a Kakeya set, then $E$  is a correct test sequence for $P_d^K$ (with respect to  $\{0\}$). In particular, we conclude:
$$\sharp(E) \geq \dim\left(P_d^K\right)= {{d+n}\choose{n}}, \; \forall d\leq q-1.$$
\end{corollary}

\begin{proof} First of all, observe that every Kakeya set $E\subseteq \A^n(\F_q)$ is a $q-$Kakeya set with directions in $V=\PP_{n-1}(\F_q)$. Then, it is a correct test sequence for $\Omega:=P_d^K$ with respect to $\Sigma$, where  $\Sigma$ is the constructible subset of $P_d^K$ defined in the Proof of the previous Proposition. Additionally, by the previous Proposition we conclude that:
$$\sharp(E) \geq \chi_V(d)=\dim_K(H_d(X_1,\ldots, X_n))- \dim_K(I_d(V)), \; \forall d\leq q-1,$$
where $H_d(X_1,\ldots, X_n)$ is the vector space over $\overline{\F_q}$ spanned by all homogeneous polynomials of degree $d$ with coefficients in $\overline{\F_q}$ and $I_d(V)$ is the vector space spanned by all polynomials $f\in H_d(X_1,\ldots, X_n)$ that vanish on $V$. The Corollary follows by proving that $\Sigma=\{0\}$ for every $d\leq q-1$. This may be achieved if we prove the following statement:
$$\chi_{\PP_{n-1}(\F_q)}(d)={{d+n}\choose {n}}, \; \forall d \leq q-1.$$
Equivalently, it will be enough to prove that $I_m(\PP_{n-1}(\F_q))=\{0\}$ for all $m\leq d \leq q-1$. In order to prove this last claim, let $f\in H_m(X_1,\ldots, X_n)$ a homogeneous non-zero polynomial of degree $m$. As $f$ is non-zero, at least one of its affine traces must be non-zero.  Without loss of generality, we may assume that  $^af(X_2,\ldots,X_m):=f(1,X_2,\ldots, X_n)\in \overline{\F_q}[X_2,\ldots, X_n]$ is a non-zero polynomial of degree at most $m\leq q-1$.
As $f$ vanishes in the whole space $\PP_{n-1}(\F_q)$, then $^af$ vanishes into the affine points of $\A^{n-1}(\F_q)\subseteq \PP_{n-1}(\F_q)$, where:
$$\A^{n-1}(\F_q):=\{(1:x_2: \cdots : x_n)\in \PP_{n-1}(\F_q)\; : \; x_2,\ldots, x_n\in \F_q\}.$$
If $^af$ vanishes on $\F_q^{n-1}$ we would have:
$$q^{n-1}\leq \sharp\left( V_{\F_q}(^af)\right),$$
where $V_{\F_q}(^af)$ are the $\F_q-$rational points of the hyper-surface  $V_\A(^af)$.
This cannot be true because of  Corollary \ref{comptage-des-points-rationels:corol}: Since $^af$ is a non-zero polynomial of degree at most $m\leq d \leq q-1$, we would have
$$q^{n-1}\leq \sharp\left(V_{\F_q}(^af)\right) \leq \deg(^af) q^{n-2}\leq (q-1)q^{n-2},$$
which is a contradiction and yields the Corollary.
\end{proof}

Nevertheless, most correct test sequences are no Kakeya sets.

\begin{corollary}\label{CTS-not-Kakeya:corol} Let $\F_q$ be a finite field with $q$ elements and let $\overline{\F_q}$ be its algebraic closure. Let $d\in \N$ be a degree and $k\in \N$ a positive integer such that:
\begin{equation}\label{grado-q:eqn}
d< q^{1 - {1\over k}}-1.
\end{equation}
Then, for $s:= k{{d+n}\choose{n}}$, the probability that a list of points ${\bf Q}\in \A^n(\F_q)^s$ were a correct test sequence for $P_d(X_1, \ldots, X_n)$ with respect to $\{0\}$ is at least:
$$1- {{1}\over{ q^{\dim(\Omega)}}}.$$
In particular,  given $d, k, q$ and $n$ such that:
\begin{equation}\label{Dviretal-inequality:eqn}
k{{d+1}\choose{n}}< {{q^n}\over{ 2^n}},
\end{equation}
and  inequality (\ref{grado-q:eqn}) holds, there are correct test sequences for $P_d(X_1,\ldots, X_n)$ with respect to $\{0\}$ which are not Kakeya sets.
\end{corollary}
\begin{proof}
With the notations of Subsection \ref{Kakeya:subsec}, we consider the constructible subset $\Omega$ given by the following identity:
$$\Omega:=P_d(X_1,\ldots, X_n):=\{ f\in \overline{\F_q}[X_1,\ldots, X_n] \; : \; \deg(f) \leq d \} .$$
We reproduce the Proof of Theorem \ref{teorema-principal-densidad-questores:teor}. Given $s:=k\dim(\Omega)\in \N$ a positive integer, let us consider the incidence variety:
$$V(\Omega,s):=\{ (f, x_1, \ldots, x_s)\in \Omega \times \left(\A^n(\overline{\F}_q)\right)^s\; : \; f(x_i)=0, \; 1\leq i \leq s\}.$$
Let us consider the two canonical projections $\pi_1: V(\Omega,s) \longrightarrow \Omega$ and $\pi_2: V(\Omega,s) \longrightarrow \left(\A^n(\overline{\F}_q)\right)^s$ and the class $\scrD$ formed by all irreducible components $C$ of $V(\Omega,s)$ such that $\pi_1(C)\setminus \{0\}\not=\emptyset$. By the same arguments as in the Proof of Theorem
\ref{teorema-principal-densidad-questores:teor} we conclude the following properties:
\begin{itemize}
\item For every $C\in \scrD$, its dimension satisfies:
$$\dim(C) \leq (n-1)s + \dim(\Omega)= (n-1)s + {{d+n}\choose{n}}.$$
\item Defining $B(\Omega,s):= \cup_{C\in \scrD} C$, by B\'ezout's Inequality we have that:
$$\deg(B(\Omega,s))\leq \deg(\Omega) \left( d+1 \right)^{s}.$$
\end{itemize}
As $\Omega$ is linear in $P_d$, this yields:
$$\dim(B(\Omega,s)) \leq  (n-1)s + {{d+n}\choose{n}},$$
and
$$\deg(B(\Omega,s))\leq \left( d+1\right)^{s}.$$
Again, by the same arguments as in Theorem \ref{teorema-principal-densidad-questores:teor}, taking $C:=\F_q^n$, the probability that a list ${\bf Q}\in C^s$ is not a  correct test sequence for $\Omega$ with respect to $\Sigma=\{0\}$ is at most:
$${{\deg(B(\Omega,s)) q^{\dim(B(\Omega))}}\over{\sharp\left( \left( \F_q^n\right)^s\right)}}\leq {{(d+1)^{s}q^{(n-1)s +\dim(\Omega)}}\over{q^{ns}}}\leq {{(d+1)^{s}q^{\dim(\Omega)}}\over{q^s}}\leq {{1}\over{ q^{\dim(\Omega)\lambda}}},$$
where $\lambda:= (k-1) - k\log_q\left(d+1\right)$.
Thus, the probability that a random list of length $s$ in $\F_q^n$ is a correct test sequence for $\Omega$ with respect to $\{0\}$ is at least:
$$1- {{1}\over{ q^{\dim(\Omega)\lambda}}},$$
where $\lambda:= (k-1) - k\log_q\left(d+1\right)$. In particular,
if $\log_q(d+1)< {{k-1}\over{k}}= (1-1/k)$, there are correct test sequences of length $k\dim(\Omega)$.\\

On the other hand, according to \cite{Dviretal}, the following lower bound holds for the cardinal of any Kakeya set in $\F_q^n$:
$$\sharp(E)\geq {{\sharp(\F_q^n)}\over{2^n}}.$$

Then, if $d,k,q$ and $n$ satisfy Inequalities (\ref{grado-q:eqn}) and (\ref{Dviretal-inequality:eqn}), there are correct test sequences for $P_d$ with respect to $\{0\}$ of length $k {{d+n}\choose{n}}$ which are not Kakeya sets.
\end{proof}

Whereas Inequality (\ref{Dviretal-inequality:eqn}) is natural for asymptotic values of $n$, Inequality (\ref{grado-q:eqn}) implies that our method of proof does not allow to exhibit correct test sequences which were not Kakeya sets for $d=q-1$.

\subsection{Correct Test Sequences, Alon's Combinatorial Nullstellensatz and duality}\label{CTS-Alon-combinatorio:subsec}
The outcome of \cite{Alon}, revised in \cite{Tao},  is another method that exhibits correct test sequences of exponential length, although it is commonly known as Combinatorial Nullstellensatz. Alon's statement is the following one:

\begin{theorem}[{\bf Combinatorial Nullstellensatz, according to \cite{Alon}}] Let $\kappa$ be a field and $K$ its algebraic closure. Let $d_1,\ldots, d_n$  be a list of positive integers and let $P\in K[X_1,\ldots, X_n]$ be a polynomial of degree at most  $d_1+\cdots+d_n$ such that the term of the monomial  $X_1^{d_1}\cdots X_n^{d_n}$ in $P$ is non-zero. Then,  for every family of subsets $E_1,\ldots, E_n\subseteq \kappa$, such that $\sharp(E_i)>d_i$,  there is a point $\zeta$ in the Cartesian product $E=E_1\times \cdots \times E_n$ such that $P(\zeta)\not=0$.
\end{theorem}

Here, we revise this statement under a different format. Notations are the same as in previous sections.
For each polynomial  $f\in P_d(X_1,\ldots, X_n)$ and for every exponent $\underline{\mu}:=(\mu_1,\ldots, \mu_n)\in \N^n$, such that $\vert\underline{\mu}\vert = \mu_1+ \cdots + \mu_n\leq d$, we denote by $f_{\underline{\mu}}$ the coefficient of the monomial $X_1^{\mu_1}\cdots X_n^{\mu_n}$ in the monomial expansion of $f$.
We prove the following statement:

\begin{theorem}[{\bf Combinatorial Nullstellensatz}]\label{extension-Alon-Nullstellensatz-Comninatorio:teor}
With the same notations as above, let $(d):=(d_1,\ldots, d_n)$ be a degree list, with $d_i\geq 1$. Let  $ D:=d_1+\cdots+d_n -n$ be a non-negative integer. Let $\Delta_{(d)}\subseteq \N^n$ be the subset given by the following equality:
\begin{equation}\label{exponentes-monomiales-listagrados:eqn}
\Delta_{(d)}:=\{ \underline{\mu}\in \N^n\; : \; \underline{\mu}= (\mu_1, \ldots, \mu_n), \; 0\leq \mu_i\leq d_i-1 \}.
\end{equation}
Let $\Omega_{(d)}\subseteq P_D^K(X_1,\ldots, X_n)$ be the following constructible subset:
$$\Omega_{(d)}:=\{ f \in P_D(X_1,\ldots, X_n) \; :\; \exists \underline{\mu}\in \Delta_{(d)}, \; f_{\underline{\mu}}\not=0, \; \vert \underline{\mu}\vert = \deg(f)\}\cup\{0\}.$$
Then, for every finite sequence of subsets $E_1,\ldots, E_n\subseteq K$, such that $\sharp(E_i)=d_i$, the Cartesian product $E:=E_1\times \cdots \times E_n$ is a correct test
sequence for $\Omega_{(d)}$ with respect to $\Sigma=\{0\}$. Additionally, if $\sharp(\kappa) \geq \max\{ d_i+1\; :\; 1\leq i\leq n\}$, we may choose $E_i\subseteq \kappa$ for every $i$, $1\leq i \leq n$.
\end{theorem}

As $(d_1-1, \ldots, d_n-1) \in \Delta_{(d)}$, this last Theorem includes Alon's Combinatorial Nullstellensatz.  We will exhibit a new proof of it, vaguely inspired by the aspects discussed in \cite{Tao}, but using as main mathematical ingredient trace and duality on algebras of zero-dimensional varieties.

We begin by a classical construction that we recall in order to fix notations. With the same notations and assumptions as above, let $\kappa$ be a field and $K$ its algebraic closure.
For every point $z\in \A^n(K)$ we denote by ${\frak m}_z\subseteq K[X_1,\ldots, X_n]$ the maximal ideal of all polynomials in $K[X_1,\ldots, X_n]$ vanishing at $z$, i.e.
$${\frak m}_z:=\{ f\in K[X_1,\ldots, X_n]\; :\; f(z)=0\}.$$
Let $E\subseteq \A^n(K)$ be a finite set and let  $D=\deg(E)=\sharp(E)$ be its cardinal. Assume $E=\{z_1,\ldots,z_D\}\subseteq\A^n(K)$.
Let us denote by $I(E)\subseteq K[X_1,\ldots,X_n]$ the ideal of all polynomials in $K[X_1,\ldots, X_n]$ vanishing at all points in $E$, i.e.
$$I(E)=\{f\in K[X_1,\ldots, X_n]\; :\; f(z)=0, \; \forall z\in E\}=\bigcap_{i=1}^D \mathfrak{m}_{z_i},$$
Let $K[E]:=K[X_1,\ldots, X_n]/I(E)$ be the residual ring, which is a zero-dimensional (Artinian) $K-$algebra and a vector space of dimension equal to $\sharp(E)$. \\
If $E\subseteq \A^n(\kappa)$ let us also consider the ring $\kappa[E]$ given by:
$$\kappa[X_1,\ldots, X_n]/I(E)^c,$$
where $I(E)^c:= I(E)\cap \kappa[X_1,\ldots, X_n]$ is the contraction of $I(E)$. Then, $\kappa[E]$ is also a zero-dimensional (Artinian) ring  and we have:
$$K\otimes_{\kappa}\kappa[E]= K[E].$$
In particular, if $E\subseteq \A^n (\kappa)$, the following dimensions (as vector spaces) agree:
$$\dim_\kappa\left(\kappa[E]\right)= \dim_K\left(K[E]\right).$$
 We now recall a classical identification of residual classes in $K[E]$ (or $\kappa[E]$)  with homotheties over the same vector spaces. With the same notations, for every $h\in K[X_1,\ldots, X_n]$, let us denote by $\eta_h:K[E]\longrightarrow K[E]$ the following endomorphism of $K-$vector spaces:
$$\begin{matrix}
\eta_h: & K[E] & \longrightarrow & K[E]\\
& g+ I(E) & \longmapsto & (hg) +I(E).
\end{matrix}$$
The trace of these endomorphisms allow us to introduce the next bilinear form on $K[E]$:
$$\begin{matrix}
\langle \cdot, \cdot\rangle_E: & K[E]\times K[E] & \longrightarrow & K\\
& (g_1+ I(E), g_2 + I(E)) & \longmapsto & Tr(\eta_{g_1g_2}).
\end{matrix}$$

The following are well-known consequences of the Chinese Remainder Theorem applied to $K[E]$:\\
For every  $h\in K[X_1,\ldots,X_n]$, $\eta_h$ is a diagonalizable endomorphism and its Jordan canonical form is the following diagonal matrix:
    $$\Diag(h(z_1),\ldots,h(z_D)).$$
In particular, the trace of $\eta_h$ is given by the following identity:
$$\Tr(\eta_h)=\sum_{i=1}^Dh(z_i)=\sum_{z\in E}h(z).$$

Moreover, we also have the following statement:

\begin{lemma}\label{bilineal:lema}
With the same notations as above, for every basis  $\scrB=\{v_1,\ldots,v_D\}$, $v_i=f_i+I(E)$, of $K[E]$ as $K-$vector space there is a basis (called dual with respect to $\langle \cdot, \cdot \rangle_E$)  $\scrB^*=\{w_1,\ldots,w_D\}$, $w_j=g_j+I(E)$, such that for all $i,j\in \{1,\ldots, D\}$ we have:
$$\left\langle v_i,w_j\right\rangle_E=\sum_{z\in E}f_i(z)g_j(z)=\delta_{i,j},$$
where $\delta_{i,j}$ is the Kronecker delta.
\end{lemma}
\begin{proof} Let $\scrB:=\{v_1,\ldots, v_d\}$ be an ordered basis of $K[E]$ as vector space.
By the Chinese Remainder Theorem, we may identify $K[E]$ and $K^D$ by the following isomorphism:
$$\begin{matrix}
\widetilde{\varphi}: & K[E] & \longrightarrow & K^D\\
& g+ I(E) & \longrightarrow & (g(z_1), \ldots, g (z_D)).
\end{matrix}$$
Then, $\scrB$ is a basis of  $K[E]$ if and only if the following matrix (vanderMonde's) is a regular matrix:
$$vdM(\scrB):=\left(\begin{matrix}v_1(z_1) & \cdots & v_1(z_D)\\
\vdots & \ddots & \vdots \\
v_D(z_1) & \cdots & v_D(Z_D)\end{matrix}\right).$$
For every $i$, $1\leq i \leq D$, let $e_k$ be the $k-$th vector of the ``canonical'' basis of $K^D$ and let $\omega_i:=(\omega_{i,1}, \ldots, \omega_{i,D})\in K^D$ be the unique solution of the following linear system of equations:
$$vdM(\scrB)\left(\begin{matrix} \omega_{i,1} \\ \vdots \\ \omega_{i,D}\end{matrix}\right)= e_i^t,$$
where $e_i^t$ is the transposed matrix of the vector $e_i$. Then,
by the isomorphism $\widetilde{\varphi}$ there exist $g_i\in K[X_1,\ldots, X_n]$ and $w_i= g_i+I(E)\in K[E]$ such that:
$$\widetilde{\varphi}(g_i+I(E))=(g_i(z_1), \ldots, g_i(z_D))=(\omega_{i,1}, \cdots, \omega_{i,D}).$$
The family $\scrB^*:=\{w_1, \ldots, w_D\}$ is the ``dual'' basis of $\scrB$.\end{proof}

\begin{lemma}\label{ejemplos-NC:lema}
Let $h_1,\ldots,h_n\in \kappa[T]$ be univariate polynomials with respective degrees  $d_1,\ldots,d_n$. Assume that $h_1,\ldots,h_n$ split completely in $\kappa$ and that they are square-free (i.e. they do not have multiple roots). We define the ideal
$\mathfrak{a}=(h_1(X_1),\ldots,h_n(X_n))\subseteq K[X_1,\ldots,X_n]$ and, for each $i$, $1\leq i \leq n$, let $E_i\subseteq\kappa$ be the set of zeros of $h_i$ in  $\kappa$. Let us consider also the zero-dimensional variety given by the Cartesian product $E=E_1\times \cdots\times E_n\subseteq\A^n(K)$. We have:
\begin{enumerate}
\item The ideal  ${\frak a}=I(E)$ is a radical ideal (i.e. the residual ring $K[X_1,\ldots, X_n]/{\frak a}$ has no non-zero nilpotente element) and, hence, ${\frak a}=I(E)=\{ f\in K[X_1,\ldots, X_n]\; :\; f|_E= 0\}$.
\item The family $\{h_1(X_1),\ldots,h_n(X_n)\}$ is a Gr\"obner basis of $\mathfrak{a}=I(E)$ with respect to the monomial order ``degree+lexicographic'' with the variable order $X_1<X_2<\ldots<X_n$.
\item The following is a monomial basis of $K[E]$:
$$\scrB=\{X_1^{\mu_1}\cdots X_n^{\mu_n}+{\frak a}\; :\; 0\leq\mu_i\leq d_i-1,1\leq i\leq n\}.$$
\item For every $i$, $1\leq i\leq n$, let us denote by $\scrB_i^*$ the dual basis in $K[E_i]$ with respect to the trace $\langle \cdot, \cdot \rangle_{E_i}$ of $\scrB_i:=\{T^{\mu_i}+(h_i):0\leq\mu_i\leq d_i-1\}$. Let us denote this dual basis by:
$$\scrB_i^*:=\{g_{k}^{(i)}(T)+(h_i):0\leq k \leq d_i-1\}\subseteq K[E_i]:=K[T]/(h_i).$$
Then, the following is a dual basis of  $\scrB$ in $K[E]$ with respect to the bilinear form $\langle\cdot,\cdot\rangle_E$:
$$\scrB^*=\left\lbrace\left(\prod_{i=1}^n g_{\mu_i}^{(i)}(X_i)\right)+{\frak a}\; :\; (\mu_1,\ldots,\mu_n)\in\N^n,0\leq\mu_i\leq d_i-1, 1\leq i\leq n\right\rbrace.$$
\end{enumerate}
\end{lemma}
\begin{proof} Properties $i)$ to $iii)$ are well-known and need no additional proof.
As for Claim  $iv)$, let us denote by $(d)=(d_1,\ldots, d_n)$ the degree list and by $\Delta_{(d)}$ the class of monomial exponents introduced in Identity (\ref{exponentes-monomiales-listagrados:eqn}) of Theorem \ref{extension-Alon-Nullstellensatz-Comninatorio:teor}.

Let us observe that for each $i$, $1\leq i \leq n$, we have:
$$\sum_{z_i\in E_i} g_k^{(i)}(z_i)z_i^r:=\delta_{k,r},$$
for $0\leq k,r\leq d_i-1$. \\
Given every pair of monomial exponents $\underline{\theta}:=(\theta_1,\ldots, \theta_n)\in \Delta_{(d)}$ and
$\underline{\mu}:=(\mu_1,\ldots, \mu_n)\in \Delta_{(d)}$, let us define the following quantity:
$$D_{\underline{\mu}, \underline{\theta}}:=\langle \prod_{i=1}^n g_{\mu_i}^{(i)}(X_i) + {\frak a}, X_1^{\theta_1}\cdots X_n^{\theta_n} + {\frak a}  \rangle_E=\sum_{(z_1,\ldots, z_n)\in E}\left(\left(\prod_{i=1}^n g_{\mu_i}^{(i)}(z_i)\right)z_1^{\theta_1}\cdots z_n^{\theta_n}\right).$$
We have:
$$D_{\underline{\mu}, \underline{\theta}}= \left(\sum_{z_1\in E_1} g_{\mu_1}^{(1)}(z_1) z_1^{\theta_1}\right)
\left(\sum_{(z_2,\ldots, z_n)\in E_2\times \cdots \times E_n}\left(\left(\prod_{i=2}^n g_{\mu_i}^{(i)}(z_i)\right)z_2^{\theta_2}\cdots z_n^{\theta_n}\right)\right).$$
As $\scrB_1^*$ is a dual basis of $\scrB_1$, we conclude that $\left(\sum_{z_1\in E_1} g_{\mu_1}^{(1)}(z_1) z_1^{\theta_1}\right)=\delta_{\mu_1, \theta_1}$ and we also have:
$$D_{\underline{\mu}, \underline{\theta}}= \delta_{\mu_1,\theta_1}
\left(\sum_{(z_2,\ldots, z_n)\in E_2\times \cdots \times E_n}\left(\left(\prod_{i=2}^n g_{\mu_i}^{(i)}(z_i)\right)z_2^{\theta_2}\cdots z_n^{\theta_n}\right)\right).$$
By induction on $n$, we finally conclude:
$$D_{\underline{\mu}, \underline{\theta}}= \delta_{\mu_1,\theta_1}\delta_{\mu_2, \theta_2} \cdots \delta_{\mu_n,\theta_n}= \delta_{\underline{\mu}, \underline{\theta}},$$
and Claim $iv)$ is proven.
\end{proof}

The following statement is a variation of a Lemma that may be found in \cite{Tao}. We just give  a proof based on our duality arguments.
\begin{lemma}\label{Tao:lema}
With the same notations as above, let $\underline{\theta}=(\theta_1,\ldots,\theta_n)\in\Delta_{(d)}$ be a monomial exponent in $\Delta_{(d)}$. Let $\underline{\mu}=(\mu_1,\ldots,\mu_n)\in\N^n$ be another monomial exponent such that  $|\underline{\mu}|=\mu_1+\cdots+\mu_n= \vert \underline{\theta}\vert$. If there is some index  $i$ such that $\mu_i\geq d_i$,  then the following holds:
$$D_{\underline{\theta},\underline{\mu}}=\left\langle\left(\prod_{i=1}^n g_{\theta_i}^{(i)}(X_i)\right)+{\frak a},X_1^{\mu_1}\cdots X_n^{\mu_n}+{\frak a}\right\rangle_E=0.$$
\end{lemma}
\begin{proof} The proof is similar to that of the preceding Lemma. Without loss of generality, let us assume that  $\mu_n\geq d_n$. Since we have:
$$\mu_1+ \cdots + \mu_n = \theta_1 + \cdots + \theta_n, \; 0\leq \theta_i \leq d_i-1,$$
then there must exists  $j\not= n$ such that $\mu_j\not= \theta_j$ and $0\leq \mu_j \leq d_i-1$. In order to prove this, observe that if for all $j\not=n$ either we have $\mu_j= \theta_j$ or $\mu_j\geq d_i$, then we would have:
$$\vert\underline{\mu}\vert=\mu_1+ \cdots + \mu_n \geq \theta_1+ \cdots + \theta_{n-1} + \mu_n \geq \theta_1+ \cdots + \theta_{n-1} + \theta_n +1>\vert \underline{\theta}\vert +1,$$
which contradicts the hypothesis  $\vert\underline{\mu}\vert = \vert\underline{\theta}\vert$.

Without loss of generality again, we assume that $\mu_1\not= \theta_1$ and $0 \leq \mu_1\leq d_1-1$. Then, we have:
$$D_{\underline{\mu}, \underline{\theta}}:= \left\langle \left( \prod_{i=1}^n g_{\mu_i}^{(i)}(X_i) \right) + {\frak a}, X_1^{\theta_1}\cdots X_n^{\theta_n} + {\frak a}  \right\rangle_E=\sum_{(z_1,\ldots, z_n)\in E}\left(\left(\prod_{i=1}^n g_{\mu_i}^{(i)}(z_i)\right)z_1^{\theta_1}\cdots z_n^{\theta_n}\right).$$
Thus, we get:
$$D_{\underline{\mu}, \underline{\theta}}= \left(\sum_{z_1\in E_1} g_{\mu_1}^{(1)}(z_1) z_1^{\theta_1}\right)
\left(\sum_{(z_2,\ldots, z_n)\in E_2\times \cdots \times E_n}\left(\left(\prod_{i=2}^n g_{\mu_i}^{(i)}(z_i)\right)z_2^{\theta_2}\cdots z_n^{\theta_n}\right)\right).$$
And, since $\mu_1\not=\theta_1$, we conclude:
$$D_{\underline{\mu}, \underline{\theta}}= 0\cdot
\left(\sum_{(z_2,\ldots, z_n)\in E_2\times \cdots \times E_n}\left(\left(\prod_{i=2}^n g_{\mu_i}^{(i)}(z_i)\right)z_2^{\theta_2}\cdots z_n^{\theta_n}\right)\right)=0,$$
which finishes the proof of the Lemma.\end{proof}

\begin{corollary}\label{Alon-final:corol} With the same notations as above, let $\underline{\theta}\in \Delta_{(d)}$ and
$\underline{\mu}\in \N^n$ be two monomial exponents. If $\vert\underline{\mu}\vert \leq \vert \underline{\theta}\vert$, we have:
$$D_{\underline{\theta},\underline{\mu}}=\left\langle\left(\prod_{i=1}^n g_{\theta_i}^{(i)}(X_i)\right)+{\frak a},X_1^{\mu_1}\cdots X_n^{\mu_n}+{\frak a}\right\rangle_E= \delta_{\underline{\mu}, \underline{\theta}},$$
where $\delta_{\underline{\mu}, \underline{\theta}}$ is Kronecker delta.
\end{corollary}
\begin{proof} If $\underline{\mu}\in \Delta_{(d)}$, Lemma \ref{ejemplos-NC:lema} implies the given statement. Otherwise, if $\vert\underline{\mu}\vert \leq \vert \underline{\theta}\vert$ and $\underline{\mu}\not\in \Delta_{(d)}$, the property holds because of Lemma \ref{Tao:lema}.
\end{proof}

\bigskip

\begin{proof}{\sc(of Theorem \ref{extension-Alon-Nullstellensatz-Comninatorio:teor}).}
Let $d_1,\ldots, d_n$ be some degrees,  let $(d)=(d_1,\ldots, d_n)$ be the degree list they define and $\Delta_{(d)}$ the class of monomial exponents  as in the statement of Theorem \ref{extension-Alon-Nullstellensatz-Comninatorio:teor}. Let $D:=d_1+\cdots + d_n-n$ be the quantity also defined in the statement of this Theorem.
Let $\Omega_{(d)}\subseteq P_D(X_1,\ldots, X_n)$ be the following constructible subset made of polynomials of degree at most $D$:
$$\Omega_{(d)}:=\{ f \in P_D^K(X_1,\ldots, X_n) \; :\; \exists \underline{\mu}\in \Delta_{(d)}, \; f_{\underline{\mu}}\not=0, \; \vert \underline{\mu}\vert = \deg(f)\}\cup\{0\}.$$
Let $E_1,\ldots, E_n\subseteq K$ be some finite subsets such that $\sharp(E_i)=d_i$. Let $E:=E_1\times \cdots \times E_n$ be the Cartesian product and let $K[E]$ be the residual ring of $K[X_1,\ldots, X_n]$ modulo $I(E)$. Let $\langle \cdot, \cdot \rangle_E$ be the bilinear form associated to the trace on $K[E]$.

For every $i$, $1\leq i \leq n$, let $h_i\in K[T]$ be the following univariate polynomial:
$$h_i(T):=\prod_{\zeta\in E_i} (T-\zeta).$$
According to Lemma \ref{ejemplos-NC:lema}, let ${\frak a}:=(h_1(X_1), \ldots, h_n(X_n))$ be the ideal they generate. Therefore, we have:
$$K[E]=K[X_1,\ldots, X_n]/{\frak a}.$$
Let $\scrB$ be the monomial basis of $K[E]$ given by the following equality:
$$\scrB=\{X_1^{\mu_1}\cdots X_n^{\mu_n}+{\frak a}\; :\; 0\leq\mu_i\leq d_i-1,1\leq i\leq n\},$$
and let $\scrB^*$ be the dual basis of $\scrB$ exhibited in the Lemma \ref{ejemplos-NC:lema} above:
$$\scrB^*=\left\lbrace\left(\prod_{i=1}^n g_{\mu_i}^{(i)}(X_i)\right)+{\frak a}\; :\; (\mu_1,\ldots,\mu_n)\in\N^n,0\leq\mu_i\leq d_i-1, 1\leq i\leq n\right\rbrace.$$

Let $f\in \Omega_{(d)}$ be a polynomial and suppose  $t:=\deg(f)$. The monomial expansion of $f$ has the following form:
$$f:=\sum_{\vert\underline{\mu}\vert \leq t}a_{\underline{\mu}}X_1^{\mu_1}\cdots X_n^{\mu_n}.$$
Since $f\in \Omega_{(d)}$, then there exists $\underline{\theta}\in \Delta_{(d)}$ such that $\vert \underline{\theta}\vert = \deg(f)=t$ and the coefficient $f_{\underline{\theta}}= a_{\underline{\theta}}\not=0$ is non-zero.

Let $G_{\underline{\theta}}:=\left(\prod_{i=1}^n g_{\theta_i}^{(i)}(X_i)\right)+{\frak a}\in K[E]$. By Corollary
\ref{Alon-final:corol} above, we deduce that for all $\underline{\mu}\in \N^n$ such that $\underline{\mu}\not = \underline{\theta}$ the following holds:
$$D_{\underline{\theta}, \underline{\mu}}=\left\langle G_{\underline{\theta}},X_1^{\mu_1}\cdots X_n^{\mu_n}+{\frak a}\right\rangle_E=0.$$
Since
$$D_{\underline{\theta}, \underline{\theta}}= \left\langle G_{\underline{\theta}},X_1^{\theta_1}\cdots X_n^{\theta_n}+{\frak a}\right\rangle_E=1,$$
we thus conclude:
$$\left\langle G_{\underline{\theta}},f+{\frak a}\right\rangle_E= a_{\underline{\theta}}\not=0.$$
On the other hand, we have:
$$a_{\theta}=\left\langle G_{\underline{\theta}},f+{\frak a}\right\rangle_E=
\sum_{(z_1,\ldots, z_n)\in E}\left(\prod_{i=1}^n g_{\theta_i}^{(i)}(z_i)\right) f(z_1,\ldots, z_n).$$
Then, we conclude that it is not possible that $f$ vanishes at all points of $E$ since, otherwise, $a_{\underline{\theta}}=0$, which cannot be possible.
\end{proof}

\end{document}